\documentclass[70pt,a4paper]{article}

\usepackage[utf8]{inputenc}
\usepackage{amsmath}
\usepackage{amsfonts}
\usepackage{amssymb}
\usepackage{enumitem}
\usepackage{color}
\usepackage{graphicx}
\usepackage{alltt}
\usepackage[all,cmtip]{xy}
\usepackage{amsthm}
\usepackage{bbm}
\usepackage{setspace}
\usepackage{geometry}
\usepackage{tikz}
\usetikzlibrary{arrows.meta}
\tikzset{%
	>={Latex[width=2mm,length=2mm]},
	base/.style = {rectangle, rounded corners, draw=black, minimum width=7cm, minimum height=1.5cm, text centered},
	input/.style = {base, fill=green!30},
	algostep/.style = {base, fill=blue!30},
	theory/.style = {base, fill=yellow!30},
}
\usepackage{subcaption}
\usetikzlibrary{patterns,arrows,chains,matrix,positioning,scopes}
\usepackage{xfrac}
\usepackage{nicefrac}
\usepackage[backend=bibtex,style=numeric,doi=false,isbn=false,url=true,sortlocale=en_EN,giveninits=true,sorting=nyt]{biblatex}
\usepackage{hyperref}
\hypersetup{linktocpage}
\usepackage{cleveref}
\usepackage{thmtools} 
\usepackage[export]{adjustbox}
\usepackage{makecell} 
\setcellgapes{5pt}

\newtheorem{theorem}{Theorem}[section]
\newtheorem{proposition}[theorem]{Proposition}
\newtheorem{lemma}[theorem]{Lemma}
\newtheorem*{claim}{Claim}

\theoremstyle{definition}
\newtheorem{definition}[theorem]{Definition}
\newtheorem*{definition*}{Definition}

\newtheorem*{remark}{Remark}
\newtheorem{example}[theorem]{Example}

\DeclareMathOperator{\Ima}{Im}

\DeclareMathOperator{\spn}{Span}

\DeclareMathOperator{\diam}{diam}

\DeclareMathOperator{\Id}{Id}

\DeclareMathOperator{\supp}{supp}
\DeclareMathOperator{\Int}{Int}

\newcommand{\RR}{\mathbb{R}}
\newcommand{\RRn}{\mathbb{R}^{n}}
\newcommand{\RRnn}{\mathbb{R}^{2n}}
\newcommand{\ZZ}{\mathbb{Z}}
\newcommand{\NN}{\mathbb{N}}
\renewcommand{\SS}{\mathbb{S}}

\newcommand{\DD}{\mathbb{D}}
\newcommand{\BB}{\mathbb{B}}
\newcommand{\FF}{\mathbb{F}}
\newcommand*\diff{\mathop{}\!\mathrm{d}}

\newcommand{\dint}{d_{\text{int}}}
\newcommand{\dbot}{d_{\text{bot}}}

\DeclareMathOperator{\Ham}{Ham}

\DeclareMathOperator{\Crit}{Crit}
\DeclareMathOperator{\Pivot}{Pivot}

\title{Approximation of Generating Function Barcode for Hamiltonian Diffeomorphisms}
\author{Pazit Haim-Kislev and Ofir Karin}
\date{}
\begin{document}
	\maketitle
	\begin{abstract}
		Persistence modules and barcodes are used in symplectic topology to define various invariants of Hamiltonian diffeomorphisms, however numerical methods for computing these barcodes are not yet well developed. In this paper we define one such invariant called the \textit{generating function barcode} of compactly supported Hamiltonian diffeomorphisms of $ \RRnn $ by applying Morse theory to generating functions quadratic at infinity associated to such Hamiltonian diffeomorphisms and provide an algorithm (i.e a finite sequence of explicit calculation steps) that approximates it.
	\end{abstract}
	
		\tableofcontents
	\pagebreak

	\section{Introduction and main results}
	
	Started by Polterovich and Shelukhin \cite{PolShel15}, it is by now standard to use the language of {\it persistence modules} \cite{ZomCar05} to describe the structure of action filtered Floer homologies associated to a Hamiltonian diffeomorphism (e.g. \cite{Team,Fraser,KisShe,PolShel15,PolSheSto,She19,stevenson,UsherZhang,Zhang} ). One useful ingredient in the theory of persistence modules is that they are determined by a {\it barcode} which is a multiset $\{I_k, m_k\}$ of intervals $I_k \subset \RR$ bounded from below, and multiplicities $m_k \in \NN$. The isometry theorem \cite{BauLes} of persistence modules states that the interleaving distance between two persistence modules, which is a metric calculated from algebraic properties, equals the bottleneck distance between their barcodes, which is a combinatorial metric which measures the minimal distance one is required to move the endpoints of one barcode in order to obtain the other. As of now it is not known how to explicitly calculate the barcode of a filtered Hamiltonian Floer homology of a Hamiltonian diffeomorphism except in very few specific examples. 
	
	In this paper we consider the generating function homology \cite{Tr94} of compactly supported Hamiltonian diffeomorphisms $\varphi \in \Ham_c(\RRnn,\omega_0)$, which is a notion closely related to Floer homology, and we construct an appropriate barcode. We then describe a finite time algorithm to calculate such barcodes up to an arbitrarily small error with respect to the bottleneck distance.
	
	An important feature of generating functions for $\Ham(\RRnn,\omega_0)$ is the existence of composition formulas, attributed to Chekanov and appearing in \cite{Chap90}, that allow one to explicitly obtain generating functions for $\varphi \in \Ham_c(\RRnn, \omega_0)$ by knowing such functions for $\{\varphi_j\}_{j=1}^N$ that satisfy $\varphi = \varphi_n \circ \ldots \circ \varphi_1$. Taking $N$ large enough one may assume that for all $j$, $\varphi_j$ is $C^1$-small, which means that its generating function takes a simpler form with no auxiliary variables (cf. Section \ref{sec:genfunc} below). Our main result uses this fact to construct an algorithm for calculating the barcode associated to the filtered generating function homology of a general compactly supported Hamiltonian diffeomorphism.

	\begin{theorem} \label{thm:introthm}
		Let $\varphi_1,\ldots,\varphi_N$ be compactly supported $C^1$ small Hamiltonian diffeomorphisms, and let $\varepsilon >0$. Then there exists a finite time algorithm that gets as input the values of the generating functions of $\varphi_1,\ldots,\varphi_N$ applied to an appropriate finite sample, and returns a barcode whose bottleneck distance from the barcode of $\varphi_N \circ \ldots \circ \varphi_1$ is bounded above by $\varepsilon$. Moreover, if one fixes $N$ and the bounds on the $C^1$-norms and the radius of the support of $\varphi_1,\ldots,\varphi_N$, then the time complexity bound of the algorithm (i.e. bound on the number of operations needed to compute all steps) is polynomial in $\frac{1}{\varepsilon}$.
	\end{theorem}
	
	The diagram in \Cref{fig:intro} lays out the theoretical construction of the generating funtion barcodes (left hand side) along the approximation algorithm (right hand side). It is important to note that while the suggested algorithm is only capable of calculating relatively simple examples as of now, it is meant to be seen more as a "proof of concept", in the sense that it demonstrates the computability of such barcodes in symplectic topology and may be vastly improved in future research.

	We conclude the paper by offering an implementation of the suggested algorithm in a special case, where the generating functions of Hamiltonian diffeormophisms close to the identity may be themselves approximated (see \Cref{sec:radialprofile}). Namely, we consider the class of Hamiltonian diffeomorphisms of $ \RR^2 $ given by compositions of time-$ 1 $ maps of Hamiltonian flows generated by autonomous radial Hamiltonian functions. We note that even this simple construction of composing autonomous radial Hamiltonians can achieve very complicated dynamics, in which finding the barcode is considered very hard. We offer a computation example (see \Cref{sec:exmpale}) where the barcode of such composition is indeed well approximated.

	\subsection*{Structure of the paper}
	In \Cref{sec:prelim} we go over the relevant background and constructions needed to define GF-barcodes.
	In \Cref{sec:approxalgo} we prove \Cref{thm:introthm} by showing how GF-Barcodes can be numerically approximated.	
	In \Cref{sec:radialprofile} we discuss the implementation of the algorithm for autonomous radial Hamiltonian functions in $\RR^2$.

	\subsection*{Acknowledgements}
	This paper is a part of the second author’s thesis, carried out under the supervision of Prof. Leonid Polterovich and Prof. Lev Buhovsky at Tel-Aviv university. We thank them both for many meaningful discussions and for their original ideas motivating this project. The authors also wish to thank Prof. Yoel Shkolnisky from the applied math department of Tel-Aviv university for his kind assistance with computational issues. The first author is partially supported by the European Research Council grant No. 637386. The second author is supported by ISF grant numbers 1102/20 and 2026/17.

\pagebreak
	
	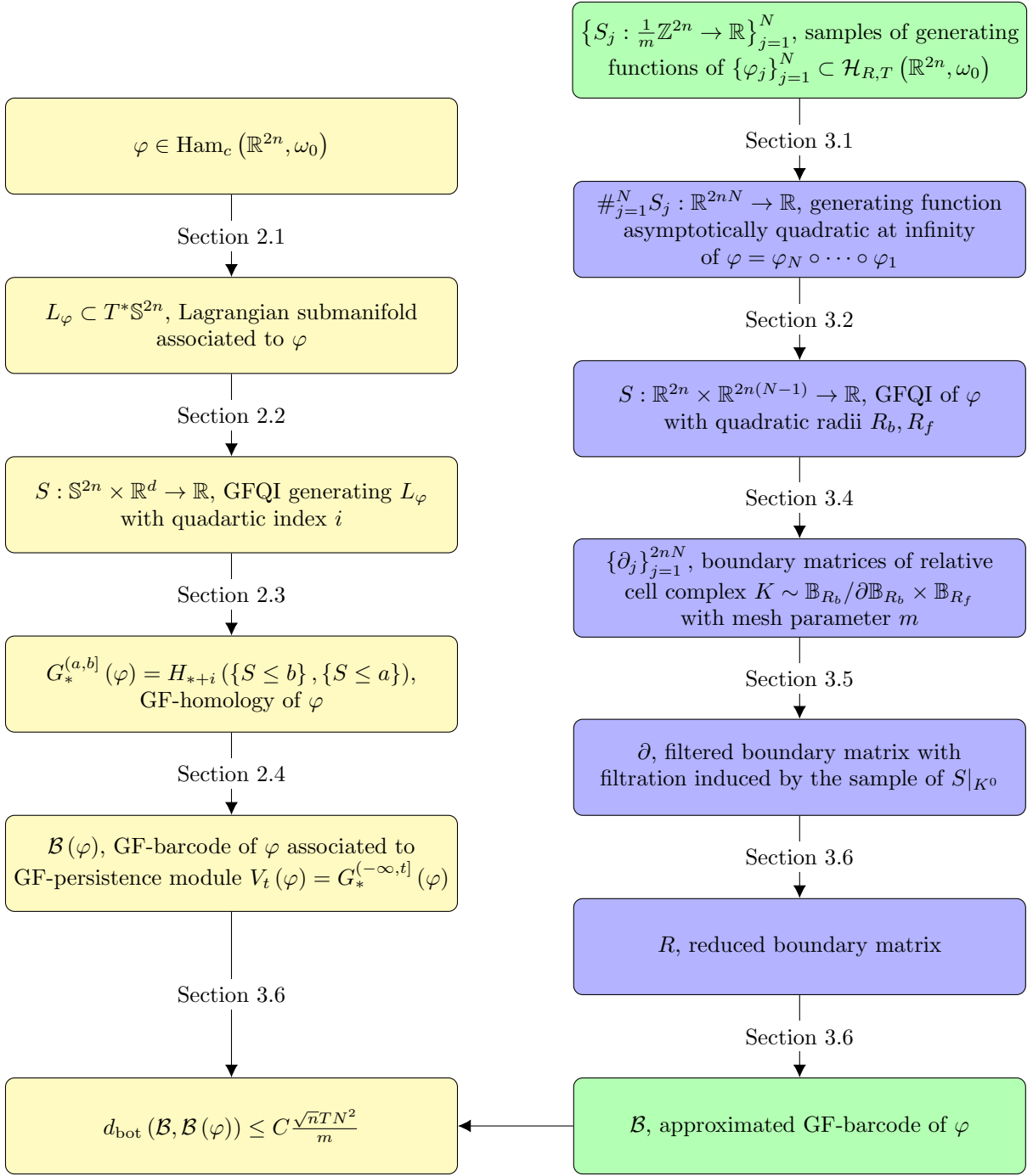
\begin{figure}[h]
		\centering
		\caption{Theoretical construction of GF-barcode (left) and the approximation algorithm (right)}
		\begin{tikzpicture}[node distance=2.8cm, every node/.style={fill=white}, align=center]
			\node (input)[input]{$ \left\{S_j:\frac{1}{m}\ZZ^{2n}\to\RR\right\}_{j=1}^N $, samples of generating\\functions of $ \left\{\varphi_j\right\}_{j=1}^N\subset\mathcal{H}_{R,T}\left(\RRnn,\omega_{0}\right) $};
			\node (compositionfomula)[algostep, below of=input]{$ \#_{j=1}^NS_j:\RR^{2nN}\to\RR $, generating function\\asymptotically quadratic at infinity\\of $ \varphi=\varphi_N\circ\dots\circ\varphi_1 $};
			\node (GFQIformula)[algostep, below of=compositionfomula]{$ S:\RRnn\times\RR^{2n\left(N-1\right)}\to\RR $, GFQI of $ \varphi $\\with quadratic radii $ R_b,R_f $};
			\node (bndmat)[algostep, below of=GFQIformula]{$ \left\{\partial_j\right\}_{j=1}^{2nN} $, boundary matrices of relative\\cell complex $ K\sim\BB_{R_b}/\partial\BB_{R_b}\times\BB_{R_f} $\\with mesh parameter $ m $};
			\node (filterredbndmat)[algostep, below of=bndmat]{$ \partial $, filtered boundary matrix with\\filtration induced by the sample of $ S\vert_{K^0} $};
			\node (reducedmat)[algostep, below of=filterredbndmat]{$ R $, reduced boundary matrix};
			\node (approxbarcode)[input, below of=reducedmat]{$ \mathcal{B} $, approximated GF-barcode of $ \varphi $};
			
			\node (Ham)[theory, left of=input,xshift=-6cm,yshift=-1.5cm]{$ \varphi\in\Ham_c\left(\RRnn,\omega_{0}\right) $};
			\node (lagsbmfd)[theory, below of=Ham]{$ L_\varphi\subset T^*\SS^{2n} $, Lagrangian submanifold\\associated to $ \varphi $};
			\node (GFQI)[theory, below of=lagsbmfd]{$ S:\SS^{2n}\times\RR^d\to\RR $, GFQI generating $ L_\varphi $\\with quadartic index $ i $};
			\node (GFhomology)[theory, below of=GFQI]{$ G_\ast^{\left(a,b\right]}\left(\varphi\right) = H_{\ast+i}\left(\left\{S\leq b\right\},\left\{S\leq a\right\}\right) $,\\GF-homology of $ \varphi $};
			\node (GFbarcode)[theory, below of=GFhomology]{$ \mathcal{B}\left(\varphi\right) $, GF-barcode of $ \varphi $ associated to\\GF-persistence module $ V_t\left(\varphi\right) = G_\ast^{\left(-\infty,t\right]}\left(\varphi\right) $};
			\node (bottleneckdist)[theory, left of=approxbarcode,xshift=-6cm]{$ \dbot\left(\mathcal{B},\mathcal{B}\left(\varphi\right)\right)\leq C\frac{\sqrt{n}TN^2}{m} $};
			
			\draw[->] (input) -- node[text width=4cm]{\Cref{sec:compositionformula}}(compositionfomula);
			\draw[->] (compositionfomula) -- node[text width=4cm]{\Cref{sec:obtainingGFQI}}(GFQIformula);
			\draw[->] (GFQIformula) -- node[text width=4cm]{\Cref{sec:step1}}(bndmat);
			\draw[->] (bndmat) -- node[text width=4cm]{\Cref{sec:step2}}(filterredbndmat);
			\draw[->] (filterredbndmat) -- node[text width=4cm]{\Cref{sec:step3}}(reducedmat);
			\draw[->] (reducedmat) -- node[text width=4cm]{\Cref{sec:step3}}(approxbarcode);
			\draw[->] (approxbarcode) -- (bottleneckdist);
			\draw[->] (Ham) -- node[text width=4cm]{\Cref{sec:hamtolagsbmnfd}}(lagsbmfd);
			\draw[->] (lagsbmfd) -- node[text width=4cm]{\Cref{sec:genfunc}}(GFQI);
			\draw[->] (GFQI) -- node[text width=4cm]{\Cref{sec:GFhomology}}(GFhomology);
			\draw[->] (GFhomology) -- node[text width=4cm]{\Cref{sec:GFbarcode}}(GFbarcode);
			\draw[->] (GFbarcode) -- node[text width=4cm]{\Cref{sec:step3}}(bottleneckdist);
		\end{tikzpicture}
	\label{fig:intro}
	\end{figure}

	\pagebreak
	
	\section{Preliminaries}	\label{sec:prelim}
	\subsection{From $ \Ham_c\left(\RRnn,\omega_0\right) $ to exact Lagrangian submanifolds of $ T^*\SS^{2n} $} \label{sec:hamtolagsbmnfd}
	Let us recall how to associate a Lagrangian submanifold of $ T^*\SS^{2n} $ to compactly supported Hamiltonian diffeomorphisms on $ \RRnn $.
	
	\begin{enumerate}
		\item Let $ \varphi\in\Ham_c\left(\RRnn,\omega_0\right) $. Denote the coordinate functions of $ \varphi $ with respect to the coordinates $ q = \left(q_1,\dots,q_n\right), p = \left(p_1,\dots,p_n\right) $ by $ \left(Q_\varphi,P_\varphi\right) = \varphi\left(q,p\right) $ and its graph by $ \Gamma_{\varphi} $, i.e
		\begin{equation*}
			\Gamma_{\varphi} = \left\{\left(\left(q,p\right),\left(Q_\varphi,P_\varphi\right)\right); \left(q,p\right)\in\RRnn\right\} \subset \RRnn\times \RRnn
		\end{equation*}
		
		\item We consider the product space $ \RRnn\times \RRnn $ with coordinates $ \left(\left(q,p\right),\left(Q,P\right)\right) $ and endow it with the symplectic form $ \left(-\omega_0\right)\oplus\omega_0 $. This space (sometimes called a \textit{twisted product}) is denoted by $ \overline{\RRnn} \times \RRnn $. Note that the diagonal
		\begin{equation*}
			\Delta = \left\{\left(\left(q,p\right),\left(q,p\right)\right);\left(q,p\right)\in\RRnn\right\} \subset \overline{\RRnn} \times \RRnn
		\end{equation*}
		is a Lagrangian submanifold and therefore also the graph of $ \varphi $, since $ \Gamma_{\varphi} = \left(\Id\oplus\varphi\right)\left(\Delta\right) $ is the image of a Lagrangian under a symplectomorphism.
		
		\item Next, we can globally identify $ \overline{\RRnn}\times\RRnn $ with $ T^*\RRnn$ via the symplectomorphism
		\begin{equation*}
			\begin{split}
				&\tau: \overline{\RRnn}\times \RRnn \rightarrow T\RRnn \cong T^*\RRnn \\ 
				&\tau\left(\left(q,p\right),\left(Q,P\right)\right) = \left(\left(Q,p\right),\left(P-p,q-Q\right)\right),
			\end{split}
		\end{equation*}
		where the equivalence $ T\RRnn \cong T^*\RRnn $ is done through the standard scalar product of $ \RRnn $. $ \tau $ sends the diagonal $ \Delta $ to $ L_0 $ and we denote
		\begin{equation*}
			L_\varphi = \tau\left(\Gamma_{\varphi}\right) \subset T^*\RRnn,
		\end{equation*}
		so $ L_\varphi $ is a Lagrangian submanifold of $ T^*\RRnn $. Note that if $\varphi$ is $C^1$-small then $L_\varphi$ is a graph of an exact $1$-form.
		
		\item Finally, since $ \varphi $ is compactly supported, $ \Gamma_{\varphi} $ coincides with $ \Delta $ outside the compact support and thus $ L_\varphi $ coincides with $ L_0 $ outside a compact set, so we may consider $ L_\varphi $ as a subset of the cotangent bundle over the compactification $ \RRnn\cup\left\{\infty\right\} \cong \SS^{2n} $, we keep the notation $ L_\varphi $ for the Lagrangian $ L_\varphi \subset T^*\SS^{2n} $.
	\end{enumerate}
	
	The following is a well known argument so we mention it here without proof (see \cite[example 3.4.14]{McDSal17} for the first part and \cite[remark 9.4.7]{McDSal17}).
	
	\begin{lemma} \label{lemma:hamtolag}
		The mapping $ \varphi \mapsto L_\varphi $ is well defined. Furthermore, $ L_\varphi $ is Hamiltonian isotopic to $ L_0 $ and (non-degenerate) fixed points of $ \varphi $ correspond to (transversal) intersection of $ L_\varphi\cap L_0 $.
	\end{lemma}
	
	\begin{remark}
		The construction in \Cref{lemma:hamtolag} depends on the global identification of $ \overline{\RRnn}\times\RRnn $ with $ T^*\RRnn$, which does not exist for general symplectic manifolds. However, Weinstein's \textit{Lagrangian neighborhood theorem} states that any Lagrangian $ L\subset M $ has a tubular neighborhood in $ M $ which is symplectomorphic to a tubular neighborhood of the $ 0 $-section in $ T^*L $ \cite[theorem 3.4.13]{McDSal17}. In particular, a neighborhood of the diagonal $ \Delta \subset \overline{M}\times M $ is symplectomorphic to a neighborhood of the $ 0 $-section $ L_0\subset T^*M $ so the association of the Lagrangian $ L_\varphi\subset T^*M $ to $ \varphi\in\Ham\left(M,\omega\right) $ also holds in case $ M $ is a general closed symplectic manifold as long as $ \Gamma_\varphi $ is close to $ \Delta $, which is the case for example when $ \varphi $ is close to the identity in the $ C^1 $ topology.
	\end{remark}
	
	\subsection{Generating functions of Lagrangian submanifolds} \label{sec:genfunc}
	\begin{remark}
		The term \textit{generating functions} appears in more than one way in symplectic topology, in different contexts this term has a slightly different meaning and definition (even though all describe the same general idea), see for example \cite[chapter 9]{McDSal17} or \cite[section 48]{Arn89}. Our approach uses the concept of generating function associated to a Lagrangian submanifold in a cotangent bundle (sometimes referred to as a \textit{variational family}, see \cite[section 9.4]{McDSal17}). These generating functions were initially defined on a vector bundle with infinite dimensional fibers but approximation methods revealed the existence of such functions with finite dimensional domains (see for example \cite{Vit92}).
	\end{remark}
	
	Let $ \varphi \in \Ham_c\left(\RRnn,\omega_0\right) $, note that if $ \varphi $ is sufficiently $C^1$-small then its graph $ \Gamma_\varphi $ projects diffeomorphically on the diagonal $ \Delta $ and so the submanifold $ L_\varphi $ is the graph $ \Gamma_\alpha $ of some exact $ 1 $-form $ \alpha $, seen as a section $ \alpha:\RRnn\to T^*\RRnn $ and so it is natural to define a generating function of $ \varphi $ as follows.
	\begin{definition} \label{def:smallgenfunc}
		For $ \varphi\in\Ham_c\left(\RRnn,\omega_0\right) $ such that $ L_\varphi $ is the graph of some $ 1 $-form $ \alpha $, we call the function $ S:\RRnn\to\RR $ which satisfies $ dS=\alpha $ a \textit{generating function} of $ \varphi $.
	\end{definition}
	
	In order to define generating functions for not necessarily $C^1$-small Hamiltonian diffeomorphisms, we follow an idea which goes back to H\"ormander (see \cite{Hor71}) and later developed by Viterbo in \cite{Vit92}, which roughly means that even though a certain Lagrangian $ L\subset T^*\RRnn $ may not be the graph of a $ 1 $-form, increasing the dimension of the function's domain allows one to define a more complicated generating function that recovers  $L$.
	\begin{definition} \label{def:laggenfunc}
		Let $ B $ be a smooth manifold, a function $ S $ defined on the vector bundle $ p:B\times \RR^d \rightarrow B $ ($ p\left(x,\xi\right) = x $) is called a \textit{generating function} of a Lagrangian submanifold $ L\subset T^*B $ if
		\begin{align}
			0&\in\left(\RR^d\right)^* \text{is a regular value of} \frac{\partial S}{\partial \xi}:B\times \RR^d \rightarrow \left(\RR^d\right)^* \\
			L &= \left\{\left(x,v^*\right)\in T^*B; \frac{\partial S}{\partial \xi}\left(x,\xi\right) = 0, \frac{\partial S}{\partial x}\left(x,\xi\right) = v^*\right\}. \label{eq:lagofgenfunc_cond2}
		\end{align}
		In case $L=L_\varphi$ for some Hamiltonian diffeomorphism $\varphi$, we call $S$ the {\it generating function} of $\varphi$.
	\end{definition}
	The first condition implies that the set $ \Sigma_S = \left\{\left(x,\xi\right);\frac{\partial S}{\partial\xi}\left(x,\xi\right) = 0\right\} $ of \textit{fiber critical points} is a smooth submanifold of $ B\times\RR^d $, then we associate to each $ \left(x,\xi\right) \in \Sigma_S $ an element $ v^*_{\left(x,\xi\right)}\in T^*_xB $ (sometimes called the \textit{Lagrange multiplier} of $ \left(x,\xi\right) $) such that
	\begin{align*}
		&v^*_{\left(x,\xi\right)} : T_xB \to \RR \\
		&v^*_{\left(x,\xi\right)}\left(v\right) = dS_{\left(x,\xi\right)}\left(\widehat{v}\right)
	\end{align*}
	where $ \widehat{v} $ is a lift of $ v\in T_xB $ to $ T_{\left(x,\xi\right)}\left(B\times\RR^d\right) $, i.e such that $ dp\left(\widehat{v}\right) = v $. The fact that $ \left(x,\xi\right)\in\Sigma_S $ implies that the map $ \left(x,\xi\right) \mapsto v^*_{\left(x,\xi\right)} $ do not depend on the choice of lift $ \widehat{v} $ and we denote by $ i_S:\Sigma_S\to T^*B $ the map $ i_S\left(x,\xi\right) = \left(x,v^*\left(x,\xi\right)\right) $ whose image is an exact Lagrangian $ i_S\left(\Sigma_S\right) $ generated by $ S $ in the sense of \Cref{def:laggenfunc}.	Moreover, note that just as in the case of \Cref{def:smallgenfunc}, if a function $ S $ generate the Lagrangian $ L $ in the sense of \Cref{def:laggenfunc} then non-degenerate critical points still correspond to transversal intersection of $ L\cap L_0 $, which is an important property of these generating functions, as they relate the Morse homology of $ S $ to periodic orbits and action spectrum of $ \varphi\in\Ham_c\left(\RRnn,\omega_0\right) $ (when $ L=L_\varphi $).
	\begin{remark}
		Since this paper concerns only with $ \Ham_c\left(\RRnn,\omega_0\right) $, we only give the definition for generating functions of Lagrangian submanifolds as functions defined on the trivial bundle $ B\times\RR^d $, but it is worth mentioning that in fact this construction generalizes to functions on fiber bundles $ p:E\to B $ which are not trivial (or even finite dimensional). This general approach allows us to view the \textit{action functional} associated to a Hamiltonian $ H_t:T^*B\to\RR $ as a kind of infinite dimensional generating function defined on the set of paths in $ T^*B $ which starts at the $ 0 $-section (see \cite[example 9.4.8]{McDSal17}). Then, the Lagrangian defined by this generating function is exactly the image of the $ 0 $-section under the Hamiltonian flow generated by $ H_t $. This observation was used by Viterbo \cite{Vit92}, Laudenbach, Sikorav \cite{LaudSik85} and Chaperon \cite{Chap90} to construct finite dimensional generating function by using discrete approximation of the paths on which the action functional operates.
	\end{remark}
	
	Note that If $ S $ is a generating function of $ L\subset T^*B $ defined on the vector bundle $ p:B\times\RR^d\to\RRnn $, then any $ \widetilde{S} $ which is obtained from $ S $ by one of the following operations (which we shall call \textit{basic operations}), is also a generating function of $ L $
	\begin{itemize}
		\item (Addition of a constant) $ \widetilde{S} = S + c $ for some $ c\in\RR $.
		\item (Fiber preserving diffeomorphism) $ \widetilde{S} = S\circ \Phi $ where $ \Phi:B\times\RR^d \rightarrow B\times \RR^d $ satisfies $ p \circ \Phi = p $.
		\item (Stabilization) $ \widetilde{S} = S\oplus\mathcal{Q} : B\times\RR^d \times\RR^{d^\prime}\rightarrow\RR $ where $ \mathcal{Q} $ is a non-degenerate quadratic form.
	\end{itemize}
	Next we follow \cite{Tr94} to define a condition which controls how a generating function behaves at infinity.
	\begin{definition} \label{def:asympquadatinfty}
		A generating function $ S:B\times\RR^d\to\RR $ is called \textit{asymptotically quadratic at infinity} if there exists a function $ \mathcal{Q}:B\times\RR^d\to\RR $ which is a non-degenerate quadratic form in each fiber, such that for each $ x\in B $,
		\begin{equation*}
			\left|\frac{\partial}{\partial \xi}\left(S\left(x,\xi\right)-\mathcal{Q}\left(x,\xi\right)\right)\right| \leq M_x < \infty
		\end{equation*}
	\end{definition}
	\begin{remark}
		In some texts, the asymptotically quadratic at infinity property is called \textit{almost quadratic at infinity} and when the function $ \mathcal{Q} $ is independent of the base variable this generating function is called \textit{``special"}. Th\'eret shows in \cite[24]{Th96} that every generating function can be made special through basic operations.
	\end{remark}
	Note that the three basic operations preserves the condition of being asymptotically quadratic at infinity and we call two generating function $ S_1,S_2 $ of $ L\subset T^*B $ \textit{equivalent}, if they can be made equal after a succession of basic operations. The following theorem was proved by Viterbo in \cite{Vit92} (see also Th\'eret in \cite{Th99}).
	\begin{theorem}[Viterbo's Uniqueness Theorem] \label{thm:gfqiunique}
		Let $ B $ be a closed manifold and let $ \psi^t $ be a Hamiltonian flow in $ T^*B $, $ t\in\left[0,1\right] $. Then all generating functions of $ L = \psi^1\left(L_0\right) $ which are asymptotically quadratic at infinity are equivalent.
	\end{theorem}

	A key feature of the following definition is that one is able to approximate these kinds of generating functions numerically.
	\begin{definition}
		A generating function $ S:B\times\RR^d\to\RR $ is called \textit{quadratic at infinity} if $ S\left(x,\xi\right) = \mathcal{Q}\left(\xi\right) $ outside a compact set $K \subset B\times\RR^d $, where $ \mathcal{Q}:\RR^d\to\RR $ is a non-degenerate quadratic form. We shall abbreviate the name of such functions as \textit{GFQI}. The index of $ \mathcal{Q} $ is called the \textit{quadratic index} of $ S $ and in case $ B $ is the euclidean space, a pair of radii $ R_b,R_f > 0 $ such that $K \subset \BB_{R_b}\times\BB_{R_f} $ are called \textit{quadratic radii} of $ S $.
	\end{definition}
	
	It holds that any asymptotically quadratic at infinity generating function can be made a GFQI using basic operations (see \cite{Th96}), however the proof involves a more theoretic argument and since we require an explicit expression we later use a slightly different approach in order to turn a generating function which is asymptotically quadratic at infinity (with some extra conditions) into a GFQI. Later on we show how one can construct a GFQI for $ \varphi\in\Ham_c\left(\RRnn,\omega_0\right) $ but it's worth mentioning at this point that their existence in general is a result by Laudenbach and Sikorav (see \cite{LaudSik85}, \cite{Sik87}).
	
	These steps conclude the construction of a GFQI $ S:\RRnn\times\RR^d\to\RR $ associated to a compactly supported Hamiltonian diffeomorphism $ \varphi $ of $ \RRnn $, since the uniqueness result requires the base manifold to be compact, we extend $ S $ to the compactification $ \left(\RRnn\cup\left\{\infty\right\}\right)\times\RR^d $ in a natural way by setting $ S\left(\infty,\xi\right) = \mathcal{Q}\left(\xi\right) $. We denote $ \RRnn\cup\left\{\infty\right\} $ by $ \SS^{2n} $ from now on.
	\begin{remark}
		An interesting property of this function is that non-degenerate critical points correspond to transversal intersections of $ L_\varphi $ with $ L_0 $, which in turn correspond to fixed points of $ \varphi $, moreover, after normalization we also get that critical values of $ S $ correspond to the action spectrum of $ \varphi $. These properties allows us to define a meaningful invariant for $ \varphi $ by considering the Morse homology of $ S $, which is described in the next subsection.
	\end{remark}
	
	Let us finish this subsection by explicitly defining a generating function $ S $ for $ L_\varphi $ where $ \varphi $ is the composition of several $ \varphi_j\in\Ham_c\left(\RRnn,\omega\right) $ all of which are $ C^1 $ small, in terms of the generating functions $ S_j $ associated to $ \varphi_j $. This is done using a composition formula which is attributed to Chekanov and appears in \cite{Chap90}.
	\begin{proposition}[Chekanov's formula] \label{prop:compformula}
		Let $ \varphi,\psi\in\Ham\left(\RRnn,\omega_0\right) $ and assume $ L_\varphi $ is generated by $ S_\varphi:\RRnn\times\RR^d\to\RR $ and $ L_\psi $ is generated by $ S_\psi:\RRnn\to\RR $. Then, $ L_{\psi\circ\varphi} $ is generated by
		\begin{align*}
			&S_\varphi\#S_\psi:\RRnn\times\left(\RR^d\times\RRnn\right) \to \RR \\
			&S_\varphi\#S_\psi\left(\left(Q^\prime,p\right),\left(\xi,Q,P\right)\right) = S_\varphi\left(\left(Q,p\right),\xi\right) + S_\psi\left(Q^\prime,P\right) + \left\langle P-p,Q^\prime-Q\right\rangle.
		\end{align*}
		Where we denote $ \varphi\left(q,p\right) = \left(Q,P\right) $ and $ \psi\left(Q,P\right) = \left(Q^\prime,P^\prime\right) $.
	\end{proposition}
	
	\begin{remark}
		Note that one may associate a generating function to every $ \varphi\in\Ham_c\left(\RRnn,\omega_0\right) $, since $ \varphi = \varphi^1_H $ where $ \varphi^t_H $ is a Hamiltonian flow we may take a sequence of times $ 0 < t_1 < \dots < t_N = 1 $ and break $ \varphi $ into the composition
		\begin{equation*}
			\varphi = \left(\varphi^{t_N}_H\circ\left(\varphi^{t_{N-1}}_H\right)^{-1}\right)\circ\dots\circ\left(\varphi^{t_j}_H\circ\left(\varphi^{t_{j-1}}_H\right)^{-1}\right)\circ\dots\circ\left(\varphi^{t_2}_H\circ\left(\varphi^{t_1}_H\right)^{-1}\right)\circ\varphi^{t_1}_H,
		\end{equation*}
		so when $ N $ is large enough we get that $ \varphi_j = \varphi^{t_j}_H\circ\left(\varphi^{t_{j-1}}_H\right)^{-1} $ is $ C^1 $ close enough to the identity so that there exists a corresponding generating function $ S_j $ without auxiliary variables (see \Cref{prop:smallhamtogenfunc} below).
	\end{remark}

	\subsection{Generating function homology} \label{sec:GFhomology}
	Let $ \varphi\in \Ham_c\left(\RRnn,\omega_0\right) $, then $ \varphi $ has a (unique, up to equivalence) GFQI, $ S:E\rightarrow\RR $ where $ E = \SS^{2n}\times \RR^d $ and we consider the homology of sublevel sets of $ S $ as in Morse theory. Though $ S $ is defined on a non compact space, the quadratic at infinity condition allows us to define a relative version for the homology of sublevel sets due to the nature of gradient flow arguments in Morse theory.
	\begin{definition}
		Let $ L\subset T^*B $ be a Lagrangian submanifold which coincides with the $ 0 $-section $ L_0 $ outside a compact set of $ B $ and assume $ S $ generates $ L $, then we call $ S $ a \textit{normalized} generating function of $ L $ if the critical value of $ S $ associated to the intersection of $ L\cap L_0 $ outside this compact set is $ 0 $.
	\end{definition}
	Following \cite{Tr94} we give the next definition.
	\begin{definition} \label{def:GFhomology}
		For $ \varphi\in\Ham_c\left(\RRnn,\omega_0\right) $, let $ S:\RRnn\times\RR^d\to\RR $ be a normalized GFQI with quadratic part $ \mathcal{Q} $, extend $ S $ to $ E = \SS^{2n}\times\RR^d $ and let $ 0<a<b\leq\infty $ such that $ a,b $ are not critical values of $ S $. The \textit{GF-homology} groups of $ \varphi $ with respect to $ \left(a,b\right] $ are
		\begin{equation*}
			G_\ast^{\left(a,b\right]}\left(\varphi\right) = H_{\ast+i}\left(E^b,E^a\right),
		\end{equation*}
		where $  E^c = \left\{\left(x,\xi\right)\in\SS^{2n}\times\RR^d; S\left(x,\xi\right) \leq c\right\} $ and $ i $ is the index of $ \mathcal{Q} $.
	\end{definition}
	In this paper we use homology with coefficients in $ \ZZ_2 $ and the uniqueness given by \Cref{thm:gfqiunique} implies that $ G_\ast^{\left(a,b\right]}\left(\varphi\right) $ is well defined, as proved in \cite[lemma 3.6]{Tr94}:
	\begin{proposition}
		$ G_\ast^{\left(a,b\right]}\left(\varphi\right) $ is independent on the choice of normalized GFQI for $ \varphi $.
	\end{proposition}
	\begin{remark}
		The groups $ G_\ast^{\left(a,b\right]}\left(\varphi\right) $ are indeed invariant to conjugation, i.e for $ \psi\in\Ham_c\left(\RRnn,\omega_0\right) $ there is an induced isomorphism
		\begin{equation*}
			\psi^*:G_\ast^{\left(a,b\right]}\left(\psi\varphi\psi^{-1}\right) \to G_\ast^{\left(a,b\right]}\left(\varphi\right).
		\end{equation*}
		Though this is not the focus of this paper, its worth noting that Viterbo used this invariance together with a partial order on $ \Ham_c\left(\RRnn,\omega_0\right) $ (which translates to partial order of the generating functions) to define Symplectic homology groups associated to open sets in $ \RRnn $ which is invariant under symplectic transformations, further details can be found in \cite{Vit92} and \cite{Tr94}. We give this remark to emphasize a motivation of being able to calculate $ G_\ast^{\left(a,b\right]}\left(\varphi\right) $ explicitly.
	\end{remark}
	
	\paragraph[GF homology inv' to lower value]{}
	Recall that in classical Morse theory, two sublevel sets $ f^{-1}\left(\left(-\infty,a\right]\right) $ and $f^{-1}\left(\left(-\infty,b\right]\right) $ of a Morse function $ f:M\to\RR $ such that $ f^{-1}\left(\left[a,b\right]\right) $ is compact and contains no critical points are diffeomorphic. Roughly the same argument applies in the case of a GFQI even though the function is not compactly supported, if we consider $ a^\prime < a $ such that $ S^{-1}\left(\left[a^\prime,a\right]\right) $ lies entirely inside the region where $ S $ is a quadratic form $ \mathcal{Q} $, then even though $ S^{-1}\left(\left[a^\prime,a\right]\right) $ is not compact, it is given explicitly by $ \mathcal{Q}^{-1}\left(\left[a^\prime,a\right]\right) $ and so for general $ a^\prime < a $ we split $ S^{-1}\left(\left[a^\prime,a\right]\right) $ into the quadratic part and the non-quadratic part (which is compactly supported) so the following holds.
	\begin{lemma} \label{lemma:deformrelativesmallvalues}
		Let $ \varphi\in\Ham_c\left(\RRnn,\omega_0\right) $ and let $ a^\prime < a $, if the interval $ \left[a^\prime,a\right] $ does not contain any spectral value of $ \varphi $ (i.e the action of fixed points of $ \varphi $), then
		\begin{equation*}
			G_\ast^{\left(a,b\right]}\left(\varphi\right) = G_\ast^{\left(a^\prime,b\right]}\left(\varphi\right).
		\end{equation*}
	\end{lemma}
	
	\begin{proof}
		Let $ S:\SS^{2n}\times\RR^d\to\RR $ be a normalized GFQI of $ \varphi $ with quadratic part $ \mathcal{Q}:\RR^d\to\RR $, recall that $ G_\ast^{\left(a,b\right]}\left(\varphi\right) = H_{\ast+i}\left(E^b,E^a\right) $ where $ E^c $ denotes the sub-level set $ S $ with value $ c\in\RR $ and $ E = \SS^{2n}\times\RR^d $. Let $ a^\prime < a < b $ such that the interval $ \left[a^\prime,a\right] $ does not intersect the action spectrum of $ \varphi $, since the action spectrum of $ \varphi $ is equal to the set of critical values of $ S $ we get that $ S^{-1}\left(\left[a^\prime,a\right]\right) $ does not contain any critical points, we want to define a map that takes the pair $ \left(E^b,E^{a^\prime}\right) $ to the pair $ \left(E^b,E^a\right) $ and is homotopic to the identity through maps that preserves $ E^a $. We follow the classical arguments (see \cite[section 3]{Mil68}) of Morse theory and define a $ 1 $-parameter group of diffeomorphisms which takes $ E^{a^\prime} $ to $ E^a $, then we use this group to define the desired homotopy.
		
		Note that the Morse theory arguments require the set $ S^{-1}\left(\left[a^\prime,a\right]\right) $ to be compact which is not the case here since this set is unbounded on the fibers. So, we consider a compactification of $ E $ as the space $ \overline{E} = \SS^{2n}\times\left(\RR^d\cup\left\{\infty\right\}\right) \cong \SS^{2n}\times\SS^d $, even though the function $ S $ cannot be continuously extended to $ \overline{E} $, it is equal to $ \mathcal{Q} $ outside a compact set of $ E $ so the vector field $ \frac{\nabla S}{\left|\nabla S\right|^2} $ (defined away from critical points of $ S $) is in fact $ \frac{\nabla\mathcal{Q}}{\left|\nabla\mathcal{Q}\right|^2} $ outside a compact set and since $ \mathcal{Q} $ is a non-degenerate quadratic form, this vector field can be smoothly extended to $ \SS^{2n}\times\left\{\infty\right\} $ by $ 0 $. Now, the set $ S^{-1}\left(\left[a^\prime,a\right]\right)\cup\left\{\infty\right\} $ is compact, we denote by $ K $ a compact neighborhood of $ S^{-1}\left(\left[a^\prime,a\right]\right)\cup\left\{\infty\right\} $ and consider a smooth function $ \rho:\overline{E}\to\RR $ such that
		\begin{equation*}
			\rho\left(x,\xi\right) = \begin{cases}
				\frac{1}{\left|\nabla S\left(x,\xi\right)\right|^2} & \left(x,\xi\right)\in S^{-1}\left(\left[a^\prime,a\right]\right) \\
				0 & \xi = \infty\text{ or }\left(x,\xi\right)\notin K.
			\end{cases}
		\end{equation*}
		The vector field $ \rho\nabla S $ is then well defined on $ \overline{E} $ and vanishes outside a compact set, so it generates a flow $ \phi_t:\overline{E}\to\overline{E} $ such that $ \frac{\diff}{\diff t}\big\vert_{t_0}\phi_t\left(x,\xi\right) = \rho\left(\phi_{t_0}\left(x,\xi\right)\right)\nabla S\left(\phi_{t_0}\left(x,\xi\right)\right) $ (see \cite[lemma 2.4]{Mil68}). We define a map $ r_t:\left[0,1\right]\times E^b\to E^b $ by
		\begin{equation*}
			r_t\left(x,\xi\right) = \begin{cases}
				\left(x,\xi\right) & S\left(x,\xi\right)\geq a \\
				\phi_{t\left(a-S\left(x,\xi\right)\right)} & S\left(x,\xi\right)\leq a,
			\end{cases}
		\end{equation*}
		then since $ \frac{\diff}{\diff t}\big\vert_{t_0}S\left(\phi_t\left(x,\xi\right)\right) = 1 $ whenever $ \phi_{t_0}\left(x,\xi\right)\in S^{-1}\left(\left[a^\prime,a\right]\right) $, the curves $ \phi_{t\left(a-S\left(x\xi\right)\right)}\left(x,\xi\right) $ stay inside $ S^{-1}\left(\left[a^\prime,a\right]\right) $ for $ t\in\left[0,1\right] $ and $ r_1\left(E^{a^\prime}\right) = E^a $. Note that $ r_t\left(E^a\right)\subset E^a $ for all $ t\in\left[0,1\right] $ so the map $ r_t $ is our desired homotopy between $ r_0=\Id:\left(E^b,E^{a^\prime}\right)\to\left(E^b,E^a\right) $ and $ r_1:\left(E^b,E^{a^\prime}\right)\to\left(E^b,E^a\right) $ therefore they induce the same map on relative homology $ \left(\Id\right)_* = \left(r_1\right)_*:H_{\ast+i}\left(E^b,E^{a^\prime}\right)\to H_{\ast+i}\left(E^b,E^{a}\right) $, thus
		\begin{equation*}
			H_{\ast+i}\left(E^b,E^{a^\prime}\right) = \left(\Id\right)_*\left(H_{\ast+i}\left(E^b,E^{a^\prime}\right)\right)=\left(r_1\right)_*\left(H_{\ast+i}\left(E^b,E^{a^\prime}\right)\right) = H_{\ast+i}\left(E^b,E^{a}\right),
		\end{equation*}
		which concludes the proof.
	\end{proof}
	
	Finally, $ S $ is a GFQI so it has quadratic radii $ R_b,R_f > 0 $ such that $ S $ is equal to a non-degenerate quadratic form $ \mathcal{Q} $ outside $ \BB_{R_b}\times\BB_{R_f} \subset\RRnn\times\RR^d $ (and recall that $ \SS^{2n} = \RRnn\cup\left\{\infty\right\} $). We compose $ S $ with a linear fiber preserving diffeomorphism such that in the new coordinates $ \left(\xi^-,\xi^+\right) = \left(\xi_1^-,\dots,\xi_i^-,\xi_{i+1}^+,\dots,\xi_d^+\right) $ of $ \RR^d $ we have
	\begin{equation*}
		\mathcal{Q}\left(\xi\right) = -\left(\xi_1^-\right)^2-\dots-\left(\xi_i^-\right)^2+\left(\xi_{i+1}^+\right)^2+\dots+\left(\xi_d^+\right)^2,
	\end{equation*}
	where as before, $ i $ is the quadratic index of $ S $ (which is the index of $ \mathcal{Q} $) and we call the function $ S $ in these new coordinated a \textit{diagonalized GFQI}. We may also assume that $ R_f $ is large enough so that
	\begin{equation} \label{eq:Rflarge}
		-R_f^2 = \min_{\left|\xi\right| = R_f} Q\left(\xi\right) < \min_{\left(x,\xi\right)\in\BB_{R_b}\times\BB_{R_f}} S\left(x,\xi\right),
	\end{equation}
	or in other words, that the radius along the fibers is large enough such that the sublevel sets $ E^a $ for $ a \leq -R_f^2 $ do not intersect $ \BB_{R_b}\times\BB_{R_f} $. According to \Cref{lemma:deformrelativesmallvalues}, the groups $ G_\ast^{\left(a,b\right]}\left(\varphi\right) $ are all the same for such $ a $ (and $ b $ is fixed) so we denote by $ E^{-\infty} $ the set $ E^a $ for $ a \leq -R_f^2 $ and accordingly
	\begin{equation} \label{eq:G^b}
		G_\ast^{\left(-\infty,b\right]}\left(\varphi\right) = H_{\ast+i}\left(E^b,E^{-\infty}\right).
	\end{equation}

	\subsection{Generating function barcode} \label{sec:GFbarcode}
	Following \cite{Pol19}, we introduce the concept of persistence modules which is an algebraic tool introduced by G. Carlsson and A. Zamorodian \cite{ZomCar05} that allows one to measure how the groups $ G^{\left(-\infty,b\right]}_{\ast}\left(\varphi\right) $ associated to $ \varphi\in\Ham_c\left(\RRnn,\omega_0\right) $ change as a function of $ b $, then we use barcodes to associate a combinatorial object to this persistence module.
	
	\paragraph[Persistence modules]{}
	We first recall some of the basic definitions in the theory of persistence modules and barcodes:
	\begin{definition}
		A \textit{persistence module} (abbreviated as p.m) is a pair $ \left(V,\pi\right) $, where $ V $ is a collection $ \left\{V_t\right\}_{t\in\RR} $ of finite dimensional vector spaces over a field $ \FF $, and $ \pi $ is a collection $ \left\{\pi_{s,t}\right\}_{s\leq t\in\RR} $ of linear maps $ \pi_{s,t}:V_s\to V_t $ such that
		\begin{enumerate}
			\item (\textit{Persistence}) The collection $ \pi $ satisfies $ \pi_{s,r}=\pi_{t,r}\circ\pi_{s,t} $ for any $ s\leq t\leq r $,
			\item For all but a finite number of points $ t\in\RR $ there exists a neighborhood $ U $ of $ t $, such that $ \pi_{s,r} $ is an isomorphism for any $ s<r $ in $ U $,
			\item (\textit{Semicontinuity}) For any $ t\in\RR $ and any $ s\leq t $ sufficiently close to $ t $, the map $ \pi_{s,t} $ is an isomorphism,
			\item There exists some $ s_- \in\RR $, such that $ V_s =0 $ for any $ s<s_- $.
		\end{enumerate}
	\end{definition}
	\begin{example}[Interval modules]
		For any interval $ \left(a,b\right] $ where $ -\infty<a<b\leq\infty $, the persistence module $ \FF\left(a,b\right] $ defined as
		\begin{equation*}
			\left(\FF\left(a,b\right]\right)_t = \begin{cases}
				\FF & t\in\left(a,b\right] \\
				0 & t\notin\left(a,b\right]
			\end{cases}, \pi_{s,t} = \begin{cases}
				\Id & \left\{s,t\right\}\subseteq\left(a,b\right] \\
				0 & \left\{s,t\right\}\nsubseteq\left(a,b\right]
			\end{cases},
		\end{equation*}
		is called an \textit{interval module}.
	\end{example}
	\begin{example} \label{example:morse p.m}
		Another example of a p.m which is important to us is given by Morse theory, given a closed manifold $ M $ and a Morse function $ f:M\to\RR $ we define a p.m $ V\left(f\right) $ by setting $ V_t\left(f\right) = H_\ast\left(\left\{x\in M;f\left(x\right)<t\right\};\FF\right) $ (note that we take here sub-level sets with strict inequality, this is done in order to satisfy semicontinuity). Then for $ s\leq t $ the inclusions $ \left\{f<s\right\}\overset{in_{s,t}}{\hookrightarrow}\left\{f<t\right\} $ induce linear maps in homology $ \left(in_{s,t}\right)_*:V_s\to V_t $ so $ V\left(f\right) $ defines a p.m.
	\end{example}

	\paragraph[Morphisms and normal form]{}
	Interval modules serve as the ``basic blocks" in the sense that every other p.m decomposes into a direct sum of such interval modules. In order to make sense of this decomposition we need to define morphisms between persistence modules and the notion of isomorphism:
	\begin{definition}
		A p.m morphism $ A:V\to V^\prime$ is a collection $ \left\{A_t\right\}_{t\in\RR} $ of linear maps $ A_t:V_t\to V^\prime_t $ that satisfies $ \pi^\prime_{s,t}\circ A_s = A_t\circ\pi_{s,t} $ for all $ s\leq t $.
	\end{definition}
	Two p.m $ V $ and $ V^\prime $ are then called \textit{isomorphic} and denoted $ V\cong V^\prime $ if there exists two p.m morphisms $ A:V\to V^\prime $ and $ B:V^\prime\to V $ such that $ B\circ A = \Id_{V} $ and $ A\circ B = \Id_{V^\prime} $. A key feature of p.m is that it can be expressed as a sum of intervals modules (with different powers), this is called the normal form theorem and it appears e.g. in \cite{Pol19}.
	\begin{theorem}[Normal form theorem] \label{thm:normalform}
		Let $ \left(V,\pi\right) $ be a persistence module. Then there exists a finite collection $ \left\{\left(I_i,m_i\right)\right\}_{i=1}^N $ of pair-wise distinct intervals $ I_i = \left(a_i,b_i\right] $ ($ -\infty < a_i < b_i \leq \infty $) with their multiplicities $ m_i\in\NN $, such that
		\begin{equation*}
			V \cong \bigoplus_{i=1}^N \FF\left(I_i\right)^{m_i}.
		\end{equation*}
		Moreover, this identification is unique up to permutations.
	\end{theorem}

	\paragraph[Interleaving distance]{}
	In order to define a metric on the space of p.m we first establish the concept of "shifts".
	\begin{example}[Shifted p.m]
		Let $ V $ be a p.m and $ \delta\in\RR $, the $ \delta $\textit{-shift} of $ V $, denoted by $ V\left[\delta\right] $, is a p.m defined by: $ \left(V\left[\delta\right]\right)_t = V_{t+\delta} $ and $ \left(\pi\left[\delta\right]\right)_{s,t} = \pi_{s+\delta,t+\delta} $.
	\end{example}
	Note that for a p.m $ V $ and $ \delta\in\RR $, the collection $ \Phi^\delta = \left\{\Phi^\delta_t\right\}_{t\in\RR} = \left\{\pi_{t,t+\delta}\right\}_{t\in\RR} $ defines a p.m morphism $ \Phi_V^\delta:V\to V\left[\delta\right] $, which is called the $ \delta $\textit{-shift morphism}. Furthermore, if $ A:V\to V^\prime $ is a p.m morphism, we denote by $ A\left[\delta\right]:V\left[\delta\right]\to V^\prime\left[\delta\right] $ the morphism such that $ \left(A\left[\delta\right]\right)_t = A_{t+\delta} $.
	\begin{definition}
		Given $ \delta > 0 $, two persistence modules $ \left(V,\pi\right) $ and $ \left(W,\theta\right) $ are called $ \delta $\textit{-interleaved} if there exist morphisms $ F:V\to W\left[\delta\right] $ and $ G:W\to V\left[\delta\right] $ such that the following diagrams commutes
		\begin{equation*}
			\xymatrix@C=2pc
			{
				V \ar[r]^-{F} \ar@/_1pc/[rr]_-{\Phi_V^{2\delta}} & W\left[\delta\right] \ar[r]^-{G\left[\delta\right]} & V\left[2\delta\right] && W \ar[r]^-{G} \ar@/_1pc/[rr]_-{\Phi_W^{2\delta}} & V\left[\delta\right] \ar[r]^-{F\left[\delta\right]} & W\left[2\delta\right]
			},
		\end{equation*}
		we refer to $ F $ and $ G $ as $ \delta $\textit{-interleaving morphisms}.
	\end{definition}
	Such $ \delta > 0 $ exists if and only if the spaces $ V_\infty $ and $ W_\infty $ are isomorphic and it also satisfies the properties needed to define a pseudo-metric on the space of (isomorphism classes of) persistence modules with isomorphic $ V_\infty $ by:
	\begin{definition}
		The \textit{interleaving distance} between two p.m $ V,W $ with $ V_\infty \cong W_\infty $ is
		\begin{equation*}
			\dint\left(V,W\right) = \inf\left\{\delta > 0 ; \text{$ V $ and $ W $ are $ \delta $-interleaved}\right\}
		\end{equation*}.
	\end{definition}
	Then $ \dint $ is an actual metric, i.e does not vanish on non-isomorphic p.m. See \cite[section 1.3]{Pol19} for more details.
	
	\paragraph[Barcodes and the isometry theorem]{}
	\begin{definition}
		A \textit{barcode} $\mathcal{B}$ is a finite collection $ \left\{\left(I_i,m_i\right)\right\}_{i=1}^N $ of intervals $ I_i = \left(a_i,b_i\right] $ where $ -\infty < a_i < b_i \leq \infty $ (called \textit{bars}) with multiplicities $ m_i\in\NN $. We view $ \mathcal{B} $ as a multi-set of intervals where each $ I_i $ appears $ m_i $ times.
	\end{definition}
	In view of Theorem \ref{thm:normalform} we may associate a unique barcode $ \mathcal{B}\left(V\right) $ to every (isomorphism class of) p.m $ V $, called the barcode of $ V $. Moreover, the space of barcodes (with the same amount of infinite bars) carries a metric called the \textit{bottleneck distance} and denoted by $ \dbot $ (see \cite{Pol19} for the exact definition) which makes the space of barcodes with the same amount of infinite bars a metric space and the normal form theorem gives a map $ V \mapsto \mathcal{B}\left(V\right) $
	\begin{equation*}
		\left(\left\{\text{p.m $ V $ with $ \dim V_\infty = n $}\right\},\dint\right) \mapsto \left(\left\{\text{barcode $ \mathcal{B} $ with $ n $ inifinite bars}\right\},\dbot\right),
	\end{equation*}
	furthermore, this map is in fact an isometry:
	\begin{theorem}[Isometry theorem] \label{thm:isometrythm}
		For any two p.m $ V,W $ with $ V_\infty \cong W_\infty $, the barcodes $ \mathcal{B}\left(V\right),\mathcal{B}\left(W\right) $ satisfy
		\begin{equation*}
			\dint\left(V,W\right) = \dbot\left(\mathcal{B}\left(V\right),\mathcal{B}\left(W\right)\right).
		\end{equation*}
	\end{theorem}
	In our context, we use persistence modules to define the change of $G_\ast^{\left(a,b\right]}\left(\varphi\right) $ as a function of $ b $, we then approximate it using finite samples and using the language of barcodes with bottleneck distance we can quantify the error of our approximation.

	\paragraph[gen' func' barcode]{}
	Recall the groups $ G_\ast^{\left(-\infty,b\right]}\left(\varphi\right) $ from \Cref{eq:G^b}. We wish to define a p.m associated to a Hamiltonian diffeomorphism $ \varphi $ like we did in \Cref{example:morse p.m} which measures how these groups change as a function of $ b $. Note that the homotopy type of $ E^{-\infty} $ is the same as of $ \SS^{2n}\times\left(\DD^{d-i}\times\partial\DD^i\right)  $, where $ \DD^k $ denotes the topological disc of dimension $ k $, and  the homotopy type of $ E = \SS^{2n}\times\RR^d $ (which is the same as $ \SS^{2n}\times\DD^d $) is obtained by gluing $ \DD^i $ along $ \partial\DD^i $, so when $ b = \infty $ (i.e, when $ E^b = E $) we get
	\begin{equation*}
		H_{\ast+i}\left(E,E^{-\infty}\right) = H_{\ast+i}\left(E^{-\infty}\cup_{\SS^{2n}\times\partial\DD^i}\left(\SS^{2n}\times\DD^i\right),E^{-\infty}\right) = H_{\ast+i}\left(\SS^{2n}\times\nicefrac{\DD^i}{\partial\DD^i}\right),
	\end{equation*}
	which by the K\"unneth formula (our homology groups are taken over the field $ \ZZ_2 $) is equal to $ H_{\ast}\left(\SS^{2n}\right) $. So, we see that when $ b\to\infty $ these groups recreate the original base manifold on which $ \varphi $ is defined since $ \varphi $ is compactly supported and we view it as a diffeomorphism of $ \SS^{2n} $ preserving the point at $ \infty $.
	\begin{definition} \label{def:barcodeofphi}
		Let $ \varphi\in\Ham_c\left(\RRnn\right) $ such that all fixed points of $ \varphi $ are non-degenerate, the \textit{GF-barcode} $ \mathcal{B}\left(\varphi\right) $ is the barcode associated to the p.m $ V\left(\varphi\right) $ given by
		\begin{equation*}
			\begin{split}
				V_t\left(\varphi\right) &= G_\ast^{\left(-\infty,t\right]}\left(\varphi\right) = H_{\ast+i}\left(E^t,E^{-\infty};\ZZ_2\right) \\
				\pi_{s,t} &= in_*:V_s\left(\varphi\right)\to V_t\left(\varphi\right),\text{ where } \left(E^s,E^{-\infty}\right)\overset{in}{\hookrightarrow}\left(E^t,E^{-\infty}\right),
			\end{split}
		\end{equation*}
		defined for $ t\notin\Crit\left(S\right) $ and extended by semicontinuity for all $ t\in\RR $. Where $ E = \SS^{2n}\times\RR^d $, $ S:E\to\RR $ is a normalized GFQI of $ \varphi $ with quadratic index $ i $. 
	\end{definition}
	The barcode $ \mathcal{B}\left(\varphi\right) $ is thus our desired invariant which concludes the construction of the GF-Barcode.

	\section{The approximation algrorithm} \label{sec:approxalgo}
	In this section we prove \Cref{thm:introthm} by showing how the GF-Barcode of a Hamiltonian diffeomorphism $ \varphi\in\Ham_c\left(\RRnn,\omega_{0}\right) $ can be numerically approximated.
	\begin{definition*}
		Let $ S:\RRnn\to\RR $. A \textit{sample} of $ S $ with mesh parameter $ m\in\NN $ is the restriction $ S:{\frac{1}{m}\ZZ^{2n}}\to\RR $, which is considered as a finite set when $ S $ is compactly supported.
	\end{definition*}
	
	The following theorem is essentially the same as Theorem \ref{thm:introthm} but with slightly more details.
	\begin{theorem} \label{thm:mainthm}
		There exists an algorithm which takes as input samples of $ N $ functions $ S_1,\dots,S_N:\RRnn\to\RR $ with mesh parameter $ m\in\NN $, such that each $ S_j $ is a generating functions of $$ \varphi_j\in\mathcal{H}_{R,T}\left(\RRnn,\omega_{0}\right):= \left\{\varphi\in\Ham_c\left(\RRnn,\omega_{0}\right);\supp\left(\varphi\right)\subset\BB_{R},\left\|\varphi - \Id\right\|_{C^0}\leq T\right\} ,$$
		and returns a barcode $ \mathcal{B} $ such that
		\begin{equation*}
			\dbot\left(\mathcal{B},\mathcal{B}\left(\varphi_N\circ\dots\circ\varphi_1\right)\right) \leq C\frac{\sqrt{n}TN^2}{m} = \varepsilon,
		\end{equation*}
		where the constant $ C $ depends only on $ R $.
		
		Furthermore, for fixed $ n,N,T,R $ the time complexity bound of the algorithm (i.e a bound on the number of operations needed to compute all steps) is polynomial in $ \frac{1}{\varepsilon} $, and for fixed $ \varepsilon, n, T,R $ the time complexity grows as $ N^N $.
	\end{theorem}	
	 Given $ N $ samples of functions $ S_j:\RRnn\to\RR $ with mesh parameter $ m\in\NN $ (i.e the restriction $ S_j\vert_{\frac{1}{m}\ZZ^{2n}} $) such that $ S_j $ generates $ \varphi_j\in\mathcal{H}_{R,T}\left(\RRnn,\omega_{0}\right),$
	we lay out the steps needed to approximate the GF-Barcode of $ \varphi = \varphi_N\circ\dots\circ\varphi_1 $. Going over the definition of GF-Barcodes presented previously, we address the following issues in this section:
	\begin{itemize}
		\item Obtaining an explicit expression for an asymptotically quadratic at infinity generating function of $ \varphi $ by using a \textbf{composition formula} (see \Cref{prop:compformula}).
		\item Making the generating function into a GFQI in an explicit way and calculating its quadratic radii by using an \textbf{interpolation argument} (see \Cref{lemma:GFQI}).
		\item Showing that GF homology groups may be calculated from the restriction of the GFQI to a compact domain inside its quadratic radii using \textbf{Morse theoretic arguments} (see \Cref{lemma:compdomforG^b}).
		\item Constructing a cell complex approximating the domain of our GFQI and computing its boundary matrices using standard algebraic topology tools.
		\item Filtering the cell complex according to the samples of our GFQI on its $ 0 $-skeleton and permuting the boundary matrices according to the filtration.
		\item Extracting a barcode out of the boundary matrices using a \textbf{matrix reduction process} and bounding its bottleneck distance from the desired barcode.
	\end{itemize}
	
	\subsection{Composition formula} \label{sec:compositionformula}
	We start by showing how one can explicitly obtain a generating function asymptotically quadratic at infinity using generating functions of $C^1$-small Hamiltonian diffeomorphisms in $ \mathcal{H}_{R,T} $. We use Proposition \ref{prop:compformula} to get a generating function $ \#_{j=1}^N S_j $ for $ L_\varphi $ by induction and get the following expression.
	\begin{equation} \label{eq:compgenfunc}
		\begin{split}
			&\#_{j=1}^N S_j : \RRnn\times\left(\underbrace{\RRnn\times\dots\times\RRnn}_{N-1}\right) \to \RR \\
			&\#_{j=1}^N S_j\left(q_N,p_0;q_1,p_1,\dots,q_{N-1},p_{N-1}\right) = \sum_{j=1}^N S_j\left(q_j,p_{j-1}\right) + \sum_{j=2}^N \left\langle p_{j-1}-p_0,q_j - q_{j-1}\right\rangle.
		\end{split}
	\end{equation}
	Where we denote $ \varphi_j\left(q_{j-1},p_{j-1}\right) = \left(q_j,p_j\right) $ for $ j = 1,\dots,N $.
	
	\paragraph[asymp' quad' at inf']{}
	We consider the fiber preserving diffeomorphism $ \Phi_1:\RRnn\times\RR^d\to\RRnn\times\RR^d $ given by
	\begin{equation*}
		\begin{split}
			&\Phi_1\left(q_N,p_0;q_1,p_1,\dots,q_{N-1},p_{N-1}\right) = 	\left(q_N,p_0;\widetilde{q}_1,\widetilde{p}_1,\dots,\widetilde{q}_{N-1},\widetilde{p}_{N-1}\right)\\
			&\widetilde{q}_j=\begin{cases}
				q_N - q_{n-1} - \dots - q_j &  \left(1\leq j\leq N-1\right) \\
				q_N & j = N
			\end{cases}\qquad \widetilde{p}_j=\begin{cases}
				p_j + p_0  &  \left(1\leq j\leq N-1\right) \\
				p_0 & j = 0
			\end{cases}
		\end{split}
	\end{equation*}
	then $ \#_{j=1}^NS_j\circ\Phi_1 $ still generates $ L_\varphi $ and is given by
	\begin{equation*}
		\#_{j=1}^NS_j\circ\Phi_1\left(q_N,p_0;q_1,p_1,\dots,q_{N-1},p_{N-1}\right) = 	\sum_{j=1}^{N}S_j\left(\widetilde{q}_j,\widetilde{p}_{j-1}\right) + \sum_{j=1}^{N-1}\left\langle p_j,q_j\right\rangle.
	\end{equation*}
	For later convenience, we denote by $ \Phi_2:\RRnn\times\RR^d\to\RRnn\times\RR^d $ a second fiber preserving diffeomorphism given by
	\begin{equation*}
		\begin{split}
			&\Phi_2\left(q_N,p_0;\xi_1^-\dots,\xi_{N-1}^-,\xi_1^+,\dots,\xi_{N-1}^+\right) = 	\left(q_N,p_0;q_1,p_1,\dots,q_{N-1},p_{N-1}\right)\\
			&q_j=-\xi^-_j + \xi^+_j \qquad p_j=\xi^-_j + \xi^+_j \qquad\left(1\leq j\leq N-1\right).
		\end{split}
	\end{equation*}
	We obtain a generating function $ S^\prime = \#_{j=1}^NS_j\circ\Phi_1\circ\Phi_2 $ given by
	\begin{equation} \label{eq:asympquadgenfunc}
		S^\prime\left(q_N,p_0;\xi\right) = S^\prime\left(q_N,p_0;\xi_1^-\dots,\xi_{N-1}^-,\xi_1^+,\dots,\xi_{N-1}^+\right) = \sum_{j=1}^NS_j\left(\widetilde{q}_j,\widetilde{p}_{j-1}\right)+\mathcal{Q}\left(\xi\right),
	\end{equation}
	where
	\begin{equation*}
		\begin{split}
			\widetilde{q}_j = q_N+\sum_{k=j}^N\xi^-_k-\sum_{k=j}^N\xi^+_k, &\qquad \widetilde{p}_j = p_0+\xi^-_{j}+\xi^+_{j} \qquad \left(\xi^-_N = \xi^+_N = \xi^-_0 = \xi^+_0 = 0\right) \\
			\mathcal{Q}\left(\xi\right) &= -\sum_{j=1}^{N-1}\left(\xi^-_j\right)^2 +\sum_{j=1}^{N-1}\left(\xi^+_j\right)^2.
		\end{split}
	\end{equation*}
	and $ S^\prime $ still generates $ L_\varphi $. It follows that the partial derivatives of $ S^\prime-\mathcal{Q} $ in the fiber direction are
	\begin{equation*}
		\begin{split}
			\frac{\partial}{\partial \xi_i^-}\left(S^\prime-\mathcal{Q}\right) &= \sum_{j=1}^i\frac{\partial S_j}{\partial Q}\left(\widetilde{q}_j,\widetilde{p}_{j-1}\right) + \frac{\partial S_{i+1}}{\partial p}\left(\widetilde{q}_{i+1},\widetilde{p}_{i}\right) \\
			\frac{\partial}{\partial \xi_i^+}\left(S^\prime-\mathcal{Q}\right) &= -\sum_{j=1}^i\frac{\partial S_j}{\partial 	Q}\left(\widetilde{q}_j,\widetilde{p}_{j-1}\right) + \frac{\partial S_{i+1}}{\partial p}\left(\widetilde{q}_{i+1},\widetilde{p}_{i}\right),
		\end{split}
	\end{equation*}
	so we get a bound on $ \left|\frac{\partial}{\partial\xi}\left(S^\prime - \mathcal{Q}\right)\right| $ by
	\begin{equation*}
		\begin{split}
			\left|\frac{\partial}{\partial\xi}\left(S^\prime - \mathcal{Q}\right)\right| &\leq 	\sum_{j=1}^{N-1}\sqrt{2\left(N-j\right)}\left|\frac{\partial S_j}{\partial Q}\left(\widetilde{q}_j,\widetilde{p}_{j-1}\right)\right| + \sqrt{2}\sum_{j=2}^N\left|\frac{\partial S_j}{\partial p}\left(\widetilde{q}_j,\widetilde{p}_{j-1}\right)\right| \\
			&\leq \sqrt{2}\left(\sum_{j=1}^{N-1}\sqrt{N-j}\left|\nabla S_j\right| + \left|\nabla S_{j+1}\right|\right).
		\end{split}
	\end{equation*}
	Finally, Since all $ S_j $ generates $ \varphi_j\in\mathcal{H}_{R,T} $ we know $ \left|\nabla S_j\right| = \left|\left(P_{\varphi_j} - p,q - Q_{\varphi_j}\right)\right| < \left\|\varphi_j - \Id\right\|_{C^0} \leq T $ which gives the uniform bound (over $ \RRnn $)
	\begin{equation*}
		\left|\frac{\partial}{\partial\xi}\left(S^\prime - \mathcal{Q}\right)\right| < 	\sqrt{2}\left(\sum_{j=1}^{N-1}\sqrt{N-j}\left\|\varphi_j - \Id\right\|_{C^0} + \left\|\varphi_{j+1} - \Id\right\|_{C^0}\right) < \sqrt{2}T\left(\left(N-1\right)+\sum_{j=1}^{N-1}\sqrt{j}\right).
	\end{equation*}
	So $ S^\prime $ is indeed asymptotically quadratic at infinity.

	\subsection{Obtaining a GFQI} \label{sec:obtainingGFQI}
	\paragraph[obtaining GFQI]{}
	After obtaining an expression for $ S^\prime $ which is asymptotically quadratic at infinity, we wish to obtain an explicit expression for a GFQI out of $ S^\prime $.
	
	We use the following argument due to Viterbo (see \cite{Vit11}), essentially, we explicitly find a radius (in the fiber direction) outside of which there are no fiber-critical points and then interpolates the original generating function with the quadratic form in a way that does not create any new fiber-critical points, thus preserving the generated Lagrangian.
	\begin{lemma} \label{lemma:GFQI}
		Let $ S^\prime:\RRnn\times\RR^d\to\RR $ be an asymptotically quadratic at infinity generating function and for each $ x\in\RRnn $ let $ R_x > M_x\left\|A_{\mathcal{Q}}^{-1}\right\|_{\text{op}} $ where $ A_{\mathcal{Q}} $ is the matrix such that $ \mathcal{Q}\left(x,\xi\right) = \frac{1}{2}\left\langle A_{\mathcal{Q}}\xi,\xi \right\rangle $ and $ M_x $ is the bound of $ \left|\frac{\partial}{\partial \xi}\left(S^\prime-\mathcal{Q}\right)\right| $, then the function $ S^\prime\left(x,\xi\right) $ has no fiber-critical points outside a ball of radius $ R_x $ in the fiber $ \left\{x\right\}\times\RR^d $. Furthermore, if $ \mathcal{Q} $ is independent of the base variable, $ \left|\frac{\partial}{\partial \xi}\left(S^\prime-\mathcal{Q}\right)\right| $ is bounded uniformly over $ \RRnn $ and $ \sup_{x\in\RRnn}\left|S^\prime\left(x,0\right) \right|< C_0 $, then the function
		\begin{equation} \label{eq:defGFQI}
			S\left(x,\xi\right) = \rho\left(\left|\xi\right|\right)S^\prime\left(x,\xi\right) + \left(1-\rho\left(\left|\xi\right|\right)\right)\mathcal{Q}\left(\xi\right),
		\end{equation}
		generates the same Lagrangian as $ S^\prime $, where $ \rho $ is a smooth step function with sufficiently small derivative, and $ S\equiv\mathcal{Q} $ outside a ball of radius $ \frac{C_0}{M} + 4M\left\|A_\mathcal{Q}^{-1}\right\|_{\text{op}} $ in the fiber direction.
	\end{lemma}
	
	\begin{proof}
		Let $ S^\prime:\RRnn\times\RR^d\to\RR $ be an asymptotically quadratic at infinity generating function with quadratic part $ \mathcal{Q} $, the first part of the lemma is proven by the following claim.
		\begin{claim}
			For each $ x\in\RRnn $, Let $ R_x > M_x\left\|A_\mathcal{Q}^{-1}\right\|_{\text{op}} $, where $ A_\mathcal{Q} $ is the matrix such that $ \mathcal{Q}\left(x,\xi\right)=\frac{1}{2}\left\langle A_\mathcal{Q}\xi,\xi\right\rangle $ and $ M_x > 0 $ is the bound on $ \left|\frac{\partial}{\partial \xi}\left(S^\prime\left(x,\xi\right)-\mathcal{Q}\left(x,\xi\right)\right)\right| $. Then
			\begin{equation*}
				 \left|\xi\right|>R_x \implies \frac{\partial S^\prime}{\partial\xi}\left(x,\xi\right) \neq 0.
			\end{equation*}
		\end{claim}
		\begin{proof}[Proof of claim]
			Let $ \xi\in\RR^d $ such that $ \left|\xi\right| \geq R_x $, then we have
			\begin{equation*}
				\begin{split}
					\left|\frac{\partial S^\prime}{\partial\xi}\left(x,\xi\right)\right| &= \left|\frac{\partial\mathcal{Q}}{\partial\xi}\left(x,\xi\right)+\frac{\partial}{\partial\xi}\left(S^\prime\left(x,\xi\right)-\mathcal{Q}\left(x,\xi\right)\right)\right| \\
					&\geq \left|A_\mathcal{Q}\xi\right| - \left|\frac{\partial}{\partial\xi}\left(S^\prime\left(x,\xi\right)-\mathcal{Q}\left(x,\xi\right)\right)\right| \geq \left\|A_\mathcal{Q}^{-1}\right\|_{\text{op}}^{-1}\left|\xi\right| - M_x > 0,
				\end{split}
			\end{equation*}
			since $ \left\|A_\mathcal{Q}^{-1}\right\|_{\text{op}}^{-1}\left|\xi\right| \geq \left\|A_\mathcal{Q}^{-1}\right\|_{\text{op}}^{-1}R_x > M_x $.
		\end{proof}
		
		Next, assuming $ \mathcal{Q} $ is independent of the base variable and $ \sup_{x\in\RRnn}M_x < M $, we set $ R = M\left\|A_\mathcal{Q}^{-1}\right\|_{\text{op}} $ then by the first part of the lemma we know that $ S^\prime $ does not have any fiber critical points (in all fibers) outside a ball of radius $ R $. Further assuming $ \sup_{x\in\RRnn} \left|S^\prime\left(x,0\right)\right| < C_0 $ we denote
		\begin{equation*}
			R_f^- = \frac{C_0}{M} + 2R \qquad R_f = R_f^- + 2R = \frac{C_0}{M} + 4M\left\|A_\mathcal{Q}^{-1}\right\|_{\text{op}},
		\end{equation*}
		and take a smooth function $ \rho:\left[0,\infty\right)\to\RR $ such that
		\begin{equation*}
			\rho\left(r\right) =
			\begin{cases}
				1 & r\in\left[0,R_f^-\right] \\
				0 & r\in\left[R_f,\infty\right)
			\end{cases} \qquad 0\geq\rho^\prime\geq-\varepsilon,
		\end{equation*}
		where $ \varepsilon = \frac{1}{2R-\delta} $ for $ \delta < 2\left(R-M\left\|A_{\mathcal{Q}}^{-1}\right\|_{\text{op}}\right) $. Then we consider the function
		\begin{equation*}
			S\left(x,\xi\right) = \rho\left(\left|\xi\right|\right)S^\prime\left(x,\xi\right) + 	\left(1-\rho\left(\left|\xi\right|\right)\right)\mathcal{Q}\left(\xi\right),
		\end{equation*}
		which clearly identifies with $ \mathcal{Q} $ outside of a ball of radius $ R_f $ in the fiber direction.
		\begin{claim}
			$\frac{\partial S}{\partial \xi}\left(x,\xi\right) = 0 \implies \left|\xi\right|<R_f^- $.
		\end{claim}
		\begin{proof}[Proof of claim]
			Let $ \left(x,\xi\right)\in\RRnn\times\RR^d $ such that $ \frac{\partial S}{\partial \xi}\left(x,\xi\right) = 0 $, then we have
			\begin{equation*}
				0 = \frac{\partial S}{\partial \xi}\left(x,\xi\right) = \rho^\prime\left(\left|\xi\right|\right)\frac{\xi}{\left|\xi\right|}\left(S^\prime\left(x,\xi\right)-\mathcal{Q}\left(\xi\right)\right) + \rho\left(\left|\xi\right|\right)\frac{\partial}{\partial\xi}\left(S^\prime\left(x,\xi\right)-\mathcal{Q}\left(\xi\right)\right) + d\mathcal{Q}\left(\xi\right).
			\end{equation*}
			Note that $ \left|S^\prime\left(x,\xi\right)-\mathcal{Q}\left(\xi\right)\right| \leq \left|S^\prime\left(x,0\right) + M\left|\xi\right|\right|  \leq C_0+M\left|\xi\right|$, then $ d\mathcal{Q}\left(\xi\right) = A_\mathcal{Q}\xi $ satisfies
			\begin{equation*}
				\begin{split}
					\left\|A_\mathcal{Q}^{-1}\right\|_{\text{op}}^{-1}\left|\xi\right| &\leq \left|A_\mathcal{Q}\xi\right| = \left|-\rho^\prime\left(\left|\xi\right|\right)\frac{\xi}{\left|\xi\right|}\left(S^\prime\left(x,\xi_0\right)-\mathcal{Q}\left(\xi\right)\right) - \rho\left(\left|\xi\right|\right)\frac{\partial}{\partial\xi}\left(S^\prime\left(x,\xi\right)-\mathcal{Q}\left(\xi\right)\right)\right| \\
					&\leq \varepsilon\left|S^\prime\left(x,\xi\right)-\mathcal{Q}\left(\xi\right)\right| + \left|\frac{\partial}{\partial\xi}\left(S^\prime\left(x,\xi\right)-\mathcal{Q}\left(\xi\right)\right)\right| \leq \varepsilon\left(C_0+M\left|\xi\right|\right) + M,
				\end{split}
			\end{equation*}
			and so we indeed get
			\begin{equation*}
				\left|\xi\right| \leq \frac{\varepsilon C_0 + M}{\left\|A_\mathcal{Q}^{-1}\right\|_{\text{op}}^{-1}-\varepsilon M} = \frac{\frac{C_0}{2R-\delta}+M}{\left\|A_\mathcal{Q}^{-1}\right\|_{\text{op}}^{-1} - \frac{M}{2R-\delta}} \leq \frac{\frac{C_0}{2M\left\|A_\mathcal{Q}^{-1}\right\|_{\text{op}}}+M}{\left\|A_\mathcal{Q}^{-1}\right\|_{\text{op}}^{-1} - \frac{M}{2M\left\|A_\mathcal{Q}^{-1}\right\|_{\text{op}}}} = \frac{C_0}{M}+2M\left\|A_\mathcal{Q}^{-1}\right\|_{\text{op}} < R_f^-.
			\end{equation*}
		\end{proof}
		Finally, this claim shows that $ S $ identifies with $ S^\prime $ on the set of fiber critical points $ \Sigma_{S} = \Sigma_{S^\prime} $, so they generate the same Lagrangian which concludes the proof.
	\end{proof}
	
	\paragraph[calc quad radii]{}
	In our context, we denote
	\begin{equation*}
		M = \sqrt{2}T\left(\left(N-1\right)+\sum_{j=1}^{N-1}\sqrt{j}\right),
	\end{equation*}
	then $ M $ is a uniform bound on $ \left|\frac{\partial}{\partial\xi}\left(S^\prime - \mathcal{Q}\right)\right| $, $ \mathcal{Q} $ is independent of the base variable, and since all $ S_j $ are compactly supported inside a ball of radius $ R $ and satisfy $ \left|\nabla S_j\right|\leq T $ we have
	\begin{equation*}
		\sup_{\left(q_N,p_0\right)\in\RRnn}\left|S^\prime\left(q_N,p_0,0\right)\right| = 	\sup_{\left(q_N,p_0\right)\in\RRnn}\left|\sum_{j=1}^NS_j\left(q_N,p_0\right)\right| \leq \sum_{j=1}^N\max \left|S_j\right| \leq N\cdot T\cdot R = C_0,
	\end{equation*}
	so we denote
	\begin{equation*}
		R_f^- = \frac{C_0}{M} + M \qquad R_f^+ = \frac{C_0}{M} + 2M,
	\end{equation*}
	and define $ \rho:\RR_{\geq0}\to\RR_{\geq0} $ to be a smooth function which is equal to
	\begin{equation*}
		\rho\left(r\right) =
		\begin{cases}
			1 & 0\leq r\leq R_f^- \\
			\frac{R_f^+-r}{M} & R_f^-\leq r \leq R_f^+\\
			0 & R_f \leq r
		\end{cases},
	\end{equation*}
	outside a small neighborhood of $ R_f^- $ and $ R_f^+ $ thus \Cref{lemma:GFQI} implies that the function $ S $ given by \Cref{eq:defGFQI} is a GFQI which identifies with $ \mathcal{Q} $ outside a ball of radius $ R_f^+ $ in the fiber direction. Finally, we denote
	\begin{equation*}
		R_b = \sqrt{2n}\left(R+T\right) + 	2\sqrt{n\left(N-1\right)}R_f^+,
	\end{equation*}
	and claim that $ S = \mathcal{Q} $ outside a ball of radius $ R_b $. indeed, for each $ j $, if we denote $ \varphi_j\left(q,p\right) = \left(Q,P\right) $ then $ \left|\left(q,p\right)-\left(Q,p\right)\right| = \left|q-Q\right| \leq T $ and we get
	\begin{equation*}
		\left|\left(Q,p\right)\right|\geq R+T \implies \left|\left(q,p\right)\right| \geq \left|\left(Q,p\right)\right| - T \geq R,
	\end{equation*}
	so since $ R $ contains the compact support of $ \varphi_j $ we see that $ S_j $ vanishes outside a ball of radius $ R+T $. It follows that for $ \left(q_N,p_0,\xi\right) $ such that $ \left|\xi\right| = \left|\left(\xi_1^-,\dots,\xi_{N-1}^-,\xi_1^+\dots,\xi_{N-1}^+\right)\right|\leq R_f^+ $ and $ \left|\left(q_n,p_0\right)\right|\geq R_b $ we have the following two options:
	\begin{enumerate}
		\item $ \left|p_0\right| = \left|\left(\left(p_0\right)_1,\dots,\left(p_0\right)_n\right)\right| \geq \frac{R_b}{\sqrt{2}} $, which means that there exists a coordinate $ 1\leq l \leq n  $ such that $ \left|\left(p_0\right)_l\right|\geq \frac{R_b}{\sqrt{2n}} $ therefore for all $ j=1,\dots,N $
		\begin{equation*}
			\begin{split}
				\left|\left(p_0\right)_l+\left(\xi_{j-1}^-\right)_l + \left(\xi_{j-1}^+\right)_l\right| &\geq \left|\left(p_0\right)_l\right| - 	\max_{\left|\xi\right|\leq R_f^+}\left|\left(\xi^-_{j-1}\right)_l + \left(\xi^+_{j-1}\right)_l\right| \\
				&\geq  \frac{R_b}{\sqrt{2n}} - \sqrt{2}R_f^+ \geq  R+T.
			\end{split}
		\end{equation*}
		
		\item $ \left|q_N\right| = \left|\left(\left(q_N\right)_1,\dots,\left(q_N\right)_n\right)\right| \geq \frac{R_b}{\sqrt{2}} $, which means that there exists a coordinate $ 1\leq l \leq n $ such that $ \left|\left(q_N\right)_l\right|\geq \frac{R_b}{\sqrt{2n}} $ therefore for all $ j=1,\dots,N $
		\begin{equation*}
			\begin{split}
				\left|\left(q_N\right)_l+\sum_{k=j}^N\left(\xi^-_k\right)_l-\sum_{k=j}^N\left(\xi^+_k\right)_l\right| &\geq 	\left|\left(q_N\right)_l\right| - \max_{\left|\xi\right|\leq R_f^+}\left|\sum_{k=1}^{N-1}\left(\xi^-_k\right)_l-\left(\xi^+_k\right)_l\right| \\
				&\geq \frac{R_b}{\sqrt{2n}} - \sqrt{2\left(N-1\right)}R_f^+ \geq R+T.
			\end{split}
		\end{equation*} 
	\end{enumerate}
	So, in both cases we see that $ S_j\left(q_N+\sum_{k=j}^N\xi^-_k-\sum_{k=j}^N\xi^+_k,p_0+\xi^-_{j-1}+\xi^+_{j-1}\right) = 0 $ for all $ j $ and thus $ S^\prime\left(q_N,p_0,\xi\right) = \mathcal{Q}\left(\xi\right) $. This shows that outside $ \BB_{R_b} $, the function $ S $ becomes just $ \mathcal{Q} $ and therefore it is indeed a GFQI with quadratic radii $ R_b,R_f^+ $, in order to make sure the fiber radius also satisfies \Cref{eq:Rflarge} we set
	\begin{equation*}
		R_f = \sqrt{\left(R_f^+\right)^2-\sum_{j=1}^N\min S_j},
	\end{equation*}
	this concludes the construction of a GFQI $ S $ of $ \varphi = \varphi_N\circ\dots\circ\varphi_1 $ that has quadratic radii $ R_b,R_f $ explicitly given as functions of $ N,T,R $ and $ \min S_j $.

	\subsection{GF homology on a compact domain} \label{sec:GFcompdomain}
	Our next goal is to show that $ G_\ast^{\left(a,b\right]}\left(\varphi\right) $ is equal to the relative homology of sublevel sets of $ S $ restricted to a compact set, so we may approximate it numerically. This is done by following the classical construction described by Milnor (see \cite[section 3]{Mil68}) which relates the homotopy type of sublevel sets near a critical level of a compact manifold to the Morse index of the function, then in our case (which  is not compact), considering the compact domain where $ S $ is not quadratic as a kind of "critical set" and relating the homotopy type of sublevel sets to those of the part which intersects the compact domain.
	
	\paragraph[deformation on a fiber]{}
	Recall that our GFQI $ S $ is defined on a vector bundle $ p:E = \SS^{2n}\times \RR^d \to \SS^{2n} $, then for $ x\in\SS^{2n} $, we examine a single fiber $ p^{-1}\left(x\right) = \left\{x\right\}\times\RR^d $ and introduce the following notations:
	\begin{align}
		E^c_x &= E^c\cap p^{-1}\left(x\right) = \left\{\xi\in\RR^d;S\left(x,\xi\right)\leq c\right\} \subset p^{-1}\left(x\right) \nonumber \\
		D^i &= \left\{\left(\xi^-,0\right)\in\RR^d;\left|\xi^-\right|\leq R_f\right\} \subset \BB_{R_f} \label{eq:defD^i}.
	\end{align}
	Then, following similar arguments as in \cite[section 3]{Mil68}, we show that $ E^{-\infty}_x\cup \left(E^b_x\cap\BB_{R_f}\right) $ is a deformation retract of $ E^b_x $ and use excision to remove the set $ E^{-\infty}_x\setminus\partial D^i $ which is a deformation retract of an excisable set, thus getting
	\begin{equation} \label{eq:defonfiber}
		\begin{split}
			H_*\left(E^b_x,E^{-\infty}_x\right) &= H_*\left(E^{-\infty}_x\cup \left(E^b_x\cap\BB_{R_f}\right), E^{-\infty}_x\right) \\
			&= H_*\left(\partial D^i\cup \left(E^b_x\cap\BB_{R_f}\right),\partial D^i\right) \\
			&= H_*\left(E^b_x\cap\BB_{R_f},\partial D^i\right).
		\end{split}	
	\end{equation}
	
	\paragraph[deformation on the bundle]{}
	Finally, \Cref{eq:defonfiber} implies that all the "interesting" data required to calculate $ H_*\left(E^b_x,E^{-\infty}_x\right) $ is contained in a compact set and we also show that these deformations always preserve $ \BB_{R_f} $. So, since outside $ \BB_{R_f} $ the level sets of $ S $ are all the same we may deform all fibers simultaneously which leads us to the desired result.
	\begin{lemma} \label{lemma:compdomforG^b}
		Let $ S:\SS^{2n}\times\RR^d\to\RR $ be a diagonalized GFQI for $ \varphi\in\Ham_c\left(\RRnn,\omega_0\right) $ with quadratic radii $ R_b,R_f > 0 $ and quadratic index $ i $. Then for large enough $ R_f $ (i.e, under the assumption in \Cref{eq:Rflarge}), we have
		\begin{equation*}
			G^{\left(-\infty,b\right]}_{\ast}\left(\varphi\right) = H_{\ast+i}\left(E^b,E^{-\infty}\right)  = H_{\ast+i}\left(E^b\cap\left(\SS^{2n}\times\BB_{R_f}\right),\SS^{2n}\times\partial D^i\right),
		\end{equation*}
		where $ E $ and $ E^c $ are as in \Cref{def:GFhomology}.
	\end{lemma}
	
	\begin{proof}
		Recall the following notations:
		\begin{equation*}
			\begin{split}
				&\SS^{2n} = \RRnn\cup\left\{\infty\right\}\text{ and } E=\SS^{2n}\times\RR^d, \\
				&S:E\to\RR\text{ is a diagonalized GFQI with quadratic radii } R_b,R_f\text{ and quadratic part }\mathcal{Q}, \\
				&\mathcal{Q}:\RR^d\to\RR\text{ such that }\mathcal{Q}\left(\xi\right) = -\xi_1^--\dots-\xi_i^-+\xi_{i+1}^++\dots+\xi_d^+, \\
				&E^c = \left\{\left(x,\xi\right)\in E; S\left(x,\xi\right) \leq c\right\}\text{ and }E^{-\infty} = E^a\text{ for } a \leq -R_f^2, \\
				&E^c_x = E^c\cap p^{-1}\left(x\right) = \left\{\xi\in\RR^d;S\left(x,\xi\right)\leq c\right\} \subset p^{-1}\left(x\right), \\
				&D^i = \left\{\left(\xi^-,0\right)\in\RR^d;\left|\xi^-\right|\leq R_f\right\} \subset \BB_{R_f}.
			\end{split}
		\end{equation*}
		Then, let $ -R_f^2<b\notin\Crit\left(S\right) $ and set $ a=-R_f^2 $, we start with a single fiber $ p^{-1}\left(x\right) = \left\{x\right\}\times\RR^d $ and show that the set $ E^{-\infty}_x\cup \left(E^b_x\cap\BB_{R_f}\right) $ is a deformation retract of $ E^b_x $ by considering a "bump" function $ \mu:\RR_{\geq0}\to\RR_{\geq0} $ such that
		\begin{equation*}
			\begin{cases}
				\mu\left(r\right) = \mu_0 > 2R_f^2 & r < R_f^2 \\
				\mu\left(r\right) = 0 & r\geq 5R_f^2,
			\end{cases}
		\end{equation*}
		and $ -1 < \mu^\prime\left(r\right) \leq  0 $. We define a function $ F:\RR^d\to\RR $ by $ F\left(\xi\right) = \mathcal{Q}\left(\xi\right) - \mu\left(\left|\xi\right|^2\right) $ and denote its sub-level sets by $ F^c = \left\{\xi ; F\left(\xi\right) \leq c\right\} $. Now, the result follows from a series of claims:
		\begin{claim}
			$ E_x^b \subset F^b $ and $ \BB_{R_f} \subset F^a $.
		\end{claim}
		\begin{proof}[Proof of claim]
			From the definition of $ F $ it follows that $ F\leq S $ and thus $ E^b \subset F^b $ which proves the first part. For the second part, If $ \xi\in \BB_{R_f} $ then $ \mathcal{Q}\left(\xi\right) \leq R_f^2 $ and $ \mu\left(\left|\xi\right|^2\right) > 2R_f^2 $ so
			\begin{equation*}
				F\left(\xi\right) = \mathcal{Q}\left(\xi\right) - \mu\left(\left|\xi\right|\right) < R_f^2 - 2R_f^2 = -R_f^2 = a,
			\end{equation*}
			thus $ \xi \in F^a $ as claimed.
		\end{proof}
		\begin{claim}
			$ F $ has no critical points in $ F^{-1}\left(\left[a,b\right]\right) $.
		\end{claim}
		\begin{proof}[Proof of claim]
			Differentiating $ F $ gives
			\begin{equation*}
				\begin{split}
					\frac{\partial F}{\partial\left(\xi_1^-,\dots \xi_i^-\right)} &= 2\left(\xi_1^-,\dots 	\xi_i^-\right)\left(-1-\mu^\prime\left(\left|\xi\right|^2\right)\right) \\
					\frac{\partial F}{\partial\left(\xi_{i+1}^+,\dots \xi_d^+\right)} &= 2\left(\xi_{i+1}^+,\dots 	\xi_d^+\right)\left(1-\mu^\prime\left(\left|\xi\right|^2\right)\right),
				\end{split}
			\end{equation*}
			and since $ -1-\mu^\prime\left(\left|\xi\right|^2\right) < 0 $ and $ 1-\mu^\prime\left(\left|\xi\right|^2\right) \geq 1 $ we see that $ F $ has a critical point only at the origin which as we saw in the last claim satisfies $ F\left(0\right) < a $ and so $ F $ has no critical points in $ F^{-1}\left(\left[a,b\right]\right) $.
		\end{proof}
		\begin{claim}
			$ F^a $ is a deformation retract of $ F^b $.
		\end{claim}
		\begin{proof}[Proof of claim]
			Following the same argument as in the proof of \Cref{lemma:deformrelativesmallvalues} (which are basically the classical Morse theory arguments as in \cite[section 3]{Mil68}), the set $ F^{-1}\left(\left[a,b\right]\right) $ has no critical points and $ F = \mathcal{Q} $ outside a compact set, so we consider the compactification $ \RR^d\cup\left\{\infty\right\} $ and take a compact neighborhood $ K $ of $ F^{-1}\left(\left[a,b\right]\right)\cup\left\{\infty\right\} $ (which is now compact), then the vector field $ \rho\nabla F $ for a smooth function $ \rho $ that vanish outside $ K $ and is equal to $ \frac{1}{\left|\nabla F\right|^2} $ on $ F^{-1}\left(\left[a,b\right]\right) $ is equal to $ \frac{\nabla\mathcal{Q}}{\left|\nabla \mathcal{Q}\right|^2} $ near $ \infty $ so it can be smoothly extended to $ \infty $ by $ 0 $ and we get a smooth vector field that vanishes outside a compact set. This vector field then generates a flow $ \phi_t $ that satisfied $ \frac{\diff}{\diff t}\big\vert_{t_0}F\left(\phi_t\left(x,\xi\right)\right) = 1  $ as long as $ \phi_{t_0}\left(x,\xi\right)\in F^{-1}\left(\left[a,b\right]\right) $ and is stationary at $ \infty $, so $ \phi_{b-a}\left(F^a\right) = F^b $ and we define the map $ r_t:\left[0,1\right]:F^b\to F^b $ by
			\begin{equation*}
				r_t\left(\xi\right) = \begin{cases}
					\xi & F\left(\xi\right)\leq a \\
					\phi_{t\left(a-F\left(\xi\right)\right)} & a\leq F\left(x,\xi\right)\leq b,
				\end{cases}
			\end{equation*}
			which is a retraction of $ F^b $ onto $ F^a $ as claimed.
		\end{proof}
		\begin{claim}
			$ F^a\cap E_x^b $ is a deformation retract of $ E_x^b $.
		\end{claim}
		\begin{proof}[Proof of claim]
			We first note that the reverse flow $ \phi_{-t} $ from the previous claim takes $ F^b $ to $ F^a $ (when $ t\in\left[0,b-a\right] $) and it is generated by the vector field $ -\frac{\nabla F}{\left|\nabla F\right|^2} $ as long as $ \phi_{-t}\left(\xi\right)\in F^{-1}\left(\left[a,b\right]\right) $ so it also satisfies
			\begin{equation*}
				\begin{split}
					\frac{\diff}{\diff t}\big\vert_{t_0}\mathcal{Q}\left(\phi_{-t}\left(\xi\right)\right) &= 	d\mathcal{Q}_{\phi_{-t_0}\left(\xi\right)}\left(\frac{\diff}{\diff t}\big\vert_{t_0}\phi_{-t}\left(\xi\right)\right) = d\mathcal{Q}_{\phi_{-t_0}\left(\xi\right)}\left(-\frac{\nabla F\left(\phi_{-t_0}\left(\xi\right)\right)}{\left|\nabla F\left(\phi_{-t_0}\left(\xi\right)\right)\right|^2}\right) \\
					&=\frac{1}{\left|\nabla F\left(\phi_{-t_0}\left(\xi\right)\right)\right|^2}\left\langle\nabla\mathcal{Q}\left(\phi_{-t_0}\left(\xi\right)\right),-\nabla\mathcal{Q}\left(\phi_{-t_0}\left(\xi\right)\right)+2\mu^\prime\left(\left|\phi_{-t_0}\left(\xi\right)\right|^2\right)\phi_{-t_0}\left(\xi\right)\right\rangle \\
					&= \frac{-\left|\nabla \mathcal{Q}\left(\phi_{t_0}\left(\xi\right)\right)\right|^2 + 	2\mu^\prime\left(\left|\phi_{t_0}\left(\xi\right)\right|^2\right)\left\langle\nabla\mathcal{Q}\left(\phi_{t_0}\left(\xi\right)\right),\phi_{t_0}\left(\xi\right)\right\rangle}{\left|\nabla F\left(\phi_{-t_0}\left(\xi\right)\right)\right|^2}.
				\end{split}
			\end{equation*}
			Next, since $ \mathcal{Q} $ is a diagonalized non-degenerate quadratic form we also have $ 2\left|\phi_{t_0}\left(\xi\right)\right| = \left| \nabla\mathcal{Q}\left(\phi_{t_0}\left(\xi\right)\right)\right| $ so we conclude that
			\begin{equation*}
				\left|2\mu^\prime\left(\left|\phi_{t_0}\left(\xi\right)\right|^2\right)\left\langle\nabla\mathcal{Q}\left(\phi_{t_0}\left(\xi\right)\right),\phi_{t_0}\left(\xi\right)\right\rangle\right| \leq 2 \left|\nabla \mathcal{Q}\left(\phi_{t_0}\left(\xi\right)\right)\right|\left|\phi_{t_0}\left(\xi\right)\right| = \left|\nabla \mathcal{Q}\left(\phi_{t_0}\left(\xi\right)\right)\right|^2
			\end{equation*}
			which means that $ \frac{\diff}{\diff t}\big\vert_{t_0}\mathcal{Q}\left(\phi_t\left(\xi\right)\right) \leq 0 $ and thus the values of $ \mathcal{Q} $ decreases along the deformation from $ F^b $ to $ F^a $. Now we take the deformation $ \phi_t $ and restrict it to the set $ E_x^b $, since this flow is a deformation retract onto $ F^a $ which contains the ball $ \BB_{R_f} $ then we know that all points inside $ \BB_{R_f} $ are stationary and therefore the restriction to $ E_x^b $ only moves points where $ S = \mathcal{Q} $ so we may say that this flow also decreases $ S $. We deduce that for all $ t \geq 0 $, $ \phi_t\left(E_x^b\right) \subset E_x^b $. Finally, we see that $ \phi_1\left(E_x^b\right) \subset F^a $ and of course $ \phi_1\left(F^a\cap E_x^b\right) = F^a\cap E_x^b $ so we can say that $ F^a \cap E_x^b $ is a deformation retract of $ E_x^b $ as claimed.
		\end{proof}
		\begin{claim}
			$ E_x^a \cup \BB_{R_f} $ is a deformation retract of $ F^a $.
		\end{claim}
		\begin{proof}[Proof of claim]
			We define the following function $ s_t\left(\xi\right):\left[0,1\right]\times\RR^d\to\RR $ by
			\begin{equation*}
				s_t\left(\xi\right) = \begin{cases}
					\sqrt{\left(1-t\right) + t\cdot \frac{\left|R_f^2-\left|\xi^-\right|^2\right|}{\left|\xi^+\right|^2}} & R_f^2 - \left|\xi^+\right|^2 \leq \left|\xi^-\right|^2 \leq R_f^2+\left|\xi^+\right|^2 \\
					1 & \text{o.w}
				\end{cases},
			\end{equation*}
			and using $ s_t $ we define a retraction $ \psi_t\left(\xi\right):\left[0,1\right]\times\RR^d\rightarrow\RR^d $ by
			\begin{equation} \label{eq:deformationpsi}
				\psi_t\left(\xi\right) = \left(\xi^-,s_t\left(\xi\right)\xi^+\right),
			\end{equation}
			which satisfies the following properties.
			\begin{equation*}
				\begin{split}
					\psi_0 &= \Id \\
					\psi_1\left(\xi\right) &= \begin{cases}
						\xi & \xi\in E_x^a\cup \BB_{R_f} \\
						\xi & \xi\ \in F^a \cap \BB_{5R_f}^\complement = E_x^a \cap \BB_{5R_f}^\complement \\
						\left(\xi^-,\sqrt{\frac{R_f^2-\left|\xi^-\right|^2}{\left|\xi^+\right|^2}}\xi^+\right) & \xi\in F^a\cap\left(\BB_{5R_f}\setminus \	BB_{R_f}\right)\text{ and }\left|\xi^-\right|^2 \leq R_f^2 \\
						\left(\xi^-,\sqrt{\frac{\left|\xi^-\right|^2 - R_f^2}{\left|\xi^+\right|^2}}\xi^+\right) & \xi\in F^a\cap\left(\BB_{5R_f}\setminus \	BB_{R_f}\right)\text{ and } R_f^2 < \left|\xi^-\right|^2.
					\end{cases}
				\end{split}
			\end{equation*}
			Note that
			\begin{equation*}
				\begin{split}
					\left|\left(\xi^-,\sqrt{\frac{R_f^2-\left|\xi^-\right|^2}{\left|\xi^+\right|^2}}\xi^+\right)\right|^2 &= \left|\xi^-\right|^2+\frac{R_f^2-\left|\xi^-\right|^2}{\left|\xi^+\right|^2}\left|\xi^+\right|^2 = R_f^2 \\
					\mathcal{Q}\left(\xi^-,\sqrt{\frac{\left|\xi^-\right|^2-R_f^2}{\left|\xi^+\right|^2}}\xi^+\right) &= -\left|\xi^-\right|^2+\frac{\left|\xi^-\right|^2-R_f^2}{\left|\xi^+\right|^2}\left|\xi^+\right|^2 = -R_f^2 = a,
				\end{split}
			\end{equation*}
			thus $ \psi_1\left(F^a\right) = E_x^a\cup \BB_{R_f} $ as claimed.
		\end{proof}
		\begin{claim}
			$ E_x^a\cup \left(E_x^b\cap\BB_{R_f}\right) = \left(E_x^a \cup \BB_{R_f}\right) \cap E_x^b $ is a deformation retract of $ F^a\cap E_x^b $.
		\end{claim}
		\begin{proof}
			For the deformation $ \psi_t $ defined in \Cref{eq:deformationpsi} we have for all $ \xi\in\BB_{R_f} $ that $ \frac{\diff}{\diff t}\big\vert_{t_0}\mathcal{Q}\left(\psi_t\left(\xi\right)\right) = 0 $ since $ \psi_t = \Id $ in this region, and for $ \xi\notin \BB_{R_f} $ we have
			\begin{equation*}
				\frac{\diff}{\diff t}\big\vert_{t_0}\mathcal{Q}\left(\psi_t\left(\xi\right)\right) = \left|\xi^+\right|^2\cdot\frac{\diff}{\diff t}\big\vert_{t_0}s_t^2\left(\xi\right) = \left|\xi^+\right|^2\left(-1+\frac{\left|R_f^2-\left|\xi^-\right|^2\right|}{\left|\xi^+\right|^2}\right),
			\end{equation*}
			 so since in this case $ R_f^2 \leq\left|\xi\right|^2 = \left|\xi^-\right|^2+\left|\xi^+\right|^2 $ and we get
			\begin{equation*}
				\frac{\diff}{\diff t}\big\vert_{t_0}\mathcal{Q}\left(\psi_t\left(\xi\right)\right) = \left|\xi^+\right|^2\left(-1+\frac{\left|R_f^2-\left|\xi^-\right|^2\right|}{\left|\xi^+\right|^2}\right) \leq 0.
			\end{equation*}
			So, we see that $ \mathcal{Q} $ decreases along the deformation $ \psi_t $ and thus we can restrict it to $ F^a\cap E_x^b $ and get a deformation onto $ \left(E_x^a \cup \BB_{R_f}\right) \cap E_x^b = E^a\cup\left(E_x^b\cap\BB_{R_f}\right) $ as claimed.
		\end{proof}
	
	Combining these claims, we see that indeed $ E^{-\infty}_x\cup \left(E^b_x\cap\BB_{R_f}\right) = E^a\cup\left(E_x^b\cap\BB_{R_f}\right) $ is a deformation retract of $ F^a\cap E_x^b  $ which in turn is a deformation retract of $ E^b_x $. Then, as we mentioned in \Cref{eq:defonfiber}, we use excision to remove the set $ E_a^a \setminus \partial D^i $ which is a deformation retract of an excisable set thus getting
	\begin{equation*}
		\begin{split}
			H_*\left(E^b_x,E^{-\infty}_x\right) &= H_*\left(E^{-\infty}_x\cup \left(E^b_x\cap\BB_{R_f}\right), E^{-\infty}_x\right)\\
			&= H_*\left(\partial D^i\cup \left(E^b_x\cap\BB_{R_f}\right),\partial D^i\right) \\
			&= H_*\left(E^b_x\cap\BB_{R_f},\partial D^i\right).
		\end{split}	
	\end{equation*}

	We can now use all the deformations above for the entire bundle $ E $ since they all preserve $ \BB_{R_f} $ and outside this ball the sets $ E_x^b $ are the same for all $ x\in\SS^{2n} $ whenever $ b \geq a $ (of course we also have that the sets $ E_x^a $ is the same for all fibers since they are always outside $ \BB_{R_f} $ by \Cref{eq:Rflarge}), so we get
	\begin{equation*}
		G^{\left(-\infty,b\right]}_\ast\left(\varphi\right) = H_{\ast+i}\left(E^b,E^{-\infty}\right) = H_{\ast+i}\left(E^b\cap\left(\SS^{2n}\times\BB_{R_f}\right),\SS^{2n}\times \partial D^i\right),
	\end{equation*}
	which proves the lemma.
	\end{proof}
	
	With this Lemma we see that in order to calculate $ G^{\left(-\infty,b\right]}_{\ast}\left(\varphi\right) $, one only needs to know the function $ S $ on a compact set, which is then possible to approximate using a finite sample.

	\subsection{Step 1 - Constructing a cell complex} \label{sec:step1}
	Since $ S $ identifies with the quadratic form $ \mathcal{Q} $ outside $ \BB_{R_b} $ we may consider $ S $ as a function on $ \BB_{R_b}/\partial\BB_{R_b}\times\RR^d $ instead of our domain $ \left(\RRnn\cup\left\{\infty\right\}\right)\times\RR^d $. Then the idea is to start with two lattices in $ \RRnn $ and $ \RR^d $, define cubical complexes from these lattices which are "close" to $ \BB_{R_b} $ and $ \BB_{R_f} $, use algebraic tools to manipulate the cellular complexes into the desired quotient and product, so we get a finite cellular complex homotopic to $ \left(\BB_{R_b}/\partial\BB_{R_b}\times\BB_{R_f},\BB_{R_b}/\partial\BB_{R_b}\times \partial D^i\right) $ where $ \partial D^i \subset \partial\BB_{R_f} $ is as in \Cref{eq:defD^i} and this homotopy is also close ot the identity.
	\begin{enumerate}
		\item For the vertex set
		\begin{equation*}
			X^0 = \frac{1}{m}\ZZ^{2n}\cap \BB_{R_b+\frac{\sqrt{2n}}{m}},
		\end{equation*}
		consider the full cubical complex $ X $ generated by the vertex set $ X^0 $ with the lattice structure of $ \ZZ^{2n} $. We denote the cells of $ X $ as $ \left\{b_\alpha^k\right\}_{k\in\left\{0,\dots,2n\right\},\alpha\in\mathcal{B}_k} $ and recall that the chain complex of the cellular homology $ H_k^{CW}\left(X;\ZZ_2\right) $ is given by
		\begin{equation*}
			\xymatrix@C=2pc
			{
				\dots \ar[r]^-{\partial_{k+1}^X}
				& \spn_{\ZZ_2}\left(\left\{b_\alpha^{k}\right\}_{\alpha\in\mathcal{B}_k}\right) \ar[r]^-{\partial_{k}^X}
				& \dots,
			}
		\end{equation*}
		and we view $ \spn_{\ZZ_2}\left(\left\{b_\alpha^{k}\right\}_{\alpha\in\mathcal{B}_k}\right) $ as the space generated by the indexing set $ \left\langle\mathcal{B}_k\right\rangle $ from now on. The matrices representing the boundary operators $ \partial_k^X $ with respect to the basis $ \left\{b_\alpha^{k}\right\}_{\alpha\in\mathcal{B}_k} $ and $ \left\{b_\alpha^{k-1}\right\}_{\alpha\in\mathcal{B}_{k-1}} $ may be computed explicitly using the \textit{Cellular Boundary Formula} (see \cite[140]{Hat02}). We denote these matrices as $ \partial_k^X $ as well.
		
		\item For the sub-complex $ \partial X $, we denote the cells of $ \partial X $ as $ \left\{b_\alpha^k\right\}_{k\in\left\{0,\dots,2n\right\},\alpha\in\mathcal{D}_k\subset\mathcal{B}_k} $ and compute the matrices representing the maps
		\begin{equation*}
			\begin{split}
				&q_k:\left\langle\mathcal{B}_k\right\rangle \to \left\langle\mathcal{B}_k\setminus\mathcal{D}_k\right\rangle \qquad 	1\leq k\leq 2n \\
				&q_0:\left\langle\mathcal{B}_0\right\rangle \to \left\langle\mathcal{B}_0\setminus\mathcal{D}_0\cup 	b^0_{\mathcal{D}_0}\right\rangle
			\end{split}
		\end{equation*}
		defined by the quotient map $ q:X\to X/\partial X $. Then, the matrices $ \overline{\partial}_k^X = q_{k-1}\cdot\partial_k^X $ are the boundary matrices of the chain complex whose homology is $ H_\ast^{CW}\left(X/\partial X;\ZZ_2\right) $. We denote the indexing sets for the cells of $ X/\partial X $ as $ \overline{\mathcal{B}}_k $ for $ 0\leq k\leq 2n $.
		
		\item For the vertex set
		\begin{equation*}
			Y^0 = \frac{1}{m}\ZZ^{2n\left(N-1\right)}\cap \BB_{R_f+\frac{\sqrt{2n\left(N-1\right)}}{m}},
		\end{equation*}
		consider the full cubical complex $ Y $ generated by the vertex set $ Y^0 $ with the lattice structure of $ \ZZ^{2n\left(N-1\right)} $. We denote the cells of $ Y $ as $ \left\{f_\beta^k\right\}_{k\in\left\{0,\dots,2n\left(N-1\right)\right\},\beta\in\mathcal{F}_k} $, and the matrices representing the boundary operators
		\begin{equation*}
			\partial_k^Y:\left\langle\mathcal{F}_k\right\rangle \to \left\langle\mathcal{F}_{k-1}\right\rangle \qquad 1\leq k \leq 	2n\left(N-1\right),
		\end{equation*}
		of the chain complex whose homology is $ H_\ast^{CW}\left(Y;\ZZ_2\right) $ with respect to the basis $ \left\{f^k_\beta\right\}_{\beta\in\mathcal{F}_k} $ and $ \left\{f^{k-1}_\beta\right\}_{\beta\in\mathcal{F}_{k-1}} $ may be computed explicitly using the cellular boundary formula. We denote these matrices as $ \partial_k^Y $ as well.
		
		\item For the sub-complex $ Y_0 $ which consists of all cells whose vertices belong to
		\begin{equation*}
			\left\{\left(\xi_1^-,\dots,\xi_{N-1}^-,0\dots,0\right)\in\frac{1}{m}\ZZ^{2n\left(N-1\right)}\right\}\cap\partial Y,
		\end{equation*}
		we denote the cells of $ Y_0 $ as $ \left\{f_\beta^k\right\}_{k\in\left\{0,\dots,2n\left(N-1\right)\right\},\beta\in\widetilde{\mathcal{F}}_k\subset\mathcal{F}_k} $ and restrict the matrices $ \partial_k^Y $ to
		\begin{equation*}
			\widetilde{\partial}_k^Y:\left\langle\mathcal{F}_k\setminus\widetilde{\mathcal{F}}_k\right\rangle\to\left\langle\mathcal{F}_{k-1}\setminus\widetilde{\mathcal{F}}_{k-1}\right\rangle \qquad 1\leq k \leq 2n\left(N-1\right).
		\end{equation*}
		So the restricted matrices $ \widetilde{\partial}_k^Y $ represent the boundary operators of the relative chain complex whose homology is $ H_\ast^{CW}\left(Y,Y_0;\ZZ_2\right) $ with respect to the basis $ \left\{f^k_\beta\right\}_{\beta\in\mathcal{F}_k\setminus\widetilde{\mathcal{F}}_k} $ and $ \left\{f^{k-1}_\beta\right\}_{\beta\in\mathcal{F}_{k-1}\setminus\widetilde{\mathcal{F}}_{k-1}} $. We denote the indexing sets $ \mathcal{F}_k\setminus\widetilde{\mathcal{F}}_k $ as $ \overline{\mathcal{F}}_k $.
		
		\item Following a special case of \cite[theorem A.6]{Hat02} we see that $ H_\ast^{CW}\left(K,L;\ZZ_2\right) $ where $ K = X/\partial X\times Y $ and $ L=X/\partial X\times Y_0 $ is the homology associated to the chain complex
		\begin{equation*}
			\xymatrix@C=2pc
			{
				\dots \ar[r]^-{\partial_{j+1}}
				& \bigoplus\limits_{k=0}^j\left\langle\overline{\mathcal{B}}_k\right\rangle\otimes\left\langle\overline{\mathcal{F}}_{j-k}\right\rangle \ar[r]^-{\partial_{j}}
				& \dots,
			}
		\end{equation*}
		so the matrices representing the operators $ \partial_j $ with respect to the basis $ \left\{b_\alpha^k\otimes f^{j-k}_\beta\right\}_{\left(\alpha,\beta\right)\in\overline{\mathcal{B}}_k\times\overline{\mathcal{F}}_{j-k}}^{k=0,\dots,j} $ and $ \left\{b_\alpha^k\otimes f^{j-1-k}_\beta\right\}_{\left(\alpha,\beta\right)\in\overline{\mathcal{B}}_k\times\overline{\mathcal{F}}_{j-1-k}}^{k=0,\dots,j-1} $ are given by the block matrix
		\begin{equation*}
			\newcommand{\rr}[1]{\multicolumn{1}{c|}{#1}}
			\partial_j = \left(\begin{array}{c|cccccc}
				& \langle\overline{\mathcal{B}}_0\rangle\otimes\langle\overline{\mathcal{F}}_{j}\rangle & \langle\overline{\mathcal{B}}_1\rangle\otimes\langle\overline{\mathcal{F}}_{j-1}\rangle & \dotsm & \dotsm & \langle\overline{\mathcal{B}}_{j-1}\rangle\otimes\langle\overline{\mathcal{F}}_{1}\rangle & \langle\overline{\mathcal{B}}_j\rangle\otimes\langle\overline{\mathcal{F}}_{0}\rangle \\ \hline
				\langle\overline{\mathcal{B}}_0\rangle\otimes\langle\overline{\mathcal{F}}_{j-1}\rangle & \rr{\Id_{\overline{\mathcal{B}_0}}\otimes\widetilde{\partial}^Y_{j}} & \rr{\overline{\partial}^X_1\otimes\Id_{\overline{\mathcal{F}}_{j-1}}} & 0 & \dotsm & \dotsm & 0 \\ \cline{2-3}
				\langle\overline{\mathcal{B}}_1\rangle\otimes\langle\overline{\mathcal{F}}_{j-2}\rangle &\rr{0} & \rr{\Id_{\overline{\mathcal{B}}_1}\otimes\widetilde{\partial}^Y_{j-1}} & \ddots & 0 & \dotsm & 0 \\ \cline{3-3}
				\vdots & \vdots & 0 & \ddots & \ddots && \vdots \\ \cline{6-6}
				\langle\overline{\mathcal{B}}_{j-2}\rangle\otimes\langle\overline{\mathcal{F}}_{1}\rangle & \vdots & \vdots && \rr{\ddots} &\rr{\overline{\partial}^X_{j-1}\otimes\Id_{\overline{\mathcal{F}}_1}} & 0 \\ \cline{6-7}
				\langle\overline{\mathcal{B}}_{j-1}\rangle\otimes\langle\overline{\mathcal{F}}_{0}\rangle & 0 & 0 & \dotsm & \rr{\dotsm} & \rr{\Id_{\overline{\mathcal{B}}_{j-1}}\otimes\widetilde{\partial}^Y_1} & \rr{\overline{\partial}^X_j\otimes\Id_{\overline{\mathcal{F}}_0}} \\ \cline{6-7}
			\end{array}\right).
		\end{equation*}
	\end{enumerate}
	
	After obtaining these boundary matrices, we claim that the pair of cell complexes $ \left(K,L\right) $ is homotopic to our desired domain $ \left(\BB_{R_b}/\partial\BB_{R_b}\times\BB_{R_f},\BB_{R_b}/\partial\BB_{R_b}\times \partial D^i\right) $ via homotopy equivalences which are close to the identity in the following sense.
	\begin{lemma} \label{lemma:cellcmplxhomot}
		The pairs $ \left(X,\partial X\right) $ and $ \left(Y,Y_0\right) $ satisfy
		\begin{enumerate}
			\item $ \max\left\{\diam\left(e_\alpha\right);e_\alpha\text{ is a cell of }X\times Y\right\} = \frac{\sqrt{2nN}}{m} $,
			
			\item $ \left(X,\partial X\right) $ is homotopic to $ \left(\BB_{R_b},\partial\BB_{R_b}\right) $ via homotopy equivalences
			\begin{equation*}
				h_b:\left(X,\partial X\right) \to \left(\BB_{R_b},\partial\BB_{R_b}\right) \qquad h_b^{-1}:\left(\BB_{R_b},\partial\BB_{R_b}\right)\to\left(X,\partial X\right),
			\end{equation*}
			that satisfy $ \left\|h_b^{\pm1} - \Id\right\|_{C^0} \leq \frac{\sqrt{2n}}{m} $ and descends to a homotopy equivalence on the quotient
			\begin{equation*}
				\widetilde{h_b}:X/\partial X \to \BB_{R_b}/\partial\BB_{R_b} \qquad \widetilde{h_b^{-1}}:\BB_{R_b}/\partial\BB_{R_b}\to X/\partial X,
			\end{equation*}
			
			\item $ \left(Y,Y_0\right) $ is homotopic to $ \left(\BB_{R_f},\partial D^i\right) $ via homotopy equivalences
			\begin{equation*}
				h_f:\left(Y,Y_0\right) \to \left(\BB_{R_f},\partial D^i\right) \qquad h_f^{-1}:\left(\BB_{R_f},\partial D^i\right)\to\left(Y,Y_0\right),
			\end{equation*}
			which satisfies $ \left\|h_f^{\pm1} - \Id\right\|_{C^0} \leq \frac{\sqrt{2n\left(N-1\right)}}{m} $.
		\end{enumerate}
	\end{lemma}
	
\begin{proof} \
		\begin{enumerate}
			\item Note that in our notations
			\begin{equation*}
				\max\left\{\diam\left(e_\alpha\right);e_\alpha\text{ is a cell of }X\times Y\right\} = \max_{\left(\alpha,\beta\right)\in\mathcal{B}_{2n}\times\mathcal{F}_{2n\left(N-1\right)}}\left\{\diam\left(b_\alpha^{2n}\times 	f_\beta^{2n\left(N-1\right)}\right)\right\},
			\end{equation*}
			so all top dimensional cells in $ X\times Y $ are simply $ 2nN $-dimensional cubes with edges of length $ \frac{1}{m} $, which means that the maximal diameter of such cells is exactly $ \frac{\sqrt{2nN}}{m} $.
			
			\item Since $ X $ is defined as the full cubical complex induced by the lattice structure of $ \ZZ^{2n} $ with all vertices contained in $ \frac{1}{m}\ZZ^{2n}\cap \BB_{R_b+\frac{\sqrt{2n}}{m}} $, it clearly holds that $ \BB_{R_b} \subset X \subset \BB_{R_b+\frac{\sqrt{2n}}{m}} $. We define two homotopies on $ X $:
			\begin{itemize}
				\item $ F_t:\left[0,1\right]\times X \to X $ such that $ F_0 = \Id $ and $ F_1\left(\left(X\setminus\BB_{R_b}\right)\cup\partial\BB_{R_b}\right) = \partial X $ is a radial function that acts increasingly (in variable $ t $) on each ray and such that $ F_t\vert_{\partial X} = \Id $ for all $ t\in\left[0,1\right] $.
				
				\item $ G_t:\left[0,1\right]\times X \to X $ such that $ G_0 = \Id $ and $ G_1\left(\left(X\setminus\BB_{R_b}\right)\cup\partial\BB_{R_b}\right) = \partial\BB_{R_b} $ is a radial function that acts decreasingly (in variable $ t $) on each ray and such that $ G_t\vert_{\BB_{R_b}} = \Id $ for all $ t\in\left[0,1\right] $.
			\end{itemize}
			Note that since the homotopies $ F_t,G_t $ are radial functions which act monotonically (in $ t $) on each ray then $ F_1\left(\partial\BB_{R_b}\right) = \partial X \subset \BB_{R_b+\frac{\sqrt{2n}}{m}} $ and $ G_1\left(\partial X\right) = \partial\BB_{R_b} $ so we get that the desired homotopy equivalences
			\begin{equation*}
				h_b = G_1:\left(X,\partial X\right)\to\left(\BB_{R_b},\partial\BB_{R_b}\right) \qquad h_b^{-1} = 	F_1:\left(\BB_{R_b},\partial\BB_{R_b}\right)\to\left(X,\partial X\right),
			\end{equation*}
			that indeed satisfy $ \left\|h_b^{\pm1} - \Id\right\|_{C^0} \leq \frac{\sqrt{2n}}{m} $.
			
			Finally, in order to show these maps descends to homotopy equivalences on the quotients we need the following claim, the proof of which is omitted for brevity but essentially only uses elementary topological arguments.
			\begin{claim} \label{claim:quotienthomotopy}
				Let $ X $ be a topological space and let $ A_0\subset A\subset X $ be subspaces of $ X $. If $ X_0 \subset \left(X\setminus A\right)\cup A_0 $ is a subspace such that
				\begin{align}
					&\left(X,\left(X\setminus A\right)\cup A_0\right) \sim_{\text{Rel }X_0}\left(X,X_0\right)\text{ through maps which sends }\left(X\setminus A\right)\cup A_0\text{ into itself}, \label{eq:quotienthomotopycond1} \\
					&\left(X,\left(X\setminus A\right)\cup A_0\right)\sim_{\text{Rel }A}\left(A,A_0\right)\text{ through maps which sends }\left(X\setminus A\right)\cup A_0\text{ into itself}, \label{eq:quotienthomotopycond2}
				\end{align}
				then $ X/X_0 $ is homotopy equivalent to $ A/A_0 $.
			\end{claim}
			Thus, the homotopies $ h_b^{\pm1} $ show that $ \left(X,\left(X\setminus\BB_{R_b}\right)\cup\partial \BB_{R_b}\right) \sim_{\text{Rel }X_0} \left(X,\partial X\right) $ through maps that send $ \left(X\setminus\BB_{R_b}\right)\cup\partial \BB_{R_b} $ into itself and that $ \left(X,\left(X\setminus\BB_{R_b}\right)\cup\partial \BB_{R_b}\right) \sim_{\text{Rel }A} \left(\BB_{R_b},\partial\BB_{R_b}\right) $ through maps that send $ \left(X\setminus\BB_{R_b}\right)\cup\partial \BB_{R_b} $ into itself, so $ \partial\BB_{R_b} \subset \BB_{R_b}\subset X $ and $ \partial X\subset\left(X\setminus\BB_{R_b}\right)\cup\partial\BB_{R_b} $ satisfy the assumptions \Cref{eq:quotienthomotopycond1},\Cref{eq:quotienthomotopycond2} in the previous claim therefore they descend to the quotient and we get homotopy equivalences
			\begin{equation*}
				\widetilde{h_b}:X/\partial X \to \BB_{R_b}/\partial\BB_{R_b} \qquad \widetilde{h_b^{-1}}:\BB_{R_b}/\partial\BB_{R_b}\to X/\partial X,
			\end{equation*}
			ad claimed.
			
			\item Since the complex $ Y $ is defined same as $ X $ is also clearly holds that $ \BB_{R_f} \subset Y \subset \BB_{R_f+\frac{\sqrt{2n\left(N-1\right)}}{m}} $, and the same construction gives homotopy equivalences
			\begin{equation*}
				h_f:\left(Y,\partial Y\right)\to\left(\BB_{R_f},\partial\BB_{R_f}\right) \qquad 	h_f^{-1}=\left(\BB_{R_f},\partial\BB_{R_f}\right)\to\left(Y,\partial Y\right),
			\end{equation*}
			that satisfy $ \left\|h_f^{\pm1} - \Id\right\|_{C^0} \leq \frac{\sqrt{2n\left(N-1\right)}}{m} $, so by definition of the sub complex $ Y_0 $ and the subset $ \partial D^i $ we get that $ h_f\left(Y_0\right) = \partial D^i $ and $ h_f^{-1}\left(Y_0\right) = \partial D^i $, which proves the Lemma.
		\end{enumerate}
	\end{proof}
	
	So, we conclude that the following maps are homotopy equivalences of our domain and $ \left(K,L\right) $
	\begin{equation*}
		\begin{split}
			h&=\widetilde{h_b}\oplus h_f :\left(K,L\right)\to\left(\BB_{R_b}/\partial\BB_{R_b}\times\BB_{R_f},\BB_{R_b}/\partial\BB_{R_b}\times \partial D^i\right)\\
			h^{-1}&=\widetilde{h_b^{-1}}\oplus h_f^{-1} :\left(\BB_{R_b}/\partial\BB_{R_b}\times\BB_{R_f},\BB_{R_b}/\partial\BB_{R_b}\times \partial D^i\right)\to \left(K,L\right).
		\end{split}
	\end{equation*}

	\subsection{Step 2 - Filtering the complex} \label{sec:step2}
	Once we have a cell complex which approximates our domain, we wish to sample the generating function on this complex to get a barcode related to sub-level sets - we claim that this barcode is close in bottleneck distance to the barcode defined in the smooth setting. To make this more precise we start by introducing the notions of \textit{filtered chain complex} and the \textit{persistent homology} it induces, following \cite[section 6.2]{Pol19} we define everything on the chain level and so we may later use the relative chain complex as in our case.
	
	\paragraph[filtered chain cmplxs]{}
	\begin{definition} \label{def:filteredcomplex}
		A \textit{filtered chain complex} is a chain complex $ \left(C_\ast,\partial\right) $ along with a sequence of increasing sub-complexes called a \textit{filtration}
		\begin{equation*}
			\emptyset = C_\ast^0 \subset C_\ast^1 \subset \dots C_\ast^N = C,
		\end{equation*}
		such that the differential preserves the filtration, i.e $ \partial\left(C_\ast^j\right) \subset C_{\ast-1}^j $ for all $ 0\leq j\leq N $.	
	\end{definition}
	Let $ \left(C_\ast,\partial\right) $ be a bounded chain complex of finite type with maximal dimension $ n $ and for each $ 0\leq k\leq n $, let $ \mathcal{E}_k $ be a chosen (finite and unordered) basis for $ C_k $. Denote $ C = \oplus_{k=0}^n C_k $ and $ \mathcal{E} = \cup_{k=1}^n \mathcal{E}_k $, if $ u:\mathcal{E}\to\RR $ is a function defined on the basis of $ C $ then $ u $ may be extended to all chains in $ C $ by
	\begin{equation} \label{eq:extensionformbasis}
		u\left(\sum_{e\in \mathcal{E}}a_e e\right) = \max_{a_e\neq 0}u\left(e\right).
	\end{equation}
	The function $ u:C\to\RR $ is called \textit{monotonic} if it is decreasing with respect to the operation of $ \partial $, i.e $ u\left(\partial c\right)\leq u\left(c\right) $ for all $ c\in C $, the nature of the function's extension from $ \mathcal{E} $ to $ C $ implies that its enough to check monotonicity on basis elements, i.e $ u\left(\partial e\right) \leq u\left(e\right) $ for all $ e\in \mathcal{E} $. For $ t\in\RR $ we set
	\begin{equation*}
		C_k^t = \left\{c^k\in C_k;u\left(c^k\right)<t\right\}, \text{ generated by: }  \mathcal{E}_k^t = \left\{e^k\in \mathcal{E}_k;u\left(e^k\right) < t\right\},
	\end{equation*}
	and denote $ C^t = \oplus_{k=0}^n C_k^t, \mathcal{E}^t = \cup_{k=0}^n\mathcal{E}_k^t $ as before. Since $ \mathcal{E} $ is a finite set, the function $ u $ only takes values in $ \left\{\min u = t_0 \leq t_2 \leq \dots \leq t_N = \max u\right\} $ where $ N = \left|\mathcal{E}\right| $ so we have a sequence of chain complexes
	\begin{equation*}
		\emptyset = C_\ast^{t_0} \subset C_\ast^{t_2} \subset \dots \subset C_\ast^{t_N} \subset C_\ast,
	\end{equation*}
	such that $ \partial\left(C_\ast^{t_j}\right) \subset C_{\ast-1}^{t_j} $ for all $ j $. We conclude that a function $ u:\mathcal{E}\to\RR $ gives rise to a filtered chain complex this way and we say this filtration is induces by the function $ u $ and a preferred basis $ \mathcal{E} $.
	
	Assume form now on that all homology groups are taken with $ \ZZ_2 $ coefficients, then this filtration defines a p.m $ V\left(C_\ast^t\right) $ by considering the homology of sub-complexes $ \bigoplus_{k=0}^n H_k\left(C_\ast^t,\partial\right) $ with the morphisms $ \pi_{s,t} $ induced by inclusions $ C_\ast^s\subset C_\ast^t $, since the function only takes on a finite number of values this p.m is given by a finite sequence (in each degree)
	\begin{equation*}
		\xymatrix@C=2pc
		{
			0 = H_{k}\left(C_\ast^{0},\partial\right) \ar[r]
			& H_{k}\left(C_\ast^{1},\partial\right) \ar[r]
			& \dots \ar[r]
			& H_{k}\left(C_\ast^{N},\partial\right) \ar[r]
			& H_{k}\left(C_\ast,\partial\right),
		}
	\end{equation*}
	where we denote $ C_\ast^{t_j} = C_\ast^j $ and $ \pi_{i,j} = \pi_{t_i,t_j} $ for brevity. Let $ \alpha \in H_{k}\left(C_\ast^i,\partial\right) $, we say that $ \alpha $ is \textit{born} at level $ i $ if $ \alpha\notin\Ima\pi_{i-1,i} $ and we say $ \alpha $ \textit{dies entering} level $ j $ if it merges with an "older" class when moving from $ H_k\left(C_\ast^{j-1,\partial}\right) $ to $ H_k\left(C_\ast^j,\partial\right) $, i.e if $ \pi_{i,j-1}\left(\alpha\right)\notin\Ima\pi_{i-1,j-1} $ but $ \pi_{i,j}\left(\alpha\right) \in\Ima\pi_{i-1,j} $.
	\begin{remark}
		In some computational topology texts (such as \cite{EdHar10}) the previous construction takes a slightly different form, the basis $ \mathcal{E} $ is often considered as a simplicial complex and so the function $ u $ defines a sequence of increasing simplicial complexes which is referred to as a \textit{filtered complex}. This filtered complex corresponds to a sequence of homology groups connected by morphisms (induced by inclusions) and thus defines \textit{persistent homology groups} as the images of these morphisms between each two levels of the filtration, this notion corresponds to our persistence modules. Finally, a \textit{persistent diagram} is used to visualize the data and read information about the persistent homology, this corresponds to our barcode.
	\end{remark}
	
	\paragraph[filtering the cmplx]{}
	Going back to our setting, we use the input samples $ S_j\vert_{\frac{1}{m}\ZZ^{2n}} $ along with the expression obtained for the GFQI $ S $ in \Cref{eq:defGFQI} to get a sample of $ S $ on a compact set (therefore the sample is finite), i.e we have
	\begin{equation*}
		S\vert_{\frac{1}{m}\ZZ^{2nN}\cap\left(\BB_{R_b+\frac{\sqrt{2n}}{m}}\times\BB_{R_f+\frac{\sqrt{2n\left(N-1\right)}}{m}}\right)}
	\end{equation*}
	which means we have a sample of $ S $ on the set of vertices
	\begin{equation*}
		\left\{b_\alpha^0\times f_\beta^0\right\}_{\left(\alpha,\beta\right)\in\mathcal{B}_0\times\mathcal{F}_0} = X^0\times Y^0 	\subset \RR^{2nN},
	\end{equation*}
	which descends to a sample of $ S $ on $ K^0 = \left(X^0/\left(\partial X\right)^0\right) \times Y^0 $ since $ S = \mathcal{Q} $ outside $ \BB_{R_b} $ and we get a sample of $ S\vert_{K^0} $ that induces a filtration on the chain complex
	\begin{equation}\label{eq:relchaincmplx}
		\xymatrix@C=2pc
		{
			\dots \ar[r]^-{\partial_{j+1}}
			& \bigoplus\limits_{k=0}^j\left\langle\overline{\mathcal{B}}_k\right\rangle\otimes\left\langle\overline{\mathcal{F}}_{j-k}\right\rangle \ar[r]^-{\partial_{j}}
			& \dots,
		}
	\end{equation}
	the homology of which is $ H_{\ast}^{CW}\left(K,L;\ZZ_2\right) $, by extending this sample to a function on the set of basis elements
	\begin{equation*}
		\widetilde{S}:\bigcup_{j=0}^{2nN}\left\{b_\alpha^k\otimes f^{j-k}_\beta\right\}_{\left(\alpha,\beta\right)\in\overline{\mathcal{B}}_k\times\overline{\mathcal{F}}_{j-k}}^{k=0,\dots,j} \to \RR,
	\end{equation*}
	defined recursively as
	\begin{equation*}
		\widetilde{S}\left(b_\alpha^k\otimes f^{j-k}_\beta\right) = \max\left\{\widetilde{S}\left(b_{\alpha^\prime}^{k^\prime}\otimes f^{j-1-{k^\prime}}_{\beta^\prime}\right);b_{\alpha^\prime}^{k^\prime}\otimes 	f^{j-1-{k^\prime}}_{\beta^\prime}\in\partial_j\left(b_\alpha^k\otimes f^{j-k}_\beta\right)\right\},
	\end{equation*}
	so we get a function $ \widetilde{S} $ defined on the set of cells in $ K\setminus L $ that is decreasing along boundaries of cells.
	
	This collection of cells serve as a basis for the vector spaces in the chain complex associated to the relative homology $ H_\ast^{CW}\left(K,L;\ZZ_2\right) $ given by \Cref{eq:relchaincmplx}, and since $ \widetilde{S} $ decreases along boundaries, the extension of $ \widetilde{S} $ to the entire vector space (defined by \Cref{eq:extensionformbasis}) produces a monotonic function and we get a filtered chain complex induced by $ \widetilde{S} $ with a preferred basis given by the cells of $ K\setminus L $ and we say this filtration is induced by $ S\vert_{K^0} $. We can now define a p.m as follows.
	
	\begin{definition}
		Let $ \varphi\in\Ham_c\left(\RRnn,\omega_0\right) $ and let $ S:\RRnn\times\RR^d\to\RR $ be a GFQI of $ \varphi $ with quadratic radii $ R_b,R_f $ and quadratic index $ i $. Assume that the cell complexes $ K=X/\partial X \times Y $ and $ L = X/\partial X\times Y_0 $ satisfy the assumptions in \Cref{lemma:cellcmplxhomot}, then the approximated persistence module $ \widetilde{V}\left(\varphi,S,K,L\right) $ is the p.m associated to the filtered chain complex \Cref{eq:relchaincmplx} with grading shifted by $ i $, preferred basis given by all cells in $ K\setminus L $ and the filtration induced by $ S\vert_{K^0} $.
	\end{definition}
	We conclude that $ \widetilde{V}\left(\varphi,S,K,L\right) $ is then the p.m associated to the chain complex whose boundary matrices are $ \left\{\partial_j\right\}_{j=1}^{2nN} $, filtered by $ \widetilde{S} $ and with grading shifted by $ N-1 $. We permute the rows and columns of all matrices $ \partial_j $ according to the ordering of basis elements given by the values of $ \widetilde{S} $ and re-denote the new matrices as $ \partial_j $.

	\subsection{Step 3 - Extracting the barcode} \label{sec:step3}
	Finally, after permuting the rows and columns of all boundary matrices according to the filtration induced by our generating function, we wish the extract the barcode associated to the p.m $ \widetilde{V}\left(\varphi,S,K,L\right) $ from the data encoded in these matrices and claim it is close in bottleneck distance to our desired barcode, for this we first describe a process called matrix reduction which enables one to read the barcode of a filtered chain complex from its boundary matrices.
	
	\paragraph[Boundary matrix reduction]{}
	Given a filtered chain complex $ C_\ast^t $ with preferred basis $ \mathcal{E} $, the p.m $ V\left(C_\ast^t\right) $ has a barcode denoted $ \mathcal{B}\left(C_\ast^t\right) $ and we show how this barcode can be explicitly recovered from the matrix of the boundary operator $ \partial $ with respect to the basis $ \mathcal{E} $ with an appropriate ordering. This can be done by turning the matrix to \textit{reduced} form.
	\begin{definition} \label{def:reducedmatrix}
		Let $ D $ be a matrix with $ \ZZ_2 $ coefficients. For every non zero column $ j $, the \textit{pivot index} denoted $ \Pivot\left(j\right) $ is the maximal index $ i $ such that $ \partial_{ij} = 1 $ and $ D $ is called \textit{reduced} if no two columns has the same pivot index.
	\end{definition}
	We shall next give a way of obtaining reduced form of $ \partial $ by a product $ R=\partial\cdot T $ with a triangular matrix, note that zero columns correspond to cycles (since their boundary is empty) and non-zero columns represent boundaries (since they contain the result of $ \partial $ operating on the $ j $-th column of $ T $). We explain here the procedure of defining the \textit{compatible ordering} on $ \mathcal{E} $ and reducing the matrix of $ \partial $ in order to read the desired barcode.
	\begin{enumerate}
		\item We start by ordering the basis elements $ e^k\in \mathcal{E}_k $ for each degree $ k $ in ascending order according to their value $ u\left(e^k\right) $ (in case of equality the order is arbitrary). So we get sequences $ \left\{e_1^k,\dots,e_{N_k}^k\right\} = \mathcal{E}_k $ such that
		\begin{equation*}
			u\left(e_i^k\right) < u\left(e_j^k\right) \implies i<j,
		\end{equation*}
		then we put all basis elements in one sequence ordered by degree and then by the internal order, we get the sequence
		\begin{equation*}
			\left\{e_1,\dots,e_N\right\} = \left\{e_1^0,\dots,e_{N_1}^0;e_1^1,\dots,e_{N_1}^1;\dots;e_1^k,\dots,e_{N_k}^k;\dots;e_1^n\dots,e_{N_n}^n\right\}=\mathcal{E}.
		\end{equation*}
		
		\item Denote by $ \left(\partial_{ij}\right)_{i,j=1\dots,N} $ the matrix with $ \ZZ_2 $ coefficients associated to the operator $ \partial:C\to C $ with respect to this ordered basis (i.e $ \partial\left(e_j\right) = \sum_{i=1}^{N} \partial_{ij}e_i $). The boundary of every basis element is recorded in its column and note that since $ \partial\left(C_k\right) \subset C_{k-1} $, $ \partial_{ij} $ is in fact a block matrix where each block $ \left(\partial^k_{i,j}\right)^{k=1,\dots,n}_{i=1,\dots,N_{k-1},j=1,\dots,N_k} $ is upper triangular and corresponds to $ \partial\vert_{C_k} $.
		
		\item We reduce the matrix $ \partial $ by adding columns from left to right - at the $ j $-th column, we check if there exists a column $ j_0 < j $ such that $ \Pivot\left(j_0\right) = \Pivot\left(j\right) $, in which case we add column $ j_0 $ to column $ j $. Note that this pivot indices can only be the same for columns of the same degree since they represent a boundary element, furthermore, because this matrix is over $ \ZZ_2 $, addition decreases the value of $ \Pivot\left(j\right) $. We continue checking for columns $ j_0 < j $ with $ \Pivot\left(j_0\right) = \Pivot\left(j\right) $ until the first $ j $ columns form a reduced sub-matrix.
		
		Note that since we only perform column operations by degree and from left to right, the matrix $ T_{ij} $ which encodes these operations is a block matrix $ T_{ij}^k $ with upper triangular blocks that has non zero diagonal, and so we get that this reduced form is given by $ R_{ij} = \partial_{ij} \cdot T_{ij} $. $ R_{ij} $ is also a block matrix, we denote its blocks by $ R_{ij}^k = \partial_{ij}^k\cdot T_{ij}^k $ for $ k=1,\dots,n $. Furthermore, for non-zero columns in $ R_{ij}^k $ we denote by $ i=\Pivot_k\left(j\right) $ the maximal index $ i $ with $ R_{ij}^k = 1 $.
		
		\item After $ \partial_{ij} $ is reduced this way, its columns represent the boundaries of chains $ \left\{f^{k+1}_1,\dots,f^{k+1}_{N_{k+1}}\right\} $ where $ f^{k+1}_j = \sum_{i< j}T_{ij}^{k+1} e_i^{k+1} + e_j^{k+1} $ and we look at all pairs of chains $ \left(e_i^{k},f_j^{k+1}\right) $ such that $ i = \Pivot_{k+1}\left(j\right) $. Note that in this case column $ i $ in $ R_{ij}^k $ must be zero and we get $ u\left(f_j^{k+1}\right) \geq u\left(e_i^{k}\right) $ since $ e_i^{k} $ belongs to the boundary of $ f_j^{k+1} $, we call such a pair \textit{essential} if $ u\left(e_i^{k}\right) < u\left(f_j^{k+1}\right) $.
		
		\item For each degree $ k<n $, we denote the interval $ I_i^{k} = \left(u\left(e_i^{k}\right),u\left(f_j^{k+1}\right)\right] $ whenever the pair $ \left(e_i^{k},f_j^{k+1}\right) $ is essential and for all  degrees $ k $ we denote the rays $ J_i^{k} = \left(u\left(e_i^{k}\right),\infty\right) $ whenever column $ i $ of $ R_{ij}^k $ is zero and $ i \neq \Pivot_{k+1}\left(j\right) $ for any $ j $.
	\end{enumerate}
	\begin{lemma} \label{lemma:barcodepairing}
		The barcode containing all intervals $ I_i^k $ and $ J_i^k $ ($ 0\leq k \leq n$) is equal to $ \mathcal{B}\left(C_\ast^t\right) $.
	\end{lemma}
	\begin{proof}
		For each degree $ k $, the barcode $ \mathcal{B}\left(C_\ast^t\right) $ consists of intervals associated to linearly independent generators of the homology groups $ H_k\left(C_\ast^t,\partial\right) = \ker\left(\partial_k\vert_{C_k^t}\right)/\Ima\left(\partial_{k+1}\vert_{C_{k+1}^t}\right) $, we shall show that for every $ t\in\RR $ these generators correspond to chains derived from the intervals $ I_i^k $ and $ J_i^k $ which contains $ t $.
		
		First, note that for $ k>0 $ we have the change of basis from $ e^{k}_1,\dots,e^{k}_{N_{k}} $ to $ f^{k}_1,\dots,f^{k}_{N_{k}} $ which is done by a triangular matrix with non zero diagonal so the vectors $ f_j^{k} $ are all linearly independents. Moreover, the extension of $ u $ from $ \mathcal{E} $ to $ C $ given by \Cref{eq:extensionformbasis} implies that
		\begin{equation*}
			u\left(f_j\right) = u\left(\sum_{i\leq j}T_{ij}e_i\right) = \max_{T_{ij}=1}u\left(e_i\right),
		\end{equation*}
		and since our ordering of $ e_1,\dots,e_N $ assures us that $ u\left(e_i\right)\geq u\left(e_{i-1}\right) $ as long as both has the same degree, we see that in fact $ u\left(f_j^{k}\right) = u\left(e_j^{k}\right) $ and conclude the space $ C_k^t $ is generated by the chains $ f_i^k $ with $ u\left(f_i^k\right) < t $ and also by the chains $ e_i^k $ with $ u\left(e_i\right)<t $. We consider the following chains in $ C_k $ when $ 0<k<n $.
		\begin{equation*}
			c^k_i = \begin{cases}
				\partial f^{k+1}_j & i=\Pivot_{k+1}\left(j\right) \\
				f^k_i & i\neq\Pivot_{k+1}\left(j\right)\text{ for all } j\text{ and }\partial f_i^k = 0,
			\end{cases}
		\end{equation*}
		and when $ k=0,n $ we set
		\begin{equation*}
			c^0_i = \begin{cases}
				\partial f^{1}_j & i=\Pivot_{1}\left(j\right) \\
				e^0_i & i\neq\Pivot_{1}\left(j\right)\text{ for all } j
			\end{cases} \qquad c^n_i = \begin{cases}
			f^n_i & \partial f^n_i = 0
		\end{cases}.
		\end{equation*}
		So, all chains $ c_i^k $ are linearly independent cycles and in all cases we have $ u\left(c^k_i\right) = u\left(e^k_i\right) $. Let $ t\in\RR $, we claim that the set
		\begin{equation*}
			\left\{c_i^k;\exists i\text{ such that }t\in I_i^k\text{ or }t\in J_i^k\right\},
		\end{equation*}
		is a basis of $ H_k\left(C_\ast^t,\partial\right) $. Then, the interval modules generated by this basis correspond to parts of the direct sum in the p.m $ H_k\left(C_\ast^t,\partial\right) $ so the degree $ k $ part of the barcode $ \mathcal{B}\left(C_\ast^t\right) $ is exactly given by the intervals $ I_i^k,J_i^k $.
		
		We prove the claim: since $ c_i^k $ are all cycles, we just need to show that $ c_i^k\in C_k^t $ and $ c_i^k $ is not a boundary if and only if $ t\in I_i^k $ or $ t\in J_i^k $. We split into two cases:
		\begin{enumerate}
			\item If $ i=\Pivot_{k+1}\left(j\right) $, then $ c_i^k = \partial f_j^{k+1} $. So, the addition of $ e_i^k $ marks the birth of this cycle at level $ u\left(e_i^k\right) = u\left(c_i^k\right) $ and the addition of $ e^{k+1}_j $ marks the death of this cycle (it becomes a boundary) when entering level $ u\left(e^{k+1}_j\right) = u\left(f^{k+1}_j\right) $ - more precisely, we have
			\begin{itemize}
				\item $ u\left(e_i^k\right) < t $ if and only if $ c_i^k = \partial f_j^{k+1} $ is a cycle in $ C_k^t $,
				\item $ u\left(f_j^{k+1}\right) < t $ if and only if $ c_i^k = \partial f_j^{k+1} $ is a boundary in $ C_k^t $.
			\end{itemize}
			The former holds since $ u\left(c_i^k\right) = u\left(e_i^k\right) < t $ implies that $ c_i^k\in C_k^t $ (it is always a cycle) and $ u\left(c_i^k\right) = u\left(e_i^k\right) \geq t $ implies $ c_i^k\notin C_k^t $. The latter holds since $ u\left(f_j^{k+1}\right) < t $ implies $ f_j^{k+1}\in C_{k+1}^t $ so $ \partial f_j^{k+1}\in C_k^t $ is a boundary and if $ u\left(f_j^{k+1}\right) \geq t $ then $ f_j^{k+1}\notin C_{k+1}^t $ and so the cycle $ c_i^k = \partial f_j^{k+1} $ cannot be a boundary of any "younger" cycle in $ C_{k+1}^t $, indeed, using the fact that $ f^{k+1}_p $ forms a basis of $ C_{k+1} $ and by our ordering we have $ f_p^{k+1}\notin C_{k+1}^t $ for all $ p\geq j $, we assume
			\begin{equation*}
				c_i^k = \partial\left(\sum_{p<j}a_pf_p^{k+1}\right) =  \sum_{p<j}a_p\partial f_p^{k+1},
			\end{equation*}
			then since $ i=\Pivot_{k+1}\left(j\right) $ we know that $ c_i^k = \sum_{q<i}b_qe_q^k + e_i^k $ so some of the chains $ a_p\partial f_p^{k+1} $ must be non-zero and thus they have unique pivot indices $ q = \Pivot_{k+1}\left(p\right) $, it follows that one of these indices must be equal to $ i $, which contradicts the fact that the matrix $ R_{ij} $ is reduced.
			
			\item If $ i\neq\Pivot_{k+1}\left(j\right) $ for all $ j $ and $ \partial f_i^k = 0 $, then $ c_i^k = f_i^{k} $ (or $ c_i^0 = \partial e_i^0 $ in case $ k=0 $). So, the addition of $ e_i^k $ marks the birth of this cycle at level $ u\left(e_i^k\right)=u\left(c_i^k\right) $ and this cycle never died since it is not the boundary of any $ \left(k+1\right) $-chain, so we have
			\begin{itemize}
				\item $ u\left(e_i^k\right) < t $ if and only if $ c_i^k $ is a cycle in $ C_k^t $ which is not a boundary.
			\end{itemize}
			This holds since $ u\left(c_i^k\right) = u\left(e_i^k\right) \geq t $ implies the $ c_i^k\notin C_k^t $ and $ u\left(c_i^k\right) = u\left(e_i^k\right) < t $ implies $ c_i^k\in C_k^t $, then it is a cycle by our assumption and it cannot be a boundary since $ i $ is not a pivot index of any chain.
		\end{enumerate}
		Thus, we showed that in both cases, $ c_i^k $ is a generator of $ H_k\left(C_\ast^t,\partial\right) $ if and only if there exists $ i $ such that $ u\left(e_i^k\right)<t\leq u\left(f_j^{k+1}\right) $ and $ i=\Pivot_{k+1}\left(j\right) $ or when $ u\left(e_i^k\right)<t $ and $ \partial f_i^k = 0 $ which exactly means that $ t\in I_i^k = \left(u\left(e_i^k\right),u\left(f_j^{k+1}\right)\right] $ for an essential pair $ \left(e_i^k,f_j^{k+1}\right) $, or $ t\in J_i^k = \left(u\left(e_i^k\right),\infty\right) $ when $ f_i^k $ is a cycle, as claimed.
	\end{proof}	
	
	During this process we used the pairing between chains $ e_i $ and $ f_j $ given by the pivot index, it is not clear up front why this pairing is well defined and does not depend on the reduced form of $ \partial_{ij} $ but it is indeed the case as stated by the \textit{pairing lemma} \cite[183]{EdHar10}:
	\begin{proposition}[Pairing lemma]
		The pairing $ i = \Pivot\left(j\right) $ is unique in the sense that it does not depends on the reduced form $ R $ obtained from $ \partial $ by left to right column operations.
	\end{proposition}
	With \Cref{lemma:barcodepairing} we see that the barcode of a filtered chain complex can be explicitly computed and next we conclude the argument by defining a filtered complex whose barcode approximates $ \mathcal{B}\left(\varphi\right) $.
	\begin{remark}
		The procedure outlined here for the reduction of a boundary matrix is referred to as the \textit{standard} reduction technique in computational topology - this process can be explained and visualized geometrically in a simple way. It's worth mentioning that an alternative construction is available, which turns $ e_i $ into a jordan basis $ f_i $ by an upper triangular change of basis, this construction is more algebraic but does not use any "heavy" tools and is outlined in \cite[section 6.2]{Pol19}.
		
		Furthermore, many improvements over the standard reduction technique exists, see for example \cite{Phat17} for an implementation of several such reduction techniques. Even though some of these improvements have better running times in special cases, the worst case complexity remains cubic (in the number of basis elements), as is the standard reduction.
	\end{remark}

	\paragraph[extracting the barcode]{}
	Using this on the boundary matrices $ \partial_j $ gives the barcode associated to $ \widetilde{V}\left(\varphi,S,K,L\right) $ and we bound the bottleneck distance of this barcode from $ \mathcal{B}\left(\varphi\right) $ using the following lemma, which bounds the interleaving distance of $ \widetilde{V}\left(\varphi,S,K,L\right) $ from $ V\left(\varphi\right) $.
	\begin{lemma} \label{lemma:interleavingpm}
		Let $ \varphi\in\Ham_c\left(\RRnn,\omega_0\right) $ and let $ S:\RRnn\times\RR^d\to\RR $ be a GFQI of $ \varphi $ with quadratic radii $ R_b,R_f $. Assume that the cell complexes $ K=X/\partial X \times Y $ and $ L = X/\partial X\times Y_0 $ satisfy the assumptions in \Cref{lemma:cellcmplxhomot} and denote
		\begin{equation*}
			r = \max\left\{\diam\left(e_\alpha\right);e_\alpha\text{ is a cell of }X\times Y\right\} \qquad l = \max_{\left(x,\xi\right)\in\BB_{R_b}\times\BB_{R_f}}\left|\nabla S\left(x,\xi\right)\right|.
		\end{equation*}
		Then, for all $ \delta > l\left(\frac{\sqrt{2n+d}}{m} + r\right) $, the persistence modules $ V\left(\varphi\right) $ and $ \widetilde{V}\left(\varphi,S,K,L\right) $ are $ \delta $-interleaved.
	\end{lemma}

	\begin{proof}
		The approximated p.m $ \widetilde{V}\left(\varphi,S,K,L\right) $ is given by the p.m associated to the homology of the filtered chain complex
		\begin{equation*}
			\xymatrix@C=2pc
			{
				\dots \ar[r]^-{\partial_{k+i+1}}
				& \spn_{\ZZ_2}\left\{e_a^{k+i}; e_a^{k+i} \text{ is a $ \left(k+i\right) $-cell of $ K\setminus L $}\right\} \ar[r]^-{\partial_{k+i}}
				& \dots,
			}
		\end{equation*}
		with the filtration induced by $ \widetilde{S} $ which is an extension of $ S:\left(\RRnn\cup\left\{\infty\right\}\right)\times\RR^d\to\RR $ from the set of vertices $ K^0 $ to all cells taken as the maximal value over all vertices contained in each cell, i.e for $ t\in\RR $, the spaces $ \widetilde{V}_t $ are the homology associated to the chain complex
		\begin{equation} \label{eq:Vtilde_t}
			\xymatrix@C=2pc
			{
				\dots \ar[r]^-{\partial_{\ast+i+1}}
				& \spn_{\ZZ_2}\left\{e_a^{\ast+i}; e_a^{\ast+i} \text{ is an $ \left(\ast+i\right) $-cell of $ K\setminus L $ and $ \widetilde{S}\left(e_a^{\ast+i}\right)<t $}\right\} \ar[r]^-{\partial_{\ast+i}}
				&\dots.
			}
		\end{equation}
	
		Recall that we have $ K = X/\partial X\times Y $ for cell complexes $ X,Y $ satisfying the assumption of \Cref{lemma:cellcmplxhomot} and for every point $ \left(\left[x\right],\xi\right)\in K $ we denote by $ e\left(\left[x\right],\xi\right) $ the cell $ e_a^k $ which $ \left(\left[x\right],\xi\right)\in\Int e_a^k $ (in case $ k=0 $ we just have $ e\left(\left[x\right],\xi\right) = \left(\left[x\right],\xi\right) $). We then consider the extension $ \widetilde{S} $ as a function defined on the total space $ K\setminus L $ by setting
		\begin{equation*}
			\widetilde{S}\left(\left[x\right],\xi\right) = 	\widetilde{S}\left(e\left(\left[x\right],\xi\right)\right),
		\end{equation*}
		and for $ t\in\RR $ we denote by $ K_t = L\cup\left\{\left(\left[x\right],\xi\right)\in K\setminus L; \widetilde{S}\left(\left[x\right],\xi\right)<t\right\} $ then $ K_t $ is a sub-complex of $ K $ which contains $ L $ for all $ t $ and the space $ H_{\ast+i}\left(K_t,L;\ZZ_2\right) = H_{\ast+i}^{CW}\left(K_t,L;\ZZ_2\right) $ is exactly the homology associated to the chain complex in \Cref{eq:Vtilde_t} so we see that in fact $ \widetilde{V}_t =  H_{\ast+i}\left(K_t,L;\ZZ_2\right) $. We further introduce the following notations.
		\begin{itemize}
			\item The $ k $-cells of $ X $ are denoted by $ \left\{b_\alpha^k\right\}_{\alpha\in\mathcal{B}_k} $ and $ k $-cells of $ \partial X $ by $ \left\{b_\alpha^k\right\}_{\alpha\in\mathcal{D}_k\subset\mathcal{B}_k} $.
			
			\item The $ 0<k $-cells of $ X/\partial X $ are $ \left\{b_\alpha^k\right\}_{\alpha\in\mathcal{B}_k\setminus\mathcal{D}_k} $ and the $ 0 $-cells are $ \left\{b_\alpha^0\right\}_{\alpha\in\mathcal{B}_0\setminus\mathcal{D}_0}\cup \left[b_{\mathcal{D}_0}^0\right] $
			
			\item The $ k $-cells of $ Y $ are denoted by $ \left\{f_\beta^k\right\}_{\beta\in\mathcal{F}_k} $ and the $ k $-cells of $ Y_0 $ by $ \left\{f_\beta^k\right\}_{\beta\in\widetilde{\mathcal{F}}_k\subset\mathcal{F}_k} $.
			
			\item The $ k $-cells of $ K $ are then $ \left\{b_\alpha^j\times f_\beta^{k-j}\right\}_{\left(\alpha,\beta\right)\in\mathcal{B}_j\setminus\mathcal{D}_j\times\mathcal{F}_{k-j}} $ or $ \left\{\left[b_{\mathcal{D}_0}^0\right]\times f_\beta^k\right\}_{\beta\in\mathcal{F}_k} $.
			
			\item The $ k $-cells of $ L $ are thus $ \left\{b_\alpha^j\times f_\beta^{k-j}\right\}_{\left(\alpha,\beta\right)\in\mathcal{B}_j\setminus\mathcal{D}_j\times\widetilde{\mathcal{F}}_{k-j}} $ or $ \left\{\left[b_{\mathcal{D}_0}^0\right]\times f_\beta^k\right\}_{\beta\in\widetilde{\mathcal{F}}_k} $.
			
			\item By definition of $ \widetilde{S} $ as a function on the total space, each point $ \left(\left[x\right],\xi\right)\in K\setminus L $ has a vertex $ e_a^0\in K^0 $ such that $ \widetilde{S}\left(e\left(\left[x\right],\xi\right)\right) = \widetilde{S}\left(e_a^0\right) $, we denote this vertex as $ e^0\left(\left[x\right],\xi\right) $.
			
			\item Note that $ e^0\left(\left[x\right],\xi\right) $ is a $ 0 $-cell of $ K $ which takes one of two forms:
			\begin{equation*}
				b_\alpha^0\times f_\beta^0\qquad\text{ or }\qquad\left[b_{\mathcal{D}_0}^0\right]\times f_\beta^0,
			\end{equation*}
			in both cases we denote the vertex $ f_\beta^0\in\RR^d $ as $ f^0\left(\xi\right) $.
			
			\item When $ x\in\Int X $ we have $ \left[x\right] = x $ and the vertex $ e^0\left(x,\xi\right) $ might still take on both forms (that depends on how close $ x $ is to the boundary and on the values of $ S $ near $ \left(x,\xi\right) $). In both cases, there exists a vertex $ b_\alpha^0\in X $ such that $ e^0\left(x,\xi\right) = \left[b_\alpha^0\right]\times f^0\left(\xi\right) $ and $ \left|\left(x,\xi\right) - \left(b_\alpha^0,f^0\left(\xi\right)\right)\right|\leq r $. We denote the vertex $ b_\alpha^0\in\RRnn $ as $ b^0\left(\left[x\right]\right) $.
			
			\item When $ x\in\partial X $, the vertex $ e^0\left(\left[x\right],\xi\right) $ must be of the form $ \left[b_{\mathcal{D}_0}^0\right]\times f^0\left(\xi\right) $ and $ b^0\left(\left[x\right]\right) $ is no longer well defined (since it depends on the choice of representative for $ \left[x\right] $), but in this case, the quadratic at infinity nature of $ S $ implies that $ \widetilde{S}\left(\left[x\right],\xi\right) = \mathcal{Q}\left(f^0\left(\xi\right)\right) $.
			
		\end{itemize}
		Let $ \delta > l\left(\frac{\sqrt{2n+d}}{m} + r\right) $, we are now ready to show that $ \widetilde{V}\left(\varphi,S,K,L\right) $ and $ V\left(\varphi\right) $ are $ \delta $-interleaved. Note that \Cref{lemma:compdomforG^b} implies the p.m $ V\left(\varphi\right) $ is given by the collection
		\begin{equation*}
			V_t = H_{\ast+i}\left(E^t,\BB_{R_b}/\partial\BB_{R_b}\times\partial D^i\right) = H_{\ast+i}\left(\left(\BB_{R_b}/\partial\BB_{R_b}\times\partial D^i\right)\cup\left\{S<t\right\},\BB_{R_b}/\partial\BB_{R_b}\times\partial D^i\right)
		\end{equation*}
		where we view $ S $ as a function on $ E = \BB_{R_b}/\partial\BB_{R_b}\times\BB_{R_f} $ which by \Cref{eq:Rflarge} takes its minimum on the set $ \BB_{R_b}/\partial\BB_{R_b}\times\partial D^i $ so we set
		\begin{equation*}
			E^t = \left(\BB_{R_b}/\partial\BB_{R_b}\times\partial D^i\right)\cup\left\{\left(\left[x\right],\xi\right)\in\BB_{R_b}/\partial\BB_{R_b}\times\BB_{R_f};S\left(x,\xi\right)<t\right\},
		\end{equation*}
		and recall the p.m $ \widetilde{V}\left(\varphi,S,K,L\right) $ is given by the collection
		\begin{equation*}
			\widetilde{V}_t =  H_{\ast+i}\left(K_t,L;\ZZ_2\right) = H_{\ast+i}\left(L\cup\left\{\widetilde{S}<t\right\},L\right). 
		\end{equation*}
	
		We shall show that the maps $ h=\widetilde{h_b}\oplus h_f $ and $h^{-1}=\widetilde{h_b^{-1}}\oplus h_f^{-1} $ given by \Cref{lemma:cellcmplxhomot} induces $ \delta $-interleaving morphisms between these persistence modules. Starting with
		\begin{equation*}
			\left(K,L\right) \xrightarrow{h = \widetilde{h}_b\oplus h_f}\left(\BB_{R_b}/\partial\BB_{R_b}\times\BB_{R_f},\BB_{R_b}/\partial\BB_{R_b}\times\partial D^i\right),
		\end{equation*}
		let $ \left(\left[x\right],\xi\right)\in K $, we show that $ \left|S\left(h\left(\left[x\right],\xi\right)\right)-\widetilde{S}\left(\left[x\right],\xi\right)\right|<\delta $ by relating the value of $ \widetilde{S}\left(\left[x\right],\xi\right) $ to the value of $ S $ on the vertex $ e^0\left(\left[x\right],\xi\right) $ and then bounding the distence between this vertex and $ h\left(\left[x\right],\xi\right) $ by using the bound on $ \left\|h_b\oplus h_f-\Id\right\|_{C^0} $ and the diameter of the complex.
		
		Because of the quotient $ X/\partial X $ is involved, this argument has some technicalities to work out such as when the vertex $ e^0\left(\left[x\right],\xi\right) $ or the image $ h\left(\left[x\right],\xi\right) $ has multiple representatives in the quotient so formally we have to consider several different cases:
		\begin{enumerate}
			\item If $ x\in\Int X $, then $ \left(\left[x\right],\xi\right) = \left(x,\xi\right) $ and we have the vertex $ \left(b^0\left(x\right),f^0\left(\xi\right)\right)\in X\times Y\subset\RRnn\times\RR^d $ such that
			\begin{equation*}
				\left|\left(x,\xi\right) - \left(b^0\left(x\right),f^0\left(\xi\right)\right)\right|\leq r,
			\end{equation*}
			and $ e^0\left(x,\xi\right) = \left[b^0\left(x\right)\right]\times f^0\left(\xi\right) $. We further consider two sub-cases, according to the form of the vertex $ e^0\left(x,\xi\right) $:
			\begin{enumerate}
				\item If $ e^0\left(x,\xi\right)\in\Int X\times Y $, then $ e^0\left(x,\xi\right) = b^0\left(x\right)\times f^0\left(\xi\right) $ so we have
				\begin{equation*}
					\widetilde{S}\left(\left[x\right],\xi\right) = \widetilde{S}\left(e^0\left(x,\xi\right)\right) = S\left(b^0\left(x\right),f^0\left(\xi\right)\right),
				\end{equation*}
				and we then consider two (final) sub-cases, according to the image of $ x $ under the homotopy equivalence $ h_b:\left(X,\partial X\right)\to\left(\BB_{R_b},\partial\BB_{R_b}\right) $ given by \Cref{lemma:cellcmplxhomot}:
				\begin{enumerate}
					\item If $ x\in\Int\BB_{R_b}\subset X $, then since $ h_b $ is a homotopy rel $ \BB_{R_b} $ we have $ h_b\left(x\right) = x\in\Int\BB_{R_b} $ so $ \widetilde{h}_b\left(\left[x\right]\right) = \left[h_b\left(x\right)\right] = h_b\left(x\right) $ and so the inequality
					\begin{equation*}
						\left|h\left(x,\xi\right) - \left(x,\xi\right)\right|=\left|\left(h_b\left(x\right),h_f\left(\xi\right)\right) - \left(x,\xi\right)\right|\leq\frac{\sqrt{2n+d}}{m}
					\end{equation*}
					holds since by \Cref{lemma:cellcmplxhomot} we have $ \left\|h_b\oplus h_f-\Id\right\|_{C^0}\leq\frac{\sqrt{2n+d}}{m}  $. Thus, using the Lipshitz property of $ S $ we get
					\begin{equation*}
						\begin{split}
							\left|S\left(h\left(\left[x\right],\xi\right)\right) - \widetilde{S}\left(\left[x\right],\xi\right)\right| &= \left|S\left(h\left(x,\xi\right)\right) - S\left(b^0\left(x\right),f^0\left(\xi\right)\right)\right| \\
							&\leq l\left(\underbrace{\left|h\left(x,\xi\right)-\left(x,\xi\right)\right|}_{\leq\frac{\sqrt{2n+d}}{m}}+\underbrace{\left|\left(x,\xi\right)-\left(b^0\left(x\right),f^0\left(\xi\right)\right)\right|}_{\leq r}\right) < \delta
						\end{split}
					\end{equation*}
				
					\item If $ x\in\Int X\setminus\Int\BB_{R_b} $, then the expression $ \left|h\left(x,\xi\right)-\left(x,\xi\right)\right| $ is no loner well defined since $ \left[h_b\left(x\right)\right]\in\left[\partial\BB_{R_b}\right] $, but we do know that before taking the quotient we have
					\begin{equation*}
						\left|\left(h_b\left(x\right),h_f\left(\xi\right)\right)-\left(x,\xi\right)\right|\leq\frac{\sqrt{2n+d}}{m},
					\end{equation*}
					and since the value of $ S $ does not depend on the quotient $ \widetilde{h}_b\left(x\right) $ (rather only on $ h_b\left(x\right) $) we again see that
					\begin{equation*}
						\begin{split}
							\left|S\left(h\left(\left[x\right],\xi\right)\right) - \widetilde{S}\left(\left[x\right],\xi\right)\right| &= \left|S\left(h_b\left(x\right),h_f\left(\xi\right)\right) - S\left(b^0\left(x\right),f^0\left(\xi\right)\right)\right| \\
							&\leq l\left(\left|\left(h_b\left(x\right),h_f\left(\xi\right)\right)-\left(x,\xi\right)\right|+\left|\left(x,\xi\right)-\left(b^0\left(x\right),f^0\left(\xi\right)\right)\right|\right)<\delta.
						\end{split}
					\end{equation*}
				\end{enumerate}
			
			\item If $ e^0\left(x,\xi\right)\in\left[\partial X\right]\times Y $ we may still say that $ \widetilde{S}\left(\left[x\right],\xi\right) = S\left(b^0\left(x\right),f^0\left(\xi\right)\right) $ since $\left[b^0\left(x\right)\right] = \left[b_{\mathcal{D}_0}^0\right]$ therefore $ b^0\left(x\right)\in\partial X $ and so in particular $ \left|b^0\left(x\right)\right|\geq R_b $, we get
			\begin{equation*}
				\widetilde{S}\left(\left[x\right],\xi\right) = \widetilde{S}\left(e^0\left(x,\xi\right)\right) = \mathcal{Q}\left(f^0\left(\xi\right)\right) = S\left(b^0\left(x\right),f^0\left(\xi\right)\right),
			\end{equation*}
			and both sub-cases follow in the same way.
			\end{enumerate}
		
			\item If $ x\in\partial X $, the vertex $ b^0\left(\left[x\right]\right) $ is no longer well defined but $ f^0\left(\xi\right) $ still satisfies
			\begin{equation*}
				\left|\xi - f^0\left(\xi\right)\right|\leq r,
			\end{equation*}
			since $ f^0\left(\xi\right) $ is a vertex that belongs to the same cell as $ \xi $ in $ Y $. The vertex $ e^0\left(\left[x\right],\xi\right) $ must be of the form $ \left[b_{\mathcal{D}_0}^0\right]\times f^0\left(\xi\right) $ since $ \left(\left[x\right],\xi\right) $ is contained in the sub-complex $ \left[\partial X\right]\times Y $, so we also still have
			\begin{equation*}
				\widetilde{S}\left(\left[x\right],\xi\right) = \widetilde{S}\left(e^0\left(x,\xi\right)\right) = \mathcal{Q}\left(f^0\left(\xi\right)\right).
			\end{equation*}
			Finally $ x $ must lie in $ X\setminus\BB_{R_b} $ therefore $ h_b\left(x\right)\in\partial\BB_{R_b} $ and so it's enough to know that $ \left|h_f\left(\xi\right)-\xi\right|\leq \frac{\sqrt{d}}{m} $ to get
			\begin{equation*}
				\begin{split}
					\left|S\left(h\left(\left[x\right],\xi\right)\right) - \widetilde{S}\left(\left[x\right],\xi\right)\right| &= \left|S\left(h_b\left(x\right),h_f\left(\xi\right)\right) - \mathcal{Q}\left(f^0\left(\xi\right)\right)\right| \\
					&=\left|\mathcal{Q}\left(h_f\left(\xi\right)\right) - \mathcal{Q}\left(f^0\left(\xi\right)\right)\right| \\
					&\leq l\left(\left|h_f\left(\xi\right)-\xi\right|+\left|\xi-f^0\left(\xi\right)\right|\right)<\delta.
				\end{split}
			\end{equation*}
		\end{enumerate}
		
		So we indeed see that if $ \left(\left[x\right],\xi\right)\in K_t $ then $ h\left(\left[x\right],\xi\right)\in E^{t+\delta} $ and since the map $ h $ sends $ L $ to $ \BB_{R_b}/\partial\BB_{R_b}\times\partial D^i $ we get a well defined map
		\begin{equation*}
			\left(K_t,L\right)\xrightarrow{h}\left(E^{t+\delta},\BB_{R_b}/\partial\BB_{R_b}\times\partial D^i\right),
		\end{equation*}
		which then induces a map in relative homology $ h_*:\widetilde{V}_t\to V_{t+\delta} $, this maps clearly commutes with the morphisms $ \pi_{s+\delta,t+\delta}:V_{s+\delta}\to V_{t+\delta} $ and $ \widetilde{\pi}_{s,t}:\widetilde{V}_s\to\widetilde{V}_t $ induced by inclusions and so it defines a persistence morphism
		\begin{equation*}
			h_*:\widetilde{V}\left(\varphi,S,K,L\right)\to V\left(\varphi\right)\left[\delta\right].
		\end{equation*}
		
		Next we use a similar argument to show the map
		\begin{equation*}
			\left(\BB_{R_b}/\partial\BB_{R_b}\times\BB_{R_f},\BB_{R_b}/\partial\BB_{R_b}\times\partial D^i\right) \xrightarrow{h^{-1}=\widetilde{h}_b^{-1}\oplus h_f^{-1}}\left(K,L\right),
		\end{equation*}
		also satisfies that $ \left|S\left(\left[x\right],\xi\right)-\widetilde{S}\left(h^{-1}\left(\left[x\right],\xi\right)\right)\right| < \delta $, this is again done by comparing the value of $ \widetilde{S} $ on $ h^{-1}\left(\left[x\right],\xi\right) $ to the value of $ S $ on the vertex $ e^0\left(h^{-1}\left(\left[x\right],\xi\right)\right) $ then bounding the distance from this vertex to $ \left(\left[x\right],\xi\right) $ by as before. Recall that $ h^{-1} = \widetilde{h_b^{-1}}\oplus h_f^{-1} $ is not the inverse of $ h $ but rather the map such that $ h\circ h^{-1} \sim \Id $ and $ \widetilde{h_b^{-1}}:\BB_{R_b}/\partial\BB_{R_b}\to X/\partial X $ is the map induced on the quotient by $ h_b^{-1}:\left(\BB_{R_b},\partial\BB_{R_b}\right)\to\left(X,\partial X\right) $. We show the concrete construction as before for completeness and since there are fewer cases to consider, let $ \left(\left[x\right],\xi\right)\in\BB_{R_b}/\partial\BB_{R_b}\times\BB_{R_f} $, then:
		\begin{enumerate}
			\item If $ x\in\Int\BB_{R_b} $, then $ \left[x\right] = x $ and the point $ h_b^{-1}\left(x\right)\in X $ is well defined so like our previous notations, there exists a vertex $ b^0\left(h_b^{-1}\left(x\right)\right)\in X $ such that
			\begin{equation*}
				e^0\left(h^{-1}\left(\left[x\right],\xi\right)\right) = \left[b^0\left(h_b^{-1}\left(x\right)\right)\right]\times f^0\left(h_f^{-1}\left(\xi\right)\right),
			\end{equation*}
			and $ \widetilde{S}\left(h^{-1}\left(\left[x\right],\xi\right)\right) = \widetilde{S}\left(e^0\left(h^{-1}\left(\left[x\right],\xi\right)\right)\right) $. Moreover, since the value of $ S $ is independent of the choice different representatives for the class $ \left[b^0\left(h_b^{-1}\left(x\right)\right)\right] $ we get
			\begin{equation*}
				\begin{split}
					\left|S\left(\left[x\right],\xi\right)-\widetilde{S}\left(h^{-1}\left(\left[x\right],\xi\right)\right)\right| &= \left|S\left(x,\xi\right) - \widetilde{S}\left(e^0\left(h^{-1}\left(\left[x\right],\xi\right)\right)\right)\right| \\
					&= \left|S\left(x,\xi\right) - \widetilde{S}\left(\left[b^0\left(h_b^{-1}\left(x\right)\right)\right]\times f^0\left(h_f^{-1}\left(\xi\right)\right)\right)\right| \\
					&= \left|S\left(x,\xi\right) - S\left(b^0\left(h_b^{-1}\left(x\right)\right),f^0\left(h_f^{-1}\left(\xi\right)\right)\right)\right|,
				\end{split}
			\end{equation*}
			which is then bounded by
			\begin{equation*}
				l\left(\underbrace{\left|\left(x,\xi\right)-\left(h_b^{-1}\left(x\right),h_f^{-1}\left(\xi\right)\right)\right|}_{\leq\frac{\sqrt{2n+d}}{m}}+\underbrace{\left|\left(h_b^{-1}\left(x\right),h_f^{-1}\left(\xi\right)\right) - \left(b^0\left(h_b^{-1}\left(x\right)\right),f^0\left(h_f^{-1}\left(\xi\right)\right)\right)\right|}_{\leq r}\right) <\delta.
			\end{equation*}
		
		\item If $ x\in\partial\BB_{R_b} $, then $ S\left(\left[x\right],\xi\right) = \mathcal{Q}\left(\xi\right) $ and by the assumptions on $ h^{-1} $ we get $ h^{-1}\left(\left[x\right],\xi\right)\in\left[\partial X\right]\times Y $ therefore also $ e^0\left(h^{-1}\left(\left[x\right],\xi\right)\right)\in\left[\partial X\right]\times Y $, so $ \widetilde{S}\left(e^0\left(h^{-1}\left(\left[x\right],\xi\right)\right)\right) = \mathcal{Q}\left(f^0\left(h_f^{-1}\left(\xi\right)\right)\right) $ and we have
		\begin{equation*}
			\begin{split}
				\left|S\left(\left[x\right],\xi\right)-\widetilde{S}\left(h^{-1}\left(\left[x\right],\xi\right)\right)\right| &= \left|\mathcal{Q}\left(\xi\right) - \widetilde{S}\left(e^0\left(h^{-1}\left(\left[x\right],\xi\right)\right)\right)\right| \\
				&= \left|\mathcal{Q}\left(\xi\right) - \mathcal{Q}\left(f^0\left(h_f^{-1}\left(\xi\right)\right)\right)\right| \\
				&\leq l\left(\left|\xi-h_f^{-1}\left(\xi\right)\right|+\left|h_f^{-1}\left(\xi\right) - f^0\left(h_f^{-1}\left(\xi\right)\right)\right|\right) < \delta.
			\end{split}
		\end{equation*}
		\end{enumerate}
		
		So, we see that if $ \left(\left[x\right],\xi\right)\in E^t $ then $ h^{-1}\left(\left[x\right],\xi\right)\in K^{t+\delta} $ and since the map $ h^{-1} $ sends $ \BB_{R_b}/\partial\BB_{R_b}\times\partial D^i $ to $ L $ we get a well defined map
		\begin{equation*}
			\left(E^{t},\BB_{R_b}/\partial\BB_{R_b}\times\partial D^i\right)\xrightarrow{h^{-1}}\left(K_{t+\delta},L\right),
		\end{equation*}
		which induces a map in relative homology $ h^{-1}_*:V_t\to\widetilde{V}_{t+\delta} $, where as before, this map commutes with the morphisms induced by inclusions and therefore defines a persistence morphism
		\begin{equation*}
			h_*^{-1}:V\left(\varphi\right)\to\widetilde{V}\left(\varphi,S,K,L\right)\left[\delta\right].
		\end{equation*}
	
		Finally, we claim that the morphisms $ h_*,h_*^{-1} $ are indeed $ \delta $-interleaving morphisms, as the following diagrams show:
		\begin{equation*}
			\xymatrix@C=2pc
			{
				V_t \ar[r]^-{h_*^{-1}} \ar@/_1pc/[rr]_-{\left(in\right)_*} & \widetilde{V}_{t+\delta} \ar[r]^-{h_*} & V_{t+2\delta} && \widetilde{V}_t \ar[r]^-{h_*} \ar@/_1pc/[rr]_-{\left(in\right)_*} & V_{t+\delta} \ar[r]^-{h_*^{-1}} & \widetilde{V}_{t+2\delta}
			},
		\end{equation*}
		since $ h\circ h^{-1} \sim \Id $ and $ h^{-1}\circ h\sim \Id $, they induce isomorphism in homology and therefore we indeed get commutativity in the above diagrams, which concludes the proof.
	\end{proof}
	
	We then conclude the algorithm by gathering all boundary matrices $ \partial_j $ into one matrix $ \partial $ ordered by degree, then using reduction algorithms, such as the standard reduction described earlier (see \cite{Phat17} for an implementation of several such reduction algorithms), we get a boundary matrix $ R $ which is a reduced form of $ \partial $ and the barcode $ \mathcal{B} $ obtained from $ R $ as in \Cref{lemma:barcodepairing} is then the output of the algorithm. Note that by the isometry theorem (\Cref{thm:isometrythm}), we have
	\begin{equation*}
		\dbot\left(\mathcal{B},\mathcal{B}\left(\varphi\right)\right) = 	\dint\left(\widetilde{V}\left(\varphi,S,K,L\right),V\left(\varphi\right)\right).
	\end{equation*}
	Thus, according to \Cref{lemma:interleavingpm} we have $ \dbot\left(\mathcal{B},\mathcal{B}\left(\varphi\right)\right)\leq l\left(\frac{\sqrt{2nN}}{m}+r\right) $ and we turn to bound this distance. From now on, we denote by $ C $ constants independent of $ N,T $ and $ m $, which may not all be the same. 
	\begin{claim}
		$ \max_{\left(q_N,p_0,\xi\right)\in\BB_{R_b}\times\BB_{R_f}}\left|\nabla S\left(q_N,p_0,\xi\right)\right| \leq C\cdot TN^{\frac{3}{2}} $ for a constant $ C $ which depends only on $ R $.
	\end{claim}
	\begin{proof}
		Using \Cref{eq:defGFQI} we see that
		\begin{equation*}
			S\left(q_N,p_0,\xi\right) = \rho\left(\left|\xi\right|\right)\left(S^\prime\left(q_N,p_0,\xi\right)-\mathcal{Q}\left(\xi\right)\right) + \mathcal{Q}\left(\xi\right),
		\end{equation*}
		therefore, since
		\begin{equation*}
			S^\prime\left(q_N,p_0;\xi\right) - \mathcal{Q}\left(\xi\right) = \sum_{j=1}^NS_j\left(q_N+\sum_{k=j}^N\xi^-_k-\sum_{k=j}^N\xi^+_k,p_0+\xi^-_{j-1}+\xi^+_{j-1}\right),
		\end{equation*}
		we differentiate all $ S_j\left(q_N+\sum_{k=j}^N\xi^-_k-\sum_{k=j}^N\xi^+_k,p_0+\xi^-_{j-1}+\xi^+_{j-1}\right) $ to get
		\begin{equation*}
			\left|\nabla\left(S^\prime\left(q_N,p_0,\xi\right)-\mathcal{Q}\left(\xi\right)\right)\right| \leq \sum_{j=1}^N\max\left|\nabla S_j\right|\sqrt{2\left(N-j+2\right)},
		\end{equation*}
		and using $ \left|\rho^\prime\right|\leq\frac{1}{M} $ we see that
		\begin{equation*}
			\left|\nabla S\left(q_N,p_0,\xi\right)\right| \leq \frac{1}{M}\sum_{j=1}^N \max\left|S_j\right| + \sum_{j=1}^N\max\left|\nabla S_j\right|\sqrt{2\left(N-j+2\right)} + \max_{\xi\in\BB_{R_f}}\left|\nabla\mathcal{Q}\left(\xi\right)\right|,
		\end{equation*}
		$ S_j $ generates $ \varphi_j $ so we know that $ \left|\nabla S_j\right| = \left|\left(P_{\varphi_j} - p,q - Q_{\varphi_j}\right)\right| < \left\|\varphi_j - \Id\right\|_{C^0} \leq T $ and $ \mathcal{Q} $ is a diagonalized non-degenerate quadratic form, which means that for $ \left(q_N,p_0,\xi\right)\in\BB_{R_b}\times\BB_{R_f} $
		\begin{equation*}
			\left|\nabla S\left(q_N,p_0,\xi\right)\right| \leq \frac{NTR}{M} + CTN^{\frac{3}{2}} + 2R_f.
		\end{equation*}
		We bound $ M $ and $ R_f $ by
		\begin{equation} \label{eq:bndonradii}
			\begin{split}
				2\sqrt{2}T\left(N-1\right)\leq M &= \sqrt{2}T\left(\left(N-1\right)+\sum_{j=1}^{N-1}\sqrt{j}\right) \leq C\cdot TN^{\frac{3}{2}} \\
				R_f^- &= \frac{NTR}{M} + M \leq C\cdot TN^{\frac{3}{2}} \implies R_f^+ = R_f^- + M \leq C\cdot TN^{\frac{3}{2}} \\
				R_f &= \sqrt{\left(R_f^+\right)^2-\sum_{j=1}^N\min S_j} \leq \sqrt{\left(R_f^+\right)^2+NTR} \leq C\cdot TN^{\frac{3}{2}},
			\end{split}
		\end{equation}
		and plugging in these bounds gives
		\begin{equation*}
			\left|\nabla S\left(q_N,p_0,\xi\right)\right| \leq \frac{R}{C} + CTN^{\frac{3}{2}} + 2CTN^{\frac{3}{2}} \leq CTN^{\frac{3}{2}}
		\end{equation*}
		which proves the claim.
	\end{proof}
	\begin{remark}
		For completeness, we state that calculating the constants, $ C $ may be given by
		\begin{equation} \label{eq:CofR}
			C\left(R\right) = \frac{R}{2}+\sqrt{3} + 2^{\nicefrac{3}{4}}\sqrt{5+\left(1+\sqrt{2}\right)R}.
		\end{equation}
	\end{remark}
	Finally, $ r $ is given by \Cref{lemma:cellcmplxhomot} as $ \frac{\sqrt{2nN}}{m} $ so we have
	\begin{equation*} 
		\dbot\left(\mathcal{B},\mathcal{B}\left(\varphi\right)\right)\leq l\left(\frac{\sqrt{2nN}}{m}+r\right) \leq CTN^\frac{3}{2}\left(\frac{\sqrt{2nN}}{m}+\frac{\sqrt{2nN}}{m}\right) \leq C\frac{\sqrt{n}TN^2}{m}.
	\end{equation*}
	proving the desired bottleneck distance bound.
	
	\subsection{Complexity analysis}
	We move on to analyze the time complexity of the suggested algorithm. We offer a rather basic complexity analysis which bounds the count of operations needed for each step of the algorithm.
	
	Starting from the last step of the algorithm, reduction of the boundary matrix $ \partial $ in \textbf{step 3} has a worst case running time of $ \mathcal{C}^3 $, where $ \mathcal{C} $ is the number of cells in $ K\setminus L $ which corresponds to the number of columns in the matrix $ \partial $ (see \cite{Phat17}). We estimate the number of cells in the following claim:
	\begin{claim}
		The number of cells in the complex $ K\setminus L $, denoted by $ \mathcal{C} $, satisfies
		\begin{equation*}
			\mathcal{C} \leq C\cdot m^{4n}T^{4n}N^{7n}4^{nN}.
		\end{equation*}
	\end{claim}
	\begin{proof}
		In our notations $ \mathcal{C} $ is given by
		\begin{equation*}
			\mathcal{C} = \sum_{j=0}^{2nN}\sum_{k=0}^j\left|\overline{\mathcal{B}}_k\right|\cdot\left|\overline{\mathcal{F}}_{j-k}\right|.
		\end{equation*}
		Note that $ \left|\mathcal{B}_0\right| $ and $ \left|\mathcal{F}_0\right| $ is the number of points in $ X^0= \frac{1}{m}\ZZ^{2n}\cap \BB_{R_b+\frac{\sqrt{2n}}{m}} $ and $ Y^0 = \frac{1}{m}\ZZ^{2n\left(N-1\right)}\cap \BB_{R_f+\frac{\sqrt{2n\left(N-1\right)}}{m}} $ respectively and using \Cref{eq:bndonradii} we see the radius $ R_b $ satisfies
		\begin{equation*}
			R_b \leq C\cdot \sqrt{n}TN^2.
		\end{equation*}
		The number of integer points inside a $ d $-dimensional ball is estimated by the ball's volume, so using the classical estimate for the volume of $ \BB_r\subset\RR^d $ given by $ V_d\left(r\right) \sim \frac{1}{\sqrt{d\pi}}\left(\frac{2\pi e}{d}\right)^{\frac{d}{2}}r^d  $ we have
		\begin{equation*}
			\begin{split}
				&\left|\mathcal{B}_0\right| = \left|\ZZ^{2n}\cap \BB_{mR_b +\sqrt{2n}}\right| \leq C\cdot \left(2n\right)^{-n}\left(mR_b+\sqrt{2n}\right)^{2n} \leq C\cdot m^{2n}T^{2n}N^{4n} \\
				&\begin{split}
					\left|\mathcal{F}_0\right| = \left|\ZZ^{2n\left(N-1\right)}\cap \BB_{mR_f + \sqrt{2n\left(N-1\right)}}\right| &\leq C\cdot \left(2n\left(N-1\right)\right)^{-n\left(N-1\right)}\left(mR_f+\sqrt{2n\left(N-1\right)}\right)^{2n\left(N-1\right)}\\
					&\leq C\cdot m^{2n\left(N-1\right)}T^{2n\left(N-1\right)}N^{3n\left(N-1\right)}.
				\end{split}
			\end{split}
		\end{equation*}
		Next, the number of $ k $-cells in the full cubical complex generated by the vertices in $ X^0 $ (or $ Y^0 $) may be estimated by considering each $ k $-cell as a vertex in $ X^0 $ (or $ Y^0 $) together with a choice of $ k $ (unsigned) directions out of $ 2n $ (or $ 2n\left(N-1\right) $) possible choices, some of these cells might not be contained inside $ \BB_{R_b+\frac{\sqrt{2n}}{m}} $ (or $ \BB_{R_f+\frac{\sqrt{2n\left(N-1\right)}}{m}} $) and so wont be $ k $-cells in $ X $ (or $ Y $), but we get a close upper bound by
		\begin{equation*}
			\left|\mathcal{B}_k\right| \leq \left|\mathcal{B}_0\right|\binom{2n}{k} \qquad \left|\mathcal{F}_k\right| \leq \left|\mathcal{F}_0\right|\binom{2n\left(N-1\right)}{k}.
		\end{equation*}
		The sets $ \overline{\mathcal{B}}_k $ are then the indices of $ k $-cells in the quotient $ X/\partial X $ which gives the same asymptotic bound on $ \left|\overline{\mathcal{B}}_k\right| $ as $ \left|\mathcal{B}_k\right| $ and the sets $ \overline{\mathcal{F}_k} = \mathcal{F}_k\setminus\widetilde{\mathcal{F}}_k $ are the indices of all $ k $-cells in $ Y\setminus Y_0 $ where $ Y_0\subset\partial Y $, therefore we also get the same asymptotic bound for $ \left|\overline{\mathcal{F}}_k\right| $ as $ \left|\mathcal{F}_k\right| $. We conclude that the number of cells in $ K\setminus L $ satisfies
		\begin{equation} \label{eq:numofcells}
			\mathcal{C} = \left|\overline{\mathcal{B}}_0\right|\left|\overline{\mathcal{F}}_0\right|\sum_{j=0}^{2nN}\sum_{k=0}^j\binom{2n}{k}\binom{2n\left(N-1\right)}{j-k} \leq C\cdot m^{2nN}T^{2nN}N^{n}\left(\sqrt[3]{4}N\right)^{3nN},
		\end{equation}
		which proves the claim.
	\end{proof}
	Thus, the number of operations needed for \textbf{step 3} is of order $ \mathcal{O}\left(m^{6nN}T^{6nN}N^{3n}64^{N}N^{9nN}\right) $, which as we shall explain next, is the dominant term in the complexity of the algorithm.
	
	In \textbf{step 2}, sampling the generating function $ S $ on all cells recursively requires $ \mathcal{O}\left(\mathcal{C}\right) $ operations  since we need to iterate over all $ 0<j $-cells and maximize the value over the boundary of each cell, i.e we need (an upper bound of) $ 2j $ operations for each cell which gives
	\begin{equation*}
		\sum_{j=1}^{2nN}\left(\sum_{k=0}^j\left|\overline{\mathcal{B}}_k\right|\cdot\left|\overline{\mathcal{F}}_{j-k}\right|\right)2j = \left|\overline{\mathcal{B}}_0\right|\left|\overline{\mathcal{F}}_0\right|2nN2^{2nN} \leq C\cdot m^{2nN}T^{2nN}N^n\left(\sqrt[3]{4}N\right)^{3nN} = \mathcal{O}\left(\mathcal{C}\right).
	\end{equation*}\
	Then, sorting the columns of the boundary matrix $ \partial $ requires $ \mathcal{C}\log\mathcal{C} $ operations (via typical sorting algorithms), so we see that the complexity of \textbf{step 3} indeed dominates that of \textbf{step 2}.
	
	Finally, we estimate the number of operations required in \textbf{step 1} in order to construct the boundary matrix $ \partial $:
	\begin{enumerate}
		\item $ \sim\left(2mR_b+2\sqrt{2n}\right)^{2n} $ operations to determine the vertices inside $ X^0 $ out of all possible vertices in $ \frac{1}{m}\ZZ^{2n}\cap\left[R_b-\frac{\sqrt{2n}}{m},R_b+\frac{\sqrt{2n}}{m}\right]^{2n} $, then another
		\begin{equation*}
			\sum_{k=0}^{2n}\sum_{\alpha\in\mathcal{B}_k}2k \leq 2\left|\mathcal{B}_0\right|\sum_{k=0}^{2n}k\binom{2n}{k} = \left|\mathcal{B}_0\right|n2^{2n+1}
		\end{equation*}
		operation to construct the boundary matrices $ \partial_k^X $ since each $ k $-cell has $ 2k $ boundary terms. So, computation of all $ \partial_k^X $ requires $ \mathcal{O}\left(m^{2n}T^{2n}N^{4n}\right) $ operations.
		
		\item The construction of quotient matrices $ q_k $ require $ \left|\mathcal{B}_k\right| $ operations each to determine weather each cell is contained in $ \partial X $ or not, then the matrix $ \overline{\partial}_k^X = q_{k-1}\partial_k^X $ is computed by matrix multiplication which requires $ \mathcal{O}\left(\left|\mathcal{B}_k\right|\left|\mathcal{B}_{k-1}\right|^2\right) $ operations (using naive matrix multiplication algorithms), since the matrix $ \partial_k $ has $ \left|\mathcal{B}_k\right| $ columns and $ \left|\mathcal{B}_{k-1}\right| $ rows and the number of rows in $ q_{k-1} $ equals the number of cells in $ X\setminus\partial X $, which is bounded by $ \left|\mathcal{B}_{k-1}\right| $. So, computation of all $ \overline{\partial}_k^X $ requires $ \mathcal{O}\left(m^{6n}T^{6n}N^{12n}\right) $ operations.
		
		\item Following the same argument as in the first phase, determining the vertices in $ Y^0\subset\BB_{R_b+\frac{\sqrt{2n\left(N-1\right)}}{m}} $ requires $ \sim\left(2mR_f+2\sqrt{2n\left(N-1\right)}\right)^{2n\left(N-1\right)} $ operations, then another
		\begin{equation*}
			\sum_{k=0}^{2n\left(N-1\right)}\sum_{\alpha\in\mathcal{F}_k}2k \leq\left|\mathcal{F}_0\right|nN4^{nN}
		\end{equation*}
		operations to construct the boundary matrices $ \partial_k^Y $. So, computation of al matrices $ \partial_k^Y $ requires $ \mathcal{O}\left(m^{2n\left(N-1\right)}T^{2n\left(N-1\right)}\left(\sqrt[3]{4}N\right)^{3nN}\right) $ operations.
		
		\item The restriction $ \widetilde{\partial}_k^Y $ of $ \partial_k^Y $ to the cells contained in $ Y\setminus Y_0 $ requires $ \mathcal{O}\left(\sum_{k=0}^{2n\left(N-1\right)}\left|\mathcal{F}_k\right|\right) $ operations to determine which cells are in $ Y\setminus Y_0 $ and eliminate these columns and rows from each $ \partial_k^Y $. So, computation of all $ \widetilde{\partial}_k^Y $ requires $ \mathcal{O}\left(m^{2n\left(N-1\right)}T^{2n\left(N-1\right)}\left(\sqrt[3]{4}N\right)^{3n\left(N-1\right)}\right) $ operations.
		
		\item The boundary matrices $ \partial_j $ are computed by blocks where each block is given by $ \Id_{\overline{\mathcal{B}}_k}\otimes\widetilde{\partial}_{j-k}^Y $ and $ \overline{\partial}_{k+1}^X\otimes\Id_{\overline{\mathcal{F}}_j-k-1} $ where $ k=0\dots,j-1 $, so using basic complexity estimates for the Kronecker product of matrices (given by $ \mathcal{O}\left(n_1m_1n_2m_2\right) $ for two matrices of size $ n_1\times m_1 $ and $ n_2\times m_2 $) we get an upper bound on the number of needed operations by
		\begin{equation*}
			\left|\mathcal{B}_0\right|^2\left|\mathcal{F}_0\right|^2\sum_{k=0}^{j-1}\binom{2n}{k}^2\binom{2n\left(N-1\right)}{j-k}\binom{2n\left(N-1\right)}{j-k-1} + \binom{2n}{k+1}\binom{2n}{k}\binom{2n\left(N-1\right)}{j-k-1}^2,
		\end{equation*}
		which means we may bound (rather wastefully, but as we see it is still not the main term in the overall complexity) the number of operations needed to compute all $ \partial_j $ by
		\begin{equation*}
			C\cdot\left|\mathcal{B}_0\right|^2\left|\mathcal{F}_0\right|^2\sum_{j=0}^{2nN}\binom{2nN}{j}^2 = C\cdot m^{4nN}T^{4nN}N^{2n}N^{6nN}\binom{4nN}{2nN}.
		\end{equation*}
		So, we conclude that computation of $ \partial $ requires $ \mathcal{O}\left(m^{4nN}T^{4nN}N^{2n}\left(\sqrt[3]{4}N\right)^{6nN}\right) $ operations.
	\end{enumerate}
	We add up all three steps of the algorithm to see the main time complexity term is given in \textbf{step 3}, as the boundary matrix reduction, which has time complexity bound of
	\begin{equation*}
		\mathcal{O}\left(m^{6nN}T^{6nN}N^{3n}64^NN^{9nN}\right) = \mathcal{O}\left(m^{6nN}T^{6nN}N^{10nN}\right),
	\end{equation*}
	proving the desired complexity bound.
	
	\begin{remark}
		As the complexity is super exponential in the number $ N $ of $C^1$-small diffemorphisms in the composition, one needs to make vast improvements to the suggested implementation for large $ N $. We believe there is some room for that as the number of cells in the cell complex used by our implementation is undoubtedly larger than needed, and one possible direction may be to reduce the number of cells significantly in such a way that the topology of the sub level sets would not change, for example by taking different sizes of cells depending on the features of the generating functions (e.g smaller cells near critical points and larger cells away from them), but this is outside the scope of this paper.
		
		Moreover, our current implementation has heavy memory requirements that should be traded for prolonged running time, this will allow us to compute more examples though we believe that without further improvements this issue alone will not make a significant change on the type of examples we could compute.
	\end{remark}
	
	\section{Implementation for 2-dimensional radial functions} \label{sec:radialprofile}
	The suggested algorithm in the proof of \Cref{thm:mainthm} assumes the input is given as samples of functions $ S_j:\RRnn\to\RR $, each generating some $ \varphi_j\in\mathcal{H}_{R,T} $. While this result allows us to view the GF-Barcode as an invariant which can be numerically approximated, it is still not quite useful if one wishes to approximate GF-Barcodes of specific Hamiltonian diffeomorphisms since finding and computing samples of such generating functions is not explicit.
	
	In this section we look at a certain class of Hamiltonian diffeomorphisms of $ \RR^2 $, namely compositions of time-$ 1 $ maps of Hamiltonian flows generated by autonomouss radial Hamiltonian functions given as a radial profile and a center point, i.e
	\begin{equation*}
		\mathcal{R}_{T,T^\prime,R}^N = \left\{\varphi\in\Ham_c\left(\RR^2,\omega_{0}\right);\varphi = \varphi_N\circ\dots\circ\varphi_1,\forall j:\varphi_j\in\mathcal{R}_{T,T^\prime,R}\right\},
	\end{equation*}
	where
	\begin{equation*}
		\mathcal{R}_{T,T^\prime,R} = \left\{\varphi_H^1;H\left(x\right) = h\left(\frac{\left|x-c\right|^2}{2}\right),c\in\RR^2, h\in C^\infty_c\left(\RR_{\geq0}\right),\left|h^\prime\right|\leq T,\left|h^{\prime\prime}\right|\leq T^\prime,\supp\left(\varphi_H^1\right)\subset\BB_R\right\}.
	\end{equation*}

	For this class of Hamiltonian diffeomorphisms we offer an implementation that numerically approximate generating functions without auxiliary variables and then uses the suggested algorithm from \Cref{thm:mainthm} to approximate their GF-Barcode.
	
	\subsection{Hamiltonian diffeomorphism from radial profile}
	\paragraph[radial profiles]{}
	Let $ h\in C^\infty_c\left(\RR_{\geq0}\right) $ such that $ h $ is supported on $ \left[0,r\right] $, $ \left|h^\prime\right|\leq T,\left|h^{\prime\prime}\right|\leq T^\prime $ and let $ c\in\RR^2 $, define the Hamiltonian function $ H\left(x\right) = h\left(\nicefrac{\left|x-c\right|^2}{2}\right) $, then the Hamiltonian flow generated by $ H $ is given by
	\begin{equation} \label{eq:radialflow}
		\varphi^t_H\left(x\right) = R_{h^\prime\left(\nicefrac{\left\|x-c\right\|^2}{2}\right)t}\left(x-c\right) + c,
	\end{equation}
	where $ R_\theta:\RR^2\to\RR^2 $ denotes the rotation map by angle $ \theta $. We can bound the $ C^0 $ and $ C^1 $ distance of $ \varphi=\varphi_H^1 $ from the identity using $ T,T^\prime,r $ and $ \left|c\right| $ by
	\begin{equation} \label{eq:distfromIDbounds}
		\begin{split}
			\left\|\varphi-\Id\right\|_{C^0} &= \sup_{x\in\RR^2}\left|\varphi\left(x\right)-x\right| = 2\left|x-c\right|\sin\left(\nicefrac{1}{2}\cdot h^\prime\left(\nicefrac{\left|x-c\right|^2}{2}\right)\right) \leq \sqrt{2r}T \\
			\left\|D\varphi-\Id\right\|_{\text{op}} &= \sup_{x\in\RR^2}\sup_{v\in\SS^1}\left|D_x\varphi\left(v\right)-v\right| \leq \sqrt{2r}\left(\left|c\right|+\sqrt{2r}\right) T^\prime + T,
		\end{split}
	\end{equation}
	and in case $ \varphi $ has a generating function $ S:\RR^2\to\RR $ we can use the explicit expression of $ \varphi_H^t $ to approximate $ S $.
	
	\paragraph[Small ham to gen' func']{}
	Recall that if $ \varphi $ is sufficiently close to the identity in the $ C^1 $ topology then it has a generating function without auxiliary variables, in particular $ \left\|D\varphi - \Id\right\|_{\text{op}} $ is small and so the function $ q\mapsto Q_\varphi\left(q,p\right) $ is invertible for every $ p\in\RRn $ (since $ \frac{\partial Q_\varphi}{\partial q} $ is non-singular), we denote its inverse by $ q_\varphi\left(Q,p\right) $ and we get the relation $ Q_\varphi\left(q_\varphi\left(Q,p\right),p\right) = Q $ which allows us to use the coordinates $ \left(Q,p\right) $ instead of $ \left(q,p\right) $ - the existence of generating function for such $ \varphi $ follows from this observation and is described in the following proposition (\cite[lemma 9.2.1]{McDSal17}).
	\begin{proposition} \label{prop:smallhamtogenfunc}
		Let $ \varphi\in\Ham\left(\RRnn,\omega_0\right) $ such that $ \left\|D_x\varphi - \Id\right\|_{\text{op}} \leq \frac{1}{2} $ for all $ x\in\RRnn $, then there exists a function $ S:\RRnn\to\RR $ such that $ L_\varphi = \Gamma_{dS} $. 
	\end{proposition}

	In particular, given a such profile function $ h $ described earlier, if $\sqrt{2r}\left(\left|c\right|+\sqrt{2r}\right) T^\prime + T\leq\frac{1}{2} $ then the corresponding $ \varphi $ indeed has a generating function $ S:\RR^2\to\RR $.

	\subsection{Approximation of generating functions}
	\paragraph[Hamilton Jacobi for small gen' func']{}
	This special case of generating functions for $ C^1 $ small Hamiltonian diffeomorphisms also satisfies the Hamilton-Jacobi equation - this is a classical result that appears in many places (see \cite[prop' 3.1A]{BialPol94} for exmaple) and we state it here in order to obtain an expression for such generating functions.
	\begin{proposition}[Hamilton-Jacobi] \label{prop:H-J}
		Let $ \varphi^t\in\Ham_c\left(\RRnn,\omega_0\right) $ be a Hamiltonian flow starting at the identity and generated by the Hamiltonian function $ H_t:\RRnn\to\RR $. If $ \left\{S_t\right\} $ is a smooth family of generating functions with $ S_0 = 0 $ such that $ S_t:\RRnn\to\RR $ generates $ \varphi^t $ (in the sense of \Cref{def:smallgenfunc}), then $ S_t $ satisfies Hamilton-Jacobi equation
		\begin{equation*}
			\frac{\diff S_t}{\diff t}\left(Q,p\right) = H_t\left(Q,\frac{\partial S_t}{\partial Q}\left(Q,p\right) + p\right).
		\end{equation*}
	\end{proposition}

	\paragraph[Approximating gf]{}
	Thus, if $ \varphi $ is the time-$ 1 $ map of a Hamiltonian flow $ \varphi^t_H $ generated by $ H_t:I\times\RRnn\to\RR $ then its generating function $ S $ is given by the following.
	\begin{equation} \label{eq:genfuncexpression}
		S\left(Q,p\right) = \int_0^1 H_t\left(Q,P^t\left(q_{\varphi^t}\left(Q,p\right),p\right)\right)\diff t,
	\end{equation}
	where $ \varphi^t_H\left(q,p\right) = \left(Q^t,P^t\right) $ and $ q_{\varphi^t}\left(Q,p\right) $ satisfies $ Q^t\left(q_{\varphi^t}\left(Q,p\right),p\right) = Q $.
	
	Note that for each $ p\in\RRn $ the function $ q_{\varphi^t}\left(\cdot,p\right) $ is the inverse of $ Q^t\left(\cdot,p\right) $ and so
	\begin{equation*}
		\left|q_{\varphi^t}\left(Q,p\right) - Q\right| \leq \left\|\varphi^t-\Id\right\|_{C^0}
	\end{equation*}
	which means that given $ \varphi^t $ and $ Q\in\RRn,p\in\RRn $, one can numerically solve the inverse function problem
	\begin{equation*}
		Q^t\left(\widetilde{q_{\varphi^t}},p\right) = Q,
	\end{equation*}
	by minimizing the function $ \left|Q^t\left(\widetilde{q_{\varphi^t}},p\right) - Q\right|^2 $ in a ball of radius $ \left\|\varphi^t-\Id\right\|_{C^0} $ around $ Q $ with a maximal error denoted $ E $, so we get an approximation $ \widetilde{q_{\varphi^t}}\left(Q,p\right) $ satisfying
	\begin{equation*}
		\left|Q^t\left(\widetilde{q_{\varphi^t}}\left(Q,p\right),p\right) - Q\right| \leq \sqrt{E},
	\end{equation*}
	Finally, we can formalize an approximation of $ S $ in our case by:
	\begin{lemma} \label{Lemma:approxgf}
		Let $ \varphi\in\Ham_c\left(\RRnn,\omega_{0}\right) $ such that $ \max\left\|D\varphi-\Id\right\|_{\text{op}}<1 $. Assume $ \varphi $ has a generating function $ S:\RRnn\to\RR $ and that the Hamiltonian path $ \varphi^t $ is generated by an autonomous Hamiltonian $ H:\RRnn\to\RR $, for $ m\in\NN $ denote by $ \widetilde{S} $ the function
		\begin{equation*}
			\widetilde{S}\left(Q,p\right) = \frac{1}{m}\sum_{k=1}^m H\left(Q,P^{\nicefrac{k}{m}}\left(\widetilde{q_{\varphi^{\nicefrac{k}{m}}}}\left(Q,p\right),p\right)\right),
		\end{equation*}
	where $ \varphi^t\left(q,p\right) = \left(Q^t\left(q,p\right),P^t\left(q,p\right)\right) $ and $ \widetilde{q_{\varphi^t}}\left(Q,p\right) $ satisfies $ \left|Q^t\left(\widetilde{q_{\varphi^t}}\left(Q,p\right),p\right) - Q\right| \leq \sqrt{E} $, then
	\begin{equation*}
		\left\|S-\widetilde{S}\right\|_{C^0}\leq \frac{2\max\left|\nabla H\right|^2}{1-\max\left\|D\varphi-\Id\right\|_\text{op}}\cdot\frac{1}{m} + \frac{2\max\left|\nabla H\right|}{1-\max\left\|D\varphi-\Id\right\|_{\text{op}}}\sqrt{E}.
	\end{equation*}
	\end{lemma}
	
	\begin{proof}
		Let $ Q\in\RRn,p\in\RRn $, we use the expression of the generating function $ S $ given by \Cref{eq:genfuncexpression} to bound the distance $ \left|S\left(Q,p\right)-\widetilde{S}\left(Q,p\right)\right| $ by
		\begin{align*}
			 \left|S\left(Q,p\right)-\widetilde{S}\left(Q,p\right)\right| &= \left|\sum_{k=1}^m\int_{\frac{k-1}{m}}^{\frac{k}{m}}H\left(Q,P^t\left(q_{\varphi^t}\left(Q,p\right),p\right)\right)\diff t - \frac{1}{m}\sum_{k=1}^m H\left(Q,P^{\nicefrac{k}{m}}\left(\widetilde{q_{\varphi^{\nicefrac{k}{m}}}}\left(Q,p\right),p\right)\right)\right|\\
			 &\leq\sum_{k=1}^m\int_{\frac{k-1}{m}}^{\frac{k}{m}}\left|H\left(\varphi^t\left(q_{\varphi^t}\left(Q,p\right)\right),p\right)-H\left(Q,P^{\nicefrac{k}{m}}\left(\widetilde{q_{\varphi^{\nicefrac{k}{m}}}}\left(Q,p\right),p\right)\right)\right|\diff t \\
			 &\leq\sum_{k=1}^m\int_{\frac{k-1}{m}}^{\frac{k}{m}}\max\left|\nabla H\right|\max_{t\in\left[\frac{k-1}{m},\frac{k}{m}\right]}\underbrace{\left|\varphi^t\left(q_{\varphi^t}\left(Q,p\right),p\right)-\left(Q,P^{\nicefrac{k}{m}}\left(\widetilde{q_{\varphi^{\nicefrac{k}{m}}}}\left(Q,p\right),p\right)\right)\right|}_{\star}\diff t,
		\end{align*}
		and
		\begin{align*}
			\star&\leq \underbrace{\left|\varphi^t\left(q_{\varphi^t},p\right)-\varphi^{\nicefrac{k}{m}}\left(q_{\varphi^{\nicefrac{k}{m}}},p\right)\right|}_{\star\star} + \underbrace{\left|\varphi^{\nicefrac{k}{m}}\left(q_{\varphi^{\nicefrac{k}{m}}},p\right)-\left(Q,P^{\nicefrac{k}{m}}\left(\widetilde{q_{\varphi^{\nicefrac{k}{m}}}},p\right)\right)\right|}_{\star\star\star} \\\\
			\star\star&\leq\left|\varphi^t\left(q_{\varphi^t},p\right)-\varphi^{t}\left(q_{\varphi^{\nicefrac{k}{m}}},p\right)\right| + \left|\varphi^t\left(q_{\varphi^{\nicefrac{k}{m}}},p\right)-\varphi^{\nicefrac{k}{m}}\left(q_{\varphi^{\nicefrac{k}{m}}},p\right)\right| \\
			&\leq\max\left\|D\varphi^t\right\|_{\text{op}}\left|q_{\varphi^t}-q_{\varphi^{\nicefrac{k}{m}}}\right| + \max_{t\in\left[\frac{k-1}{m},\frac{k}{m}\right]}\left|\frac{\diff}{\diff t}\varphi^t\right|\frac{1}{m} \\
			&\leq \max\left\|D\varphi^t\right\|_{\text{op}}\max_{t\in\left[\frac{k-1}{m},\frac{k}{m}\right]}\left|\frac{\diff}{\diff t}q_{\varphi^t\left(Q,p\right)}\right|\frac{1}{m} + \max\left|\nabla H\right|\frac{1}{m}.
		\end{align*}
		Next, we differentiate $ Q^t\left(q_{\varphi^t}\left(Q,p\right),p\right) = Q $ w.r.t $ t $ to get
		\begin{equation*}
			\frac{\partial H}{\partial p}\left(\varphi^t\left(q_{\varphi^t}\left(Q,p\right),p\right)\right) = \frac{\partial Q^t}{\partial q}\left(q_{\varphi^t}\left(Q,p\right),p\right)\cdot\frac{\diff}{\diff t}q_{\varphi^t}\left(Q,p\right),
		\end{equation*}
		so
		\begin{equation*}
			\max_{t\in\left[\frac{k-1}{m},\frac{k}{m}\right]}\left|\frac{\diff}{\diff t}q_{\varphi^t\left(Q,p\right)}\right| \leq \max\left\|\left(\frac{\partial Q}{\partial q}\right)^{-1}\right\|_\text{op}\max\left|\nabla H\right|,
		\end{equation*}
		and since $ \left\|D\varphi-\Id\right\|_\text{op} < 1 $, the restriction also satisfies $ \left\|\frac{\partial Q}{\partial q}-\Id\right\|_\text{op} < 1 $ therefore
		\begin{equation*}
			\left\|\left(\frac{\partial Q}{\partial q}\right)^{-1}\right\|_\text{op} \leq \frac{1}{1-\left\|\frac{\partial Q}{\partial q}-\Id\right\|_\text{op}} \leq \frac{1}{1-\left\|D\varphi-\Id\right\|_\text{op}},
		\end{equation*}
		and we get
		\begin{align*}
			\star\star\leq\frac{2\max\left|\nabla H\right|}{1-\max\left\|D\varphi-\Id\right\|_\text{op}}\cdot\frac{1}{m}.
		\end{align*}
		Finally, since $ \widetilde{q_{\varphi^t}}\left(Q,p\right) $ satisfies $ \left|Q^t\left(\widetilde{q_{\varphi^t}}\left(Q,p\right),p\right) - Q\right| \leq \sqrt{E} $ and $ q_{\varphi^t}\left(\cdot,p\right) = \left(Q^t\left(\cdot,p\right)\right)^{-1} $  we get
		\begin{align*}
			\star\star\star&\leq\left|Q^{\nicefrac{k}{m}}\left(q_{\varphi^{\nicefrac{k}{m}}}\left(Q,p\right),p\right)-Q\right| + \left|P^{\nicefrac{k}{m}}\left(q_{\varphi^{\nicefrac{k}{m}}},p\right)-P^{\nicefrac{k}{m}}\left(\widetilde{q_{\varphi^{\nicefrac{k}{m}}}},p\right)\right| \\
			&\leq \sqrt{E} + \max\left\|DP^t\right\|_\text{op}\left|q_{\varphi^{\nicefrac{k}{m}}}-\widetilde{q_{\varphi^{\nicefrac{k}{m}}}}\right| \\
			&\leq \sqrt{E} + \max\left\|D\varphi\right\|_{\text{op}}\cdot\max\left\|\frac{\partial q_{\varphi^t}\left(\cdot,p\right)}{\partial Q}\right\|_{\text{op}}\left|Q^{\nicefrac{k}{m}}\left(q_{\varphi^{\nicefrac{k}{m}}}\left(Q,p\right),p\right)-Q\right| \\
			&\leq \sqrt{E} + \max\left\|D\varphi\right\|_{\text{op}}\cdot\max\left\|\left(\frac{\partial Q}{\partial q}\right)^{-1}\right\|_\text{op}\sqrt{E} \\
			&\leq \sqrt{E} + \frac{1+\max\left\|D\varphi-\Id\right\|_{\text{op}}}{1-\max\left\|D\varphi-\Id\right\|_{\text{op}}}\sqrt{E} = \frac{2}{1-\max\left\|D\varphi-\Id\right\|_{\text{op}}}\sqrt{E},
		\end{align*}
		so all together we get
		\begin{align*}
			\left|S\left(Q,p\right)-\widetilde{S}\left(Q,p\right)\right| &\leq \frac{2\max\left|\nabla H\right|^2}{1-\max\left\|D\varphi-\Id\right\|_\text{op}}\cdot\frac{1}{m} + \frac{2\max\left|\nabla H\right|}{1-\max\left\|D\varphi-\Id\right\|_{\text{op}}}\sqrt{E}.
		\end{align*}
		proving the Lemma.
	\end{proof}
	
	\subsection{Implementation steps}
	We now outline the implementation of the radial profile algorithm as follows, for $ \varphi\in\mathcal{R}_{T,T^\prime,R}^N $, assume we are given the following input
	\begin{equation*}
		N \geq 2, \quad \left\{h_j\in C_c^\infty\left(\RR_{\geq0}\right),c_j\in\RR^2,r_j > 0\right\}_{j=1}^N, \quad T,T^\prime,R>0, \quad m\in\NN,
	\end{equation*}
	that satisfies for all $ j=1,\dots,N $
	\begin{equation*}
		\supp\left(h_j\right)\subset \left[0,r_j\right],\quad \left|h_j^\prime\right|\leq T,\quad \left|h_j^{\prime\prime}\right|\leq T^\prime, \quad \max_j\left|c_j\right|+\sqrt{2r_j}\leq R,
	\end{equation*}
	and also
	\begin{equation*}
		\max_{j}\left\{\sqrt{2r_j}\left(\left|c_j\right|+\sqrt{2r_j}\right) T^\prime + T\right\} < \frac{1}{2}.
	\end{equation*}
	We use the following computation steps in order to obtain a barcode:
	\begin{enumerate}
		\item Set $ H_j:\RR^2\to\RR $ by
		\begin{equation*}
			H_j\left(x\right) = h_j\left(\frac{\left|x-c_j\right|^2}{2}\right),
		\end{equation*}
		therefore $ \max\left|\nabla H_j\right| \leq T\sqrt{2r_j} \leq TR $ and the Hamiltonian flow $ \varphi^t_j $ generated by $ H_j $ is given by \Cref{eq:radialflow} with bounds on the $ C^1 $ distance of $ \varphi_j = \varphi^1_j $ from the identity given by \Cref{eq:distfromIDbounds} so according to \Cref{prop:smallhamtogenfunc}, we know that each $ \varphi_j $ has a generating function $ S_j:\RR^2\to\RR $.
		
		$ \varphi_j $ is supported in $ \BB_{\sqrt{2r_j}}\left(c_j\right) $ so denoting $ r = \max_j r_j $ we have $ \varphi_j\in\mathcal{H}_{R,rT} $, thus $ \varphi = \varphi_N\circ\dots\circ\varphi_1 $ is indeed given by a finite composition of Hamiltonian diffeomorphisms from $ \mathcal{H}_{R,rT} $ each having a generating function $ S_j:\RR^2\to\RR $ without auxiliary variables, as the suggested algorithm from \Cref{thm:mainthm} states.
		
		\item For each $ j=1,\dots,N $ and for each $ k=1,\dots,m $, we go over all $ \left(Q,p\right)\in\frac{1}{m}\ZZ^2 \cap\BB_{r_j}\left(c_j\right) $ and numerically solve the inverse function problem
		\begin{equation*}
			Q_j^{\nicefrac{k}{m}}\left(\widetilde{q_{\varphi_j^{\nicefrac{k}{m}}}},p\right) = Q,
		\end{equation*}
		inside $ \BB_{r_jT}\left(c_j\right) $ with maximal error $ E $, where $ \varphi^t_j\left(q,p\right) = \left(Q_j^t,P_j^t\right) $. We denote the solution as $ \widetilde{q_{\varphi_j^{\nicefrac{k}{m}}}}\left(Q,p\right) $ and set
		\begin{equation*}
			\widetilde{S_j}\left(Q,p\right) = \frac{1}{m}\sum_{k=1}^m H_j\left(Q,P_j^{\nicefrac{k}{m}}\left(\widetilde{q_{\varphi_j^{\nicefrac{k}{m}}}}\left(Q,p\right),p\right)\right),
		\end{equation*}
		so by \Cref{Lemma:approxgf}, we get functions $ \widetilde{S_j}:\frac{1}{m}\ZZ^2\to\RR $ satisfying
		\begin{equation*}
			\max_{\left(Q,p\right)\in\frac{1}{m}\ZZ^2}\left|S_j\left(Q,p\right) - \widetilde{S_j}\left(Q,p\right)\right| \leq \underbrace{\frac{4RT^2}{1-T^\prime}}_{C_1\left(T,T^\prime,R\right)}\cdot\frac{1}{m}+\underbrace{\frac{2\sqrt{2R}T}{1-T^\prime}}_{C_2\left(T,T^\prime,R\right)}\sqrt{E}.
		\end{equation*}
		
		\item Following the initial steps of the algorithm suggested in \Cref{thm:mainthm} (i.e \Cref{sec:compositionformula} and \Cref{sec:obtainingGFQI}), we obtain an expression for a function $ \widetilde{S}:\frac{1}{m}\ZZ^{2N}\to\RR $ such that
		\begin{equation*}
			\left\|\widetilde{S} - S\vert_{\frac{1}{m}\ZZ^{2N}}\right\|_{C^0} \leq C_1\left(T,T^\prime,R\right)\cdot\frac{N}{m} + C_2\left(T,T^\prime,R\right)\cdot N\sqrt{E},
		\end{equation*}
		where $ S:\RR^{2N}\to\RR $ is a GFQI of $ \varphi $.
		
		\item We apply the barcode calculation algorithm on $ \widetilde{S} $ to get a barcode $ \mathcal{B} $ with an additional error term coming from $ \left\|\widetilde{S} - S\vert_{\frac{1}{m}\ZZ^{2N}}\right\|_{C^0} $ so according to \Cref{thm:mainthm} we finally have
		\begin{equation*}
			\dbot\left(\mathcal{B},\mathcal{B}\left(\varphi\right)\right) \leq C\left(R\right)\cdot\sqrt{2}RT\cdot\frac{N^2}{m} + C_1\left(T,T^\prime,R\right)\cdot\frac{N}{m} + C_2\left(T,T^\prime,R\right)\cdot N\sqrt{E},
		\end{equation*}
		where $ C\left(R\right) $ is given by \Cref{eq:CofR}.
	\end{enumerate}
	
	\subsection{Example} \label{sec:exmpale}
	In this section we offer a few examples in which our implementation computes the GF-barcode for a composition of two compactly supported Hamiltonian diffeomorphisms of $ \RR^2 $ generated by radial Hamiltonian functions with different centers and the results indeed give good approximation of the conjectured barcode.
	
	Let $ T > 0 $ and $ a\in\left(0,1\right) $ and consider the profile function $ h:\left[0,\frac{1}{2}\right]\to\RR $ with derivative $ h^\prime $ which is a smooth version of the function depicted in \Cref{img:profiles} and satisfies $ \max\left|h^\prime\right| = T $.

	\begin{figure}[h] 
		\centering
		\begin{minipage}{.4\linewidth}
			\includegraphics[width=0.7\linewidth,center]{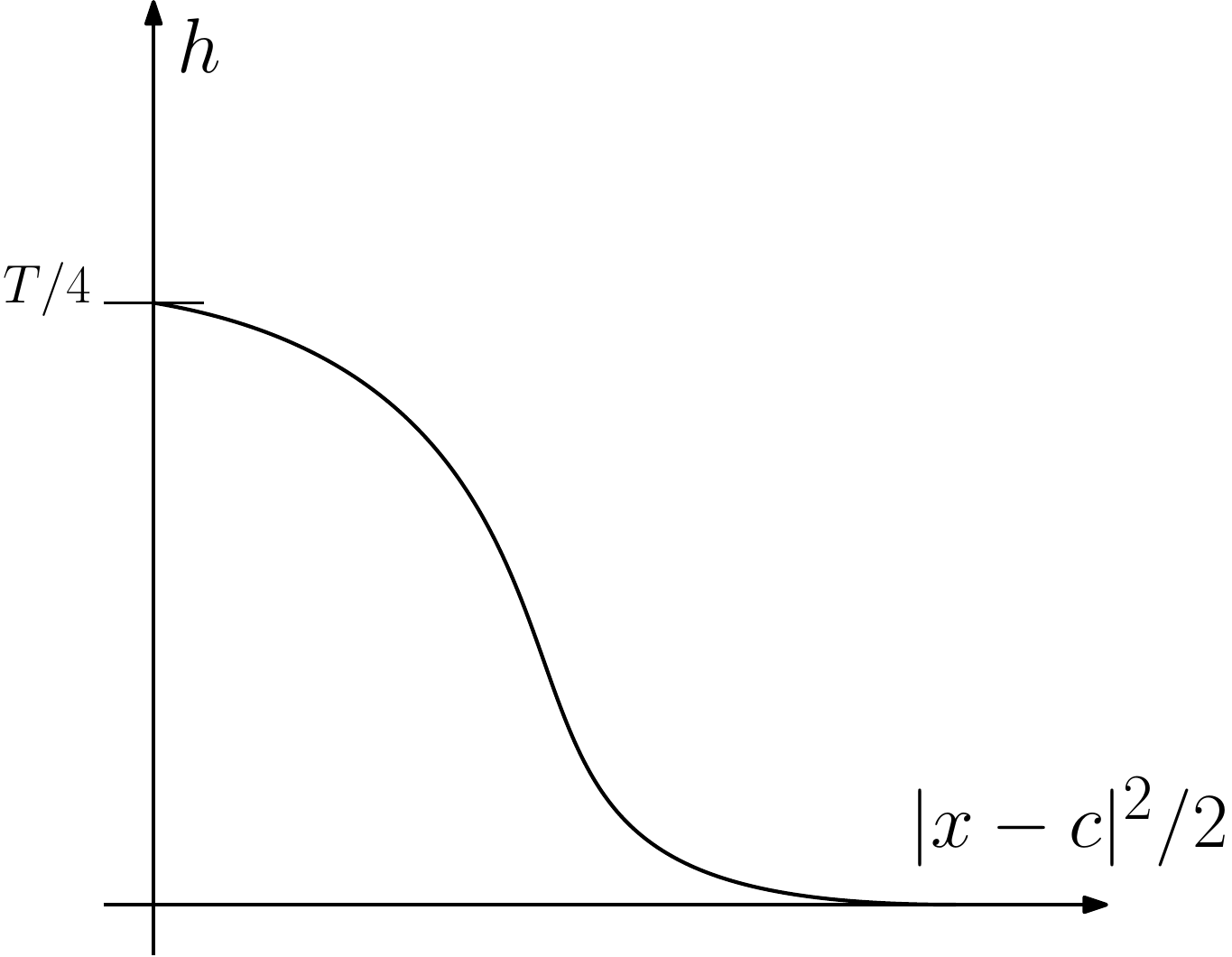}
			\label{img:hgraph}
		\end{minipage}
		\hspace{.05\linewidth}
		\begin{minipage}{.4\linewidth}
			\includegraphics[width=0.7\linewidth,center]{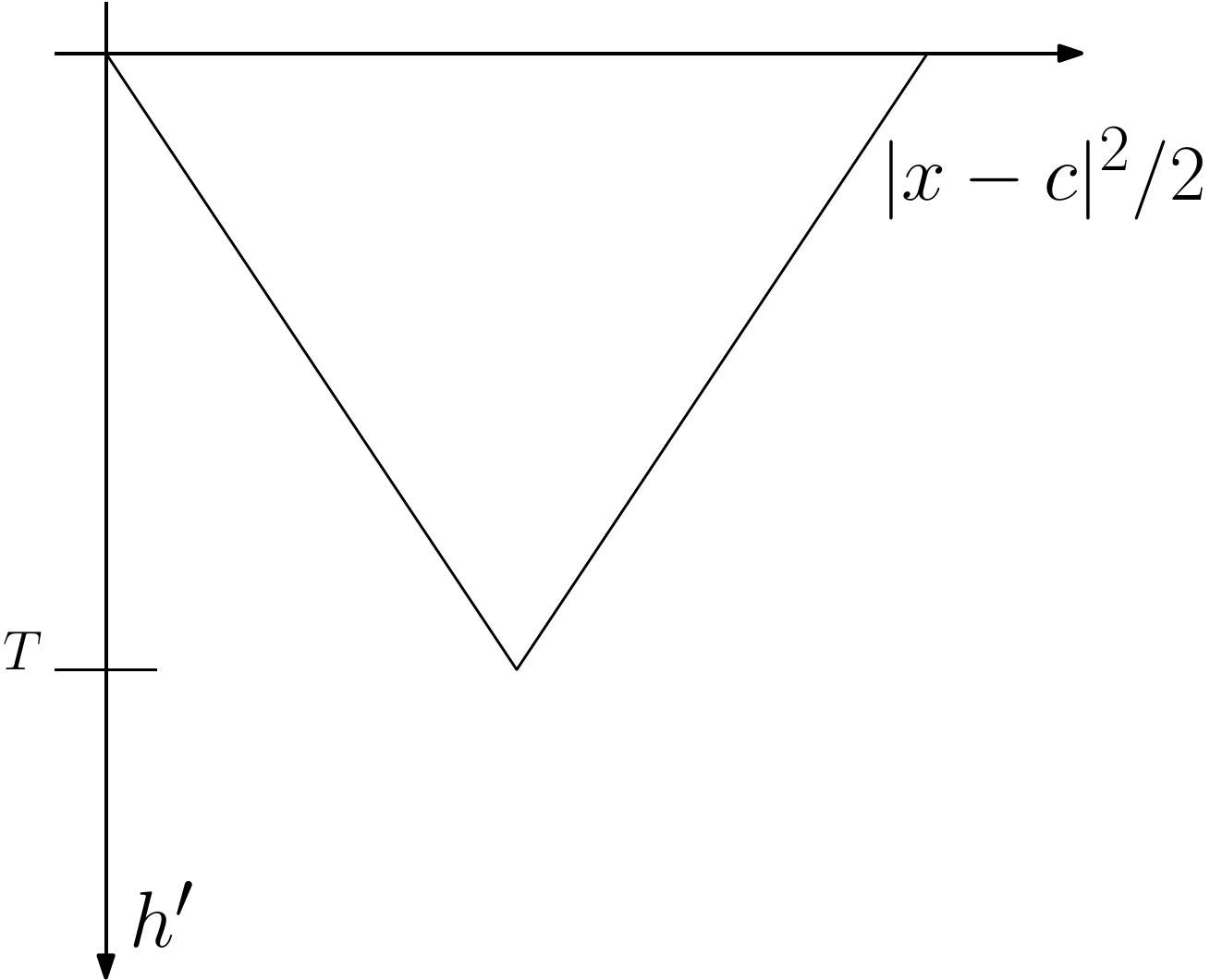}
			\label{img:hprimegraph}
		\end{minipage}
	\caption{The profile function $ h $ and its derivative $ h^\prime $}
	\label{img:profiles}
	\end{figure}

	We denote two center points by $ c_1 = \left(-a,0\right),c_2 = \left(a,0\right)\in\RR^2 $ and two radial Hamiltonian functions $ H_1\left(x\right) = h\left(\nicefrac{\left|x-c_1\right|^2}{2}\right), H_2\left(x\right) = h\left(\nicefrac{\left|x-c_2\right|^2}{2}\right) $ that generate flows $ \varphi^t,\psi^t\in\Ham_c\left(\RR^2,\omega_0\right) $ compactly supported inside $ \BB_1\left(c_1\right) $ and $ \BB_1\left(c_2\right) $ respectively. Recall that in this case $ \left\|\varphi^1-\Id\right\|_{C^0} = \left\|\psi^1-\Id\right\|_{C^0} = T $ and we call $ T $ the \textit{scaling factor} of this example. In order to ensure these diffeomorphisms poses simple generating functions and to enable realistic running time and memory requirements we take $ T\ll1 $.
	
	The diffeomorphism $ \varphi^1=\varphi $ has a trivial fixed point at $ c_1 $ and $ \psi^1=\psi $ has a trivial fixed point at $ c_2 $, we denote them as $ \left[\max H_1\right]$ and $\left[\max H_2\right] $ respectively. Since we treat our domain as the compactification $ \RR^2\cup\left\{\infty\right\} \cong \SS^2 $, the diffeomorphisms $\varphi$ and $\psi$ also have a common trivial fixed point at $ \infty $ denoted by $ \left[\min H\right] $. Furthermore, due to the symmetry of this setup one may also estimate additional fixed points in the overlap of both supports, we think of these fixed points as periodic orbits of the concatenation $ \psi^t\sharp\varphi^t $ and number them $ \left[1\right],\dots,\left[P\right] $ by their position along the horizontal axis as shown in \Cref{img:dynamics}. Note that the total number of periodic orbits $ P\left(a,T\right) $ depends on the distance $ a $ and on the scaling factor $ T $.
	
	We estimate the actions $ \mathcal{A}\left(\left[1\right]\right),\dots,\mathcal{A}\left(\left[P\right]\right) $ of each periodic orbit and deduce their Conley-Zehnder indices (see \cite[Section 8.1]{Pol19}), then we compile this information into a conjectured barcode for different values of $ a $, resulting in $ 4 $ possible cases, described in \Cref{img:barcodes}. Note that bars correspond to homology classes represented by chains of critical points of the generating function, considered as the fixed points denoted earlier.
	\pagebreak
	
	\begin{figure}[h!]
		\centering
		\includegraphics[width=0.8\linewidth]{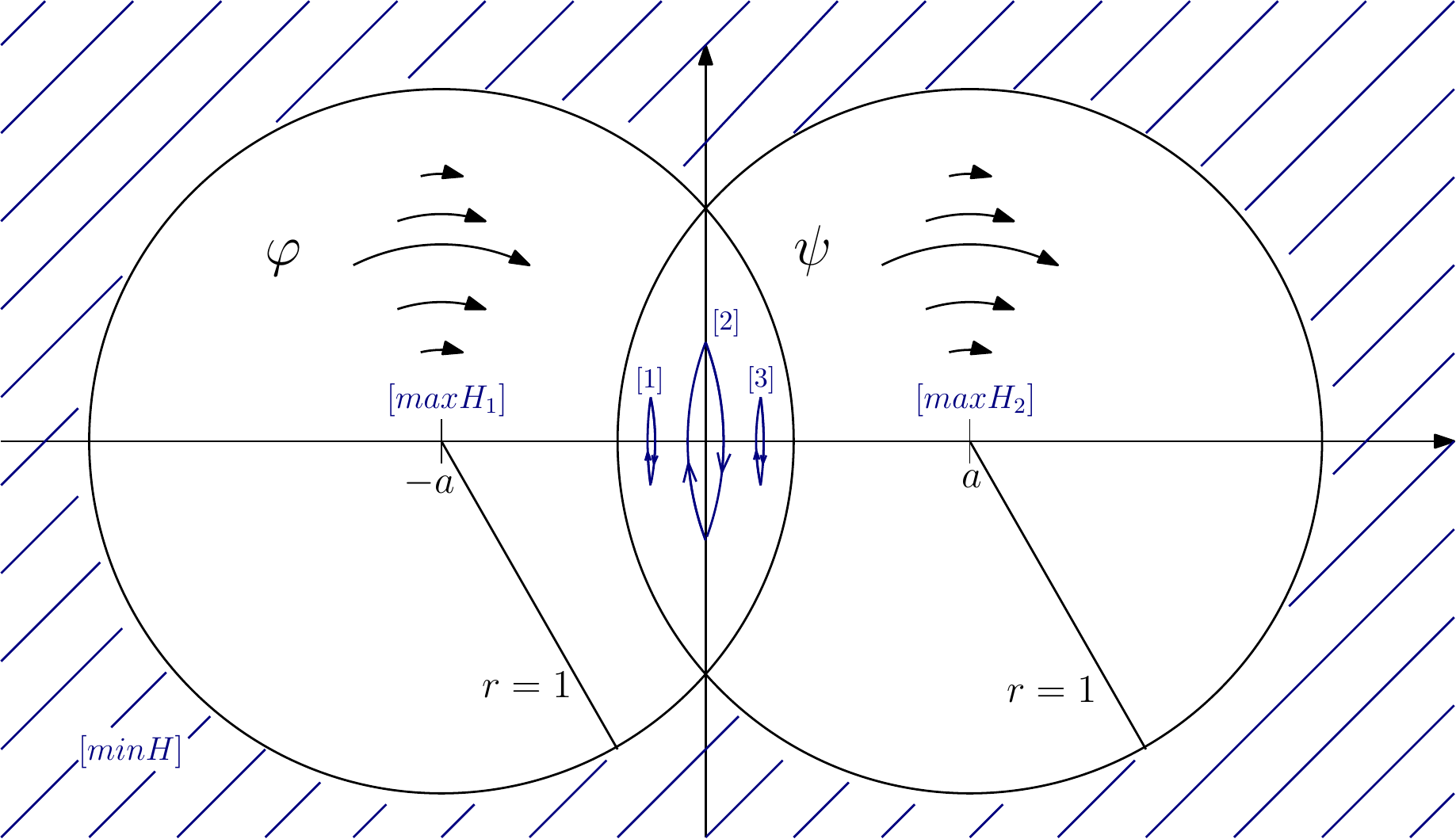}
		\caption{The dynamics of $ \psi^t\sharp\varphi^t$ including the {\color{blue} periodic orbits} for case $ \textrm{II} $. }
		\label{img:dynamics}
	\end{figure}

	\begin{figure}[h!]
	\centering
	\begin{minipage}{.45\linewidth}
		\includegraphics[width=\linewidth]{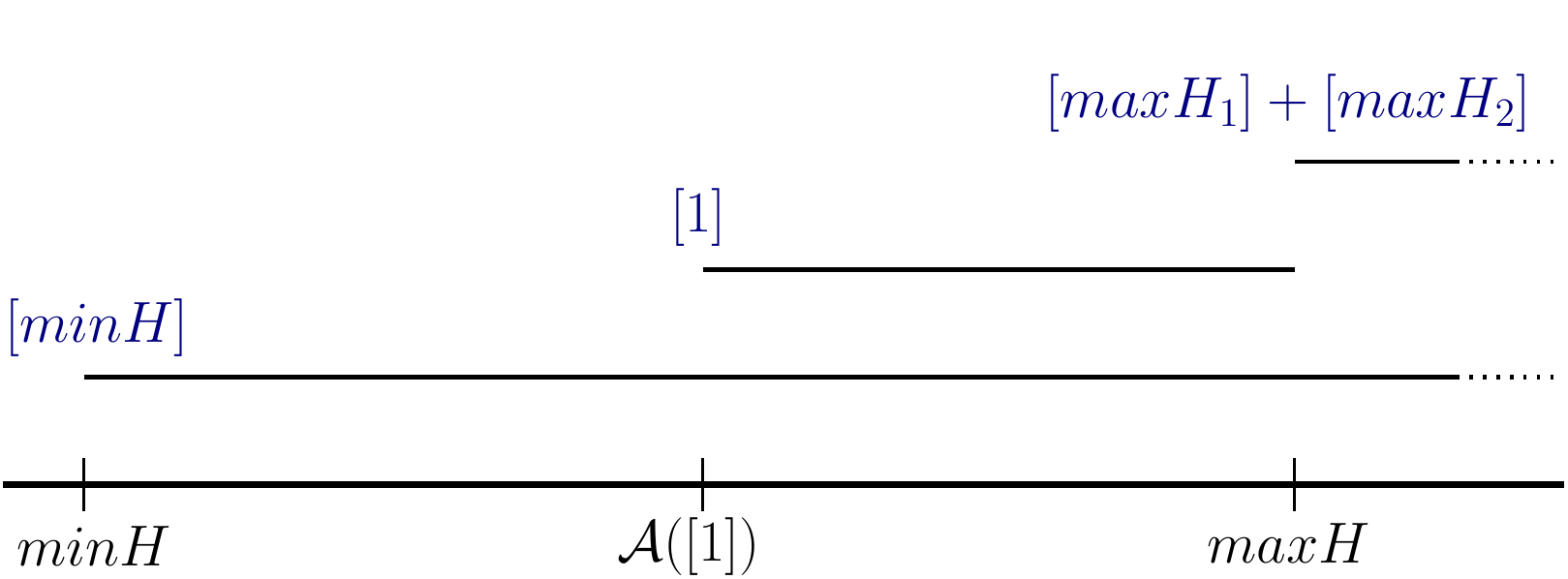}
		\caption*{Case \textrm{I}: 1 non-trivial orbit }
	\end{minipage}
	\hspace{.05\linewidth}
	\begin{minipage}{.45\linewidth}
		\includegraphics[width=\linewidth]{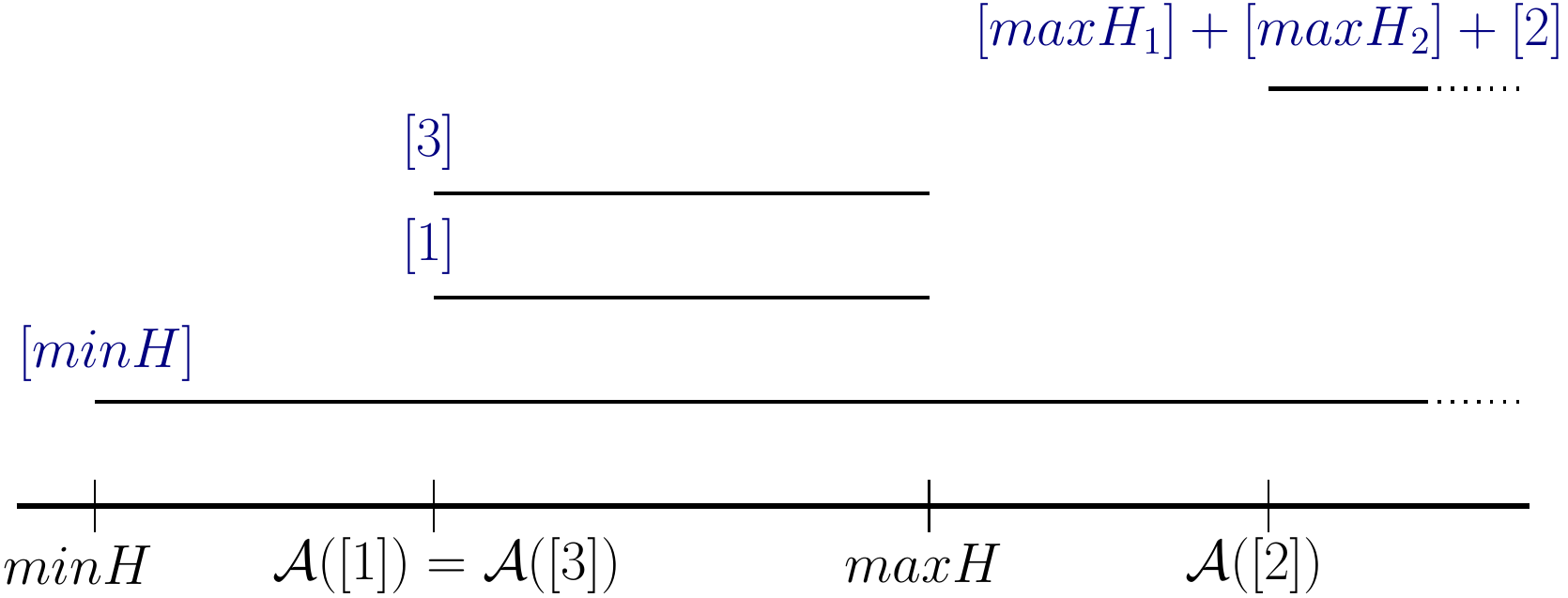}
		\caption*{Case \textrm{II}: 3 non-trivial orbits }
	\end{minipage}
	\begin{minipage}{.45\linewidth}
	\includegraphics[width=\linewidth]{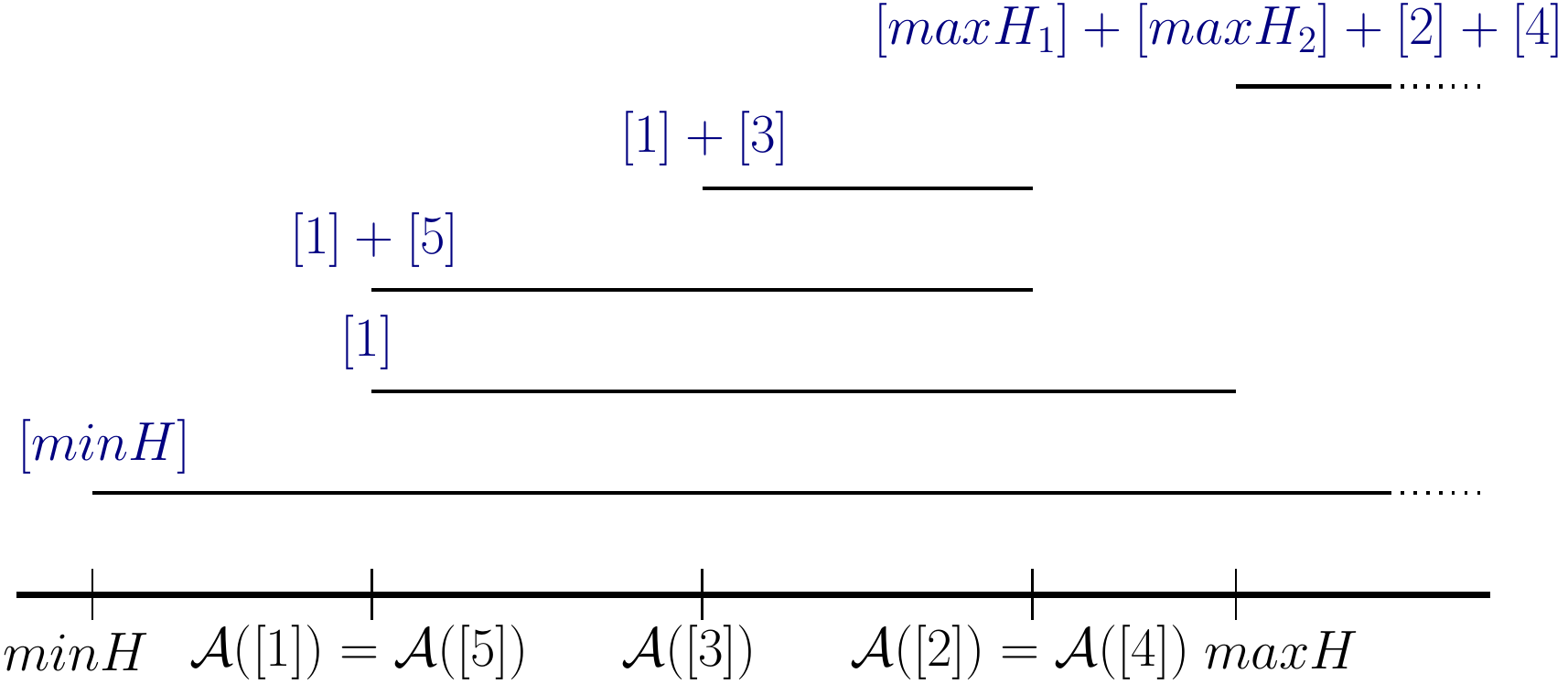}
	\caption*{Case \textrm{III}: 5 non-trivial orbits, type 1 }
	\end{minipage}
	\hspace{.05\linewidth}
	\begin{minipage}{.45\linewidth}
	\includegraphics[width=\linewidth]{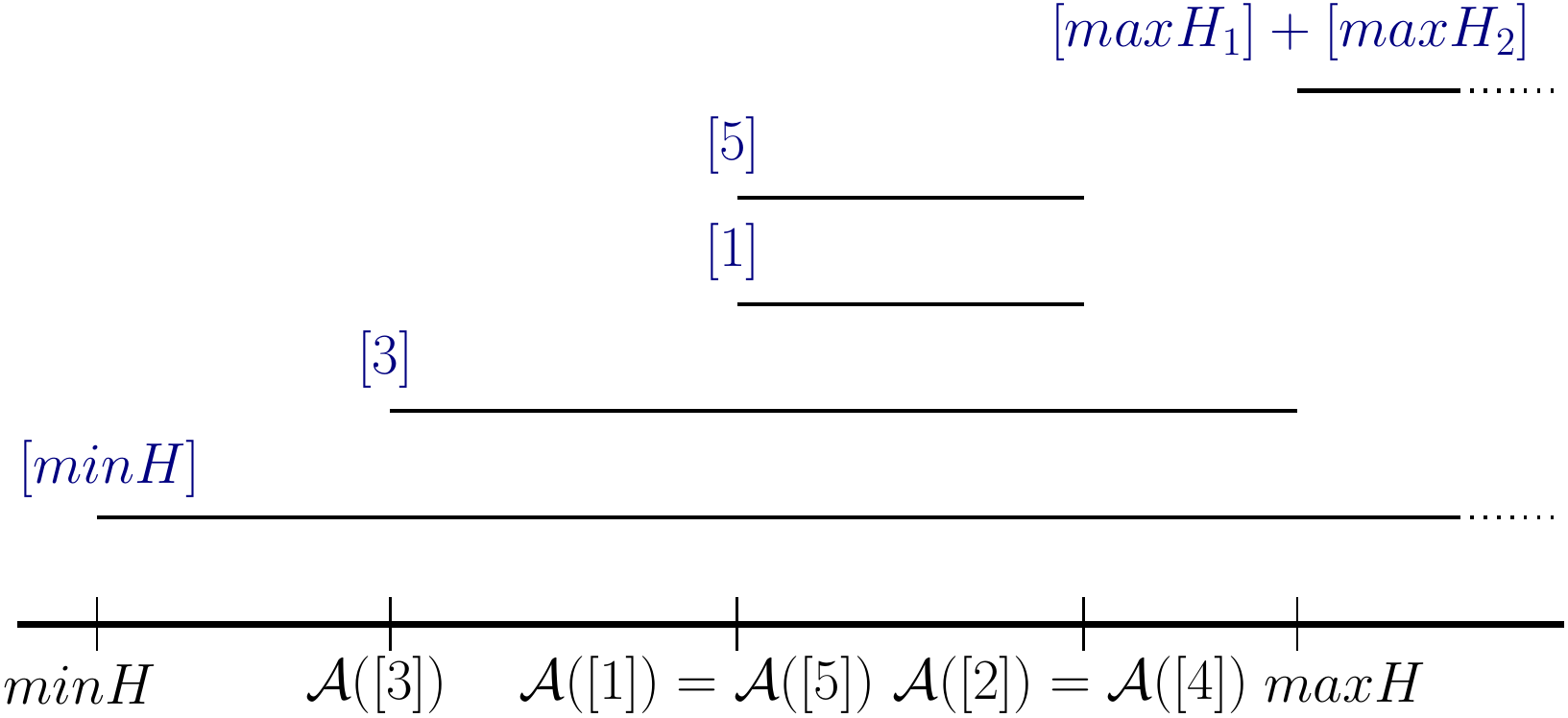}
	\caption*{Case \textrm{IV}: 5 non-trivial orbits, type 2 }
	\end{minipage}
	\caption{Conjectured barcodes $ \mathcal{B}\left(\psi\circ\varphi\right) $ for different values of $a$ including {\color{blue} generators} of each bar}
	\label{img:barcodes}
	\end{figure}

	Finally, we compute these examples using the suggested implementation of our algorithm with different mesh parameters $ m $ and scaling factors $ T $ yielding barcodes $ \mathcal{B} $ whose bottleneck distances to the conjectured barcodes are then measured relative to the scaling factor $ T $. The results are given in \Cref{table:1} for a specific choice of scaling factor ($ T = 2\pi\times 10^{-4} $) and mesh parameter ($ m = 64 $) but we mention that these results are consistent through other viable choices of parameters $ m,T $ as well.

	\begin{table}[h!]
		\centering
		{\makegapedcells\begin{tabular}{||c | c | c | c ||} 
				\hline
				$ a $ & Case of $ \mathcal{B}\left(\psi\circ\varphi\right) $ & Estimated $ \nicefrac{\text{longest finite bar}}{T} $ & $ \nicefrac{\dbot\left(\mathcal{B}\left(\psi\circ\varphi\right),\mathcal{B}\right)}{T} $ \\ [0.5ex] 
				\hline\hline
				$ 0.70 $ & 	Case \textrm{II} & $ \sim1.57\times 10^{-2} $ & $ \sim2.03\times10^{-6} $ \\
				\hline
				$ 0.72 $ & Case \textrm{III} & $ \sim2.46\times 10^{-2} $ & $ \sim3.47\times10^{-6} $ \\
				\hline
				$ 0.73 $ & Case \textrm{IV} & $ \sim3.18\times 10^{-2} $ & $ \sim1.60\times10^{-5} $ \\
				\hline
				$ 0.75 $ & Case \textrm{I} & $ \sim5.85\times 10^{-2} $ & $ \sim1.23\times10^{-8} $ \\
				\hline
		\end{tabular}}
		\caption{Different cases of conjectured barcodes along with the estimated length of their longest finite bar, and bottleneck distance between them and the approximated barcodes, relative to the scaling factor}
		\label{table:1}
	\end{table}
	
	Hence the approximated barcodes are indeed quite accurate and the error in bottleneck distance is actually much smaller than the one given by our theoretical bound in these examples which in all cases was $ \sim 3.4\times 10^{-1} $ after normalization by the scaling factor $ T $.

	\begin{remark}
		It would be interesting to implement practically our algorithm for larger values of $ N $ and $ T $.
	\end{remark}

	
	\pagebreak
	\printbibliography

\end{document}